\documentclass[11pt,english]{smfart}

\usepackage[T2A,T1]{fontenc}
\usepackage[english,francais]{babel}
\usepackage[utf8]{inputenc}

\usepackage{latexsym,amscd,color}
\usepackage{amsmath,amsfonts,amssymb,mathrsfs,mathtools}
\usepackage{bm}
\usepackage{enumitem,euscript}
\usepackage{amssymb,url,xspace,smfthm,hyperref,CJK}
\usepackage{comment}
\usepackage{tikz}
\usepackage{tikz-cd}
\usepackage{manfnt}

\usepackage[normalem]{ulem}
\usepackage{etoolbox}

\input xy
\xyoption{all}

\numberwithin{equation}{section}

\newcommand{\AMS}{$\mathcal{A}$\mathrm{Ker}n-.1667em\lower.5ex\hbox
        {$\mathcal{M}$}\mathrm{Ker}n-.125em$\mathcal{S}$}

\def\sbullet{{\scriptscriptstyle\bullet}}
\newcommand{\emptyinnprod}{\langle\kern.3em,\kern.13em\rangle}
\DeclareMathOperator{\Spec}{Spec}
\DeclarePairedDelimiter{\norm}{\lVert}{\rVert}
\newcommand{\ndot}{\raisebox{.4ex}{.}}

\def\colorsout#1{\bgroup\markoverwith{\textcolor{#1}{\rule[0.5ex]{2pt}{0.7pt}}}\ULon} 
\def\coloruline#1{\bgroup\markoverwith{\textcolor{#1}{\rule[-0.5ex]{2pt}{0.7pt}}}\ULon} 

\def\ifquery{\if00}
\ifquery
\def\lquery#1#2{\setlength\marginparwidth{#1}
\marginpar{\raggedright\fontsize{7.81}{9} 
\selectfont\upshape\hrule\smallskip 
#2\par\smallskip\hrule}}
\def\query#1{\lquery{65pt}{#1}}

\else
\def\query#1{}
\fi

\makeatletter
\def\widebreve{\mathpalette\wide@breve}
\def\wide@breve#1#2{\sbox\z@{$#1#2$}%
     \mathop{\vbox{\m@th\ialign{##\crcr
\kern0.08em\brevefill#1{0.8\wd\z@}\crcr\noalign{\nointerlineskip}%
                    $\hss#1#2\hss$\crcr}}}\limits}
\def\brevefill#1#2{$\m@th\sbox\tw@{$#1($}%
  \hss\resizebox{#2}{\wd\tw@}{\rotatebox[origin=c]{90}{\upshape(}}\hss$}
\makeatletter

\definecolor{Huayi}{rgb}{0.0, 0.0, 1.0}

\definecolor{Hide}{rgb}{1.0, 0.0, 0.0}

\definecolor{mred}{rgb}{0.83, 0.0, 0.0}

\title[Arakelov theory]{Arakelov theory on arithmetic surfaces over a trivially valued field}
\date{\today}
\author{Huayi Chen}
\address{Universit\'e de Paris, Sorbonne Universit\'e, CNRS, Institut de Math\'ematiques de Jussieu-Paris Rive Gauche, F-75013 Paris, France}
\email{huayi.chen@imj-prg.fr}
\urladdr{webusers.imj-prg.fr/~huayi.chen}
\author{Atsushi Moriwaki}
\address{Department of Mathematics, Faculty of Science, Kyoto University, Kyoto, 606-8502, Japan}
\email{moriwaki@math.kyoto-u.ac.jp}

\begin{document}
\def\smfbyname{}

\begin{abstract}In this article, we consider an analogue of Arakelov theory of arithmetic surfaces over a trivially valued field. In particular, we establish an arithmetic Hilbert-Samuel theorem and studies the effectivity up to $\mathbb R$-linear equivalence of pseudoeffective metrised $\mathbb R$-divisors.
\end{abstract}


\maketitle

\tableofcontents

\section{Introduction}

In Arakelov geometry, one considers an algebraic variety over the spectrum of a number field and studies various constructions and invariants on the variety such as metrised line bundles, intersection product, height functions etc. Although these notions have some similarities to those in classic algebraic geometry, their construction is often more sophisticated and get involved analytic tools. 

Recently, an approach of $\mathbb R$-filtration has been proposed to study several invariants in Arakelov geometry, which allows to get around analytic technics in the study of some arithmetic invariants, see for example \cite{MR2768967,MR2722508}. Let us recall briefly this approach in the setting of Euclidean lattices for simplicity. Let $\overline E=(E,\norm{\ndot})$ be a Euclidean lattice, namely a free $\mathbb Z$-module of finite type $E$ equipped with a Euclidean norm $\norm{\ndot}$ on $E_{\mathbb R}=E\otimes_{\mathbb Z}\mathbb R$. We construct a family of vector subspaces of $E_{\mathbb Q}=E\otimes_{\mathbb Z}\mathbb Q$ as follows. For any $t\in\mathbb R$, let $\mathcal F^t(\overline E)$ be the $\mathbb Q$-vector subspace of $E_{\mathbb Q}$ generated by the lattice vectors $s$ such that $\norm{s}\leqslant \mathrm{e}^{-t}$. This construction is closely related to the successive minima of Minkowski. In fact, the $i$-th minimum of the lattice $\overline E$ is equal to 
\[\exp\Big(-\sup\{t\in\mathbb R\,|\,\operatorname{rk}_{\mathbb Q}(\mathcal F^t(\overline E))\geqslant i\}\Big).\]
The family $(\mathcal F^t(\overline E))_{t\in\mathbb R}$ is therefore called the \emph{$\mathbb R$-filtration by minima} of the Euclidean lattice $\overline E$.

Classically in Diophantine geometry, one focuses on the lattice points of small length, which are analogous to global sections of a vector bundle over a smooth projective curve. However, such points are in general not stable by addition. This phenomenon brings difficulties to the study of Arakelov geometry over a number field and prevents the direct transplantation of algebraic geometry methods in the arithmetic setting. In the $\mathbb R$-filtration approach, the arithmetic invariants are encoded in a family of vector spaces, which allows to apply directly algebraic geometry methods to study some problems in Arakelov geometry. 

If we equipped $\mathbb Q$ with the trivial absolute value $|\ndot|_0$ such that $|a|_0=1$ if $a$ belongs to $\mathbb Q\setminus\{0\}$ and $|0|_0=0$, then the above $\mathbb R$-filtration by minima can be considered as an ultrametric norm $\norm{\ndot}_0$ on the $\mathbb Q$-vector space $E_{\mathbb Q}$ such that \[\norm{s}_{0}=\exp(-\sup\{t\in\mathbb R\,|\,s\in\mathcal F^t(\overline E)\}).\]
Interestingly, finite-dimensional ultrametrically normed vector spaces over a trivially valued field are also similar to vector bundles over a smooth projective curve. This method is especially successful in the study of the arithmetic volume function. Moreover, $\mathbb R$-filtrations, or equivalently, ultrametric norms with respect to the trivial absolute value, are also closely related to the geometric invariant theory of the special linear group, as shown in  \cite[\S6]{MR2496458}.

All this works suggest that there would be an Arakelov theory over a trivially valued field. From the philosophical point of view, the $\mathbb R$-filtration approach can be considered as a correspondance from the arithmetic geometry over a number field to that over a trivially valued field, which preserves some interesting arithmetic invariants. The purpose of this article is to build up such a theory for curves over a trivially valued field (which are actually analogous to arithmetic surfaces). Considering the simplicity of the trivial absolute value, one might expect  such a theory to be simple. On the contrary, the arithmetic intersection theory for adelic divisors in this setting is already highly non-trivial, which has interesting interactions with the convex analysis on infinite trees. 

Let $k$ be a field equipped with the trivial absolute value and $X$ be a regular irreducible projective curve over $\operatorname{Spec} k$. We denote by $X^{\mathrm{an}}$ the Berkovich analytic space associated with $X$, which identifies with a tree of length $1$ whose leaves correspond to closed points of $X$ (see \cite[\S3.5]{MR1070709}). 
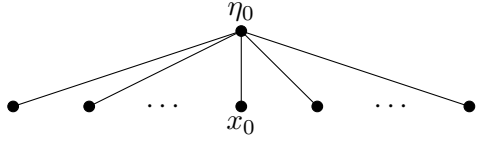
\vspace{3mm}
\begin{center}
\begin{tikzpicture}
\filldraw(0,1) circle (2pt) node[align=center, above]{$\eta_0$};
\filldraw(-3,0) circle (2pt) ;
\draw (-1,0) node{$\cdots$};
\filldraw(-2,0) circle (2pt) ;
\filldraw(-0,0) circle (2pt) node[align=center, below]{$x_0$} ;
\filldraw(1,0) circle (2pt) ;
\draw (2,0) node{$\cdots$};
\filldraw(3,0) circle (2pt) ;
\draw (0,1) -- (0,0);
\draw (0,1) -- (-3,0);
\draw (0,1) -- (1,0);
\draw (0,1) -- (-2,0);
\draw (0,1) -- (3,0);
\end{tikzpicture}
\end{center}
\vspace{3mm}
For each closed point $x$ of $X$, we denote by $[\eta_0,x_0]$ the edge connecting the root and the leaf corresponding to $x$. This edge is parametrised by the interval $[0,+\infty]$ and we denote by $t(\ndot):[\eta_0,x_0]\rightarrow[0,+\infty]$ the parametrisation map.  
Recall that an $\mathbb R$-divisor $D$ on $X$ can be viewed as an element in the free real vector space over the set $X^{(1)}$ of all closed points of $X$. We denote by $\operatorname{ord}_x(D)$ the coefficient of $x\in X^{(1)}$ in the writing of $D$ into a linear combination of elements of $X^{(1)}$. We call Green function of $D$ any continuous map $g:X^{\mathrm{an}}\rightarrow [-\infty,+\infty]$ such that there exists a continuous function $\varphi_g:X^{\mathrm{an}}\rightarrow\mathbb R$ which satisfies the following condition
\[\forall\,x\in X^{(1)},\;\forall\,\xi\in[\eta_0,x_0\mathclose{[},\quad 
\varphi_g(\xi)=g(\xi)-\operatorname{ord}_x(D)t(\xi).\] 
The couple $\overline D=(D,g)$ is called a metrised $\mathbb R$-divisor on $X$. Note that the set  $\widehat{\operatorname{Div}}_{\mathbb R}(X)$ of metrised $\mathbb R$-divisors on $X$ forms actually a vector space over $\mathbb R$.

Let $D$ be an $\mathbb R$-divisor on $X$. We denote by $H^0(D)$ the subset of the field  $\operatorname{Rat}(X)$ of rational functions on $X$ consisting of the zero rational function and all rational functions $s$ such that $D+(s)$ is effective as an $\mathbb R$-divisor, where $(s)$ denotes the principal divisor associated with $s$, whose coefficient of $x$ is the order of $s$ at $x$. The set $H^0(D)$ is actually a $k$-vector subspace of $\operatorname{Rat}(X)$. Moreover, the Green function $g$ determines an ultrametric norm $\norm{\ndot}_{g}$ on the vector space $H^0(D)$ such that 
\[\norm{s}_g=\exp\Big(-\inf_{\xi\in X^{\mathrm{an}}}(g+g_{(s)})(\xi)\Big).\]

Let $\overline D_1=(D_1,g_1)$ and $\overline D_2=(D_2,g_2)$ be adelic $\mathbb R$-divisors on $X$ such that $\varphi_{g_1}$ and $\varphi_{g_2}$ are absolutely continuous with square integrable densities, we define a pairing of $\overline D_1$ and $\overline D_2$ as  (see \S\ref{Subsec: pairing} for details)
\begin{equation}\label{Equ: intersection pairing}\begin{split}(\overline D_1\cdot\overline D_2):=g_1&(\eta_0)\deg(D_1)+g_2(\eta_0)\deg(D_1)\\
&-\sum_{x\in X^{(1)}}[k(x):k]\int_{\eta_0}^{x_0}g_1'(\xi)g_2'(\xi)\,\mathrm{d}t(\xi)
\end{split}\end{equation} 
Note that such pairing is similar to the local admissible pairing introduced in \cite[\S2]{MR1207481} or, more closely, similar to the Arakelov intersection theory on arithmetic surfaces with $L_1^2$-Green functions (see \cite[\S5]{MR1681810}). This construction is also naturally related to harmonic analysis on metrised graphes introduced in \cite{MR2310616} (see also \cite{MR1195689} for the capacity pairing in this setting), although the point $\eta_0$ is linked to an infinitely many vertices. A more conceptual way to understand the above intersection pairing (under diverse extra  conditions on Green functions) is to introduce a base change to a field extension $k'$ of $k$, which is equipped with a non-trivial absolute value extending the trivial absolute value on $k$.  It is then possible to define a Monge-Amp\`{e}re measure over $X_{k'}^{\mathrm{an}}$ for the pull-back of $g_1$, either by the theory of $\delta$-forms \cite{MR3602767,MR3975640}, or by the non-Archimedean Bedford-Taylor theory developed in \cite{2Antoine}, or more directly, by the method of Chambert-Loir measure \cite{MR2244803,MR2543659}. It turns out that the push-forward of this measure on $X^{\mathrm{an}}$ does not depend on the choice of valued extension $k'/k$ (see \cite[Lemma 7.2]{BE18}). We can then interpret the intersection pairing as the height of $D_2$ with respect to $(D_1,g_1)$ plus the integral of $g_2$ with respect to the push-forward of this Monge-Amp\`{e}re measure.

One
contribution of the article is to describe the asymptotic behaviour of the system of ultrametrically normed vector spaces $(H^0(nD),\norm{\ndot}_{ng})$ in terms of the intersection pairing, under the condition that the Green function $g$ is plurisubharmonic (see Definition \ref{Def: psh}). More precisely, we obtain an analogue of the arithmetic Hilbert-Samuel theorem as follows (see \S\ref{Sec:Hilbert-Samuel} \emph{infra}).

\begin{theo}
Let $\overline D=(D,g)$ be an adelic $\mathbb R$-divisor on $X$. We assume that $\deg(D)>0$ and $g$ is plurisubharmonic. Then one has
\[\lim_{n\rightarrow+\infty}\frac{-\ln\norm{s_1\wedge\cdots\wedge s_{r_n}}_{ng,\det}}{n^2/2}=(\overline D\cdot\overline D),\]
where $(s_i)_{i=1}^{r_n}$ is a basis of $H^0(nD)$ over $k$ (with $r_n$ being the dimension of the $k$-vector space $H^0(nD)$), $\norm{\ndot}_{ng,\det}$ denotes the determinant norm associated with $\norm{\ndot}_{ng}$, and $(\overline D\cdot\overline D)$ is the self-intersection number of $\overline D$.
\end{theo}

Diverse notions of positivity, such as bigness and pseudo-effectivity, are discussed in the article. We also study the effectivity up to $\mathbb R$-linear equivalence of pseudo-effective metrised $\mathbb R$-divisors. The analogue of this problem in algebraic geometry is very deep. It is the core of the non-vanishing conjecture, which has applications in the existence of log minimal models \cite{MR2831514}. It is also related to Keel's conjecture (see \cite[Question~0.9]{MR2007391} and \cite[Question~0.3]{MR3479310}) for the ampleness of divisors on a projective surface over a finite field. In the setting of an arithmetic curve associated with a number field,  this problem can actually be interpreted by Dirichlet's unit theorem in algebraic number theory. In the setting of higher dimensional arithmetic varieties, the above effectivity problem is very subtle. Both examples and obstructions were studied in the literature, see for example \cite{MR3049312,MR3436160} for more details.

In this article, we establish the following result.

\begin{theo}
Let $(D,g)$ be a metrised $\mathbb R$-divisor on $X$. For any $x\in X^{(1)}$, we let \[\mu_{\inf,x}(g):=\inf_{\xi\in\mathopen{]}\eta_0,x_0\mathclose{[}}\frac{g(\xi)}{t(\xi)}.\]
Let
\[\mu_{\inf}(g):=\sum_{x\in X^{(1)}}\mu_{\inf,x}(g)[k(x):k].\]
Then the following assertions hold.
\begin{enumerate}[label=\rm(\arabic*)]
\item $(D,g)$ is pseudo-effective if and only if $\mu_{\inf}(g)\geqslant 0$.
\item $(D,g)$ is $\mathbb R$-linearly equivalent to an effective metrised $\mathbb R$-divisor if and only if $\mu_{\inf,x}(g)\geqslant 0$ for all but finitely many $x\in X^{(1)}$ and if one of the following conditions holds:
\begin{enumerate}[label=\rm(\alph*)]
\item $\mu_{\inf}(g)>0$,
\item $\sum_{x\in X^{(1)}}\mu_{\inf,x}(g)x$ is a principal $\mathbb R$-divisor.
\end{enumerate}
\end{enumerate}
\end{theo}

The article is organised as follows. In the second section, we discuss several properties of convex functions on a half line. In the third section, we study Green functions on an infinite tree. The fourth section is devoted to a presentation of graded linear series on a regular projective curve. These sections prepares various tools to develop in the fifth section an Arakelov theory of metrised $\mathbb R$-divisors on a regular projective curve over a trivially valued field. In the sixth section, we discuss diverse notion of global and local positivity of metrised $\mathbb R$-divisors. Finally, in the seventh section, we prove the Hilbert-Samuel theorem for arithmetic surfaces in the setting of Arakelov geometry over a trivially valued field.
\vspace{3mm}

\noindent {\bf Acknowledgement:} We are grateful to Walter Gubler and Klaus K\"{u}nnemann for discussions.

\section{Asymptotically linear functions}

\subsection{Asymptotic linear functions on $\mathbb R_{>0}$}\label{Sec: function of Green type}

We say that a continuous function $g:\mathbb R_{>0}\rightarrow\mathbb R$ is \emph{asymptotically linear} (at the infinity) if there exists a real number $\mu(g)$ such that the function \[\varphi_g:\mathbb R_{>0}\longrightarrow\mathbb R,\quad \varphi_g(t):=g(t)-\mu(g)t\] extends to a continuous function on $[0,+\infty]$.  The real number $\mu(g)$ satisfying this condition is unique. We call it the \emph{asymptotic slope of $g$}. The set of asymptotically linear continuous functions forms a real vector space with respect to the addition and the multiplication by a scalar. The map $\mu(\ndot)$ is a linear form on this vector space.

We denote by $L^2_1(\mathbb R_{>0})$ the vector space of continuous functions $\varphi$ on $\mathbb R_{>0}$ such that the derivative (in the sense of distribution) $\varphi'$ is represented by a square-integrable function on $\mathbb R_{>0}$. We say that an asymptotically linear continuous function $g$ on $\mathbb R_{>0}$ is \emph{pairable} if  
the function $\varphi_g$
belongs to $L_1^2(\mathbb R_{>0})$. 

\begin{rema}
The functional space $L_1^2$ is a natural object of the potential theory on Riemann surfaces. In the classic setting of Arakelov geometry, it has been used in  the intersection theory on arithmetic surfaces. We refer to \cite[\S3]{MR1681810} for more details. 
\end{rema}


\subsection{Convex function on $[0,+\infty]$}

Let $\varphi$ be a convex function on  $\mathbb R_{>0}$. Then $\varphi$ is continuous on $\mathbb R_{>0}$. Moreover, for any $t\in\mathbb R_{>0}$, the right derivative of $\varphi$ at $t$, given by 
\[\lim_{\varepsilon\downarrow 0} \frac{\varphi(t+\varepsilon) - \varphi(t)}{\varepsilon},\]
exists in $\mathbb R$. By abuse of notation, we denote by $\varphi'$ the right derivative function of $\varphi$ on $\mathbb R_{>0}$. It is a right continuous increasing function. We refer to \cite[Theorem~1.26]{SB} for more details. Moreover, for any $(a, b) \in \mathbb R_{>0}^2$, one has \begin{equation}\label{Equ: integral of derivativie}\varphi(b) - \varphi(a) = \int_{\mathopen{]}a,b\mathclose{[}} \varphi'(t)\, \mathrm{d}t.\end{equation} See \cite[Theorem~1.28]{SB} for a proof.  In particular, the function $\varphi'$  represents the derivative of $\varphi$ in the sens of distribution.

\begin{prop}\phantomsection\label{prop:conv:properties}
Let $\varphi$ be a convex function  on $\mathbb R_{>0}$ which is bounded.
\begin{enumerate}[label=\rm(\arabic*)]

\item\label{Item: varphi prime negative} One has $\varphi' \leqslant 0$ on $\mathbb R_{>0}$ and $\lim_{t \rightarrow+\infty} \varphi'(t) = 0$.
In particular, the function $\varphi$ is decreasing and extends to a continuous function on $[0,+\infty]$.

\item\label{Item: limit of left derivative} We extends $\varphi$ continuously on $[0,+\infty]$. The function  \[(t\in\mathbb R_{>0}) \longmapsto \frac{\varphi(t) - \varphi(0)}{t}\] is increasing. Moreover, one has \[\lim_{t\downarrow 0} \frac{\varphi(t) - \varphi(0)}{t}=\lim_{t\downarrow 0} \varphi'(t)\in \mathopen{[}{-\infty},{0}{]},\]
which is denoted by $\varphi'(0)$.
\end{enumerate}
\end{prop}

\begin{proof}

\ref{Item: varphi prime negative} We assume by contradiction that $\varphi'(a) > 0$ at certain $a\in\mathbb R_{>0}$. By \eqref{Equ: integral of derivativie}, for any $x\in\mathbb R_{>0}$ such that $x>a$, one has
\[
\varphi(x) - \varphi(a) = \int_{\mathopen{]}a,x\mathclose{[}} \varphi'(t)\,\mathrm{d}t \geqslant \int_{\mathopen{]}a,x\mathclose{[}} \varphi'(a) \,\mathrm{d}t = (x-a) \varphi'(a),
\]
so that $\lim_{x\to+\infty} \varphi(x) = +\infty$. This is a contradiction. Thus $\varphi'(t) \leqslant 0$ for all $t \in \mathbb R_{>0}$.
Therefore, one has \[\lim_{t\rightarrow+\infty}\varphi'(t)=\sup_{t \in \mathbb R_{>0}}  \varphi'(t) \leqslant 0.\] 
To show that the equality $\lim_{t \rightarrow+\infty} \varphi'(t) = 0$ holds, we assume by contradiction that there exists $\varepsilon>0$ such that $\varphi'(t)\leqslant -\varepsilon$ for any $t\in\mathbb R_{>1}$. Then, by (1), for any $x\in\mathbb R_{>1}$,
\[
\varphi(x) - \varphi(1) = \int_{1}^x \varphi'(t)\,\mathrm{d}t \leqslant\int_{\mathopen{]}1,x\mathclose{[}} (-\varepsilon)\,\mathrm{d}t = -\varepsilon (x-1).
\]
Therefore, $\lim_{x\to+\infty} \varphi(x) = -\infty$, which leads to a contradiction.

\medskip
\ref{Item: limit of left derivative} For $0 < a < b$, since
\[
\varphi(a) = \varphi((1-a/b) 0 + (a/b) b) \leq (1-a/b) \varphi(0) + (a/b)\varphi(b),
\]
one has
\[
\frac{\varphi(a) - \varphi(0)}{a} \leqslant \frac{\varphi(b) - \varphi(0)}{b}
\]
as required. Denote by $\varphi'(0)$ the limite $\lim_{s\downarrow 0}\varphi'(s)$. Note that the equality \eqref{Equ: integral of derivativie} actually holds for $(a,b)\in[0,+\infty]^2$ (by the continuity of $\varphi$ and the monotone convergence theorem). Therefore 
\[\varphi'(0)\leqslant\frac{\varphi(t)-\varphi(0)}{t}=\frac{1}{t}\int_{\mathopen{]}{0},{t}\mathclose{[}}\varphi'(s)\,\mathrm{d}s\leqslant \varphi'(t).\]
By passing to limit when $t\downarrow 0$, we obtain the result.
\end{proof}

\begin{prop}\label{Pro: computation of integral varphi psi}
Let $\varphi$ and $\psi$ be continuous functions on $[0,+\infty]$ which are convex on $\mathbb R_{>0}$. One has 
\begin{equation}\label{Equ: integral of varphi and dpsi}
\int_{\mathopen{]}0,{+\infty}\mathclose{]}} \varphi \,\mathrm{d}\psi' = -\int_{\mathbb R_{>0}} \psi'(t) \varphi'(t) \,\mathrm{d}t - \varphi(0) \psi'(0)\in \mathopen{[}{-\infty},{+\infty}\mathclose{[}.
\end{equation}
In particular, if $\varphi(0)=\psi(0)=0$, then one has
\begin{equation}\label{Equ: symetry}\int_{\mathopen{]}0,{+\infty}\mathclose{]}}\varphi\,\mathrm{d}\psi'=\int_{\mathopen{]}0,{+\infty}\mathclose{]}}\psi\,\mathrm{d}\varphi'.\end{equation} 
\end{prop}
\begin{proof}
By \eqref{Equ: integral of derivativie}, one has
\[\int_{\mathopen{]}0,{+\infty}\mathclose{]}}\varphi\,\mathrm{d}\psi'=\int_{\mathopen{]}0,{+\infty}\mathclose{]}}\int_{\mathopen{]}0,x\mathclose{[}}\varphi'(t)\,\mathrm{d}t\,\mathrm{d}\psi'(x)+\varphi(0)\int_{\mathopen{]}0,{+\infty}\mathclose{]}}\,\mathrm{d}\psi'.\]
By Fubini's theorem, the double integral is equal to
\[\int_{\mathbb R_{>0}}\varphi'(t)\int_{\mathopen{]}t,{+\infty}\mathclose{]}}\,\mathrm{d}\psi'\,\mathrm{d}t=-\int_{\mathbb R_{>0}}\varphi'(t)\psi'(t)\,\mathrm{d}t.\] Therefore, the equality \eqref{Equ: integral of varphi and dpsi} holds. In the case where $\varphi(0)=\psi(0)=0$, one has
\[\int_{\mathopen{]}0,{+\infty}\mathclose{]}}\varphi\,\mathrm{d}\psi'=-\int_{\mathbb R_{>0}}\psi'(t)\varphi'(t)\,\mathrm{d}t=\int_{\mathopen{]}0,{+\infty}\mathclose{]}}\psi\,\mathrm{d}\varphi'.\]

\end{proof}

\begin{prop}\label{prop:x:d:varphi:prime}
Let $\varphi$ be a continuous function on $[0,+\infty]$ which is convex on $\mathbb R_{>0}$. One has
\begin{equation}
\int_{\mathbb R_{>0}} x\,\mathrm{d}\varphi'(x)=\varphi(0)-\varphi(+\infty).
\end{equation}
\end{prop}
\begin{proof}
By Fubini's theorem
\[\begin{split}\int_{\mathbb R_{>0}} x\,\mathrm{d}\varphi'(x)&=\int_{\mathbb R_{>0}}\int_{\mathopen{]}0,x\mathclose{[}}\,\mathrm{d}t\,\mathrm{d}\varphi'(x)=\int_{\mathbb R_{>0}}\int_{\mathopen{]}t,{+\infty}\mathclose{[}}\,\mathrm{d}\varphi'(x)\,\mathrm{d}t\\&=-\int_{\mathbb R_{>0}}\varphi'(t)\,\mathrm{d}t=\varphi(0)-\varphi(+\infty),
\end{split}\]
where in the third equality we have used the fact that $\lim_{t \rightarrow+\infty} \varphi'(t) = 0$ proved in Proposition \ref{prop:conv:properties}.
\end{proof}

\subsection{Transform of Legendre type}

\begin{defi}
Let $\varphi$ be a continuous function on $[0,+\infty]$ which is convex on $\mathbb R_{>0}$. We denote by $\varphi^*$ the fonction on $[0,+\infty]$ defined as
\[\forall\,\lambda\in[0,+\infty],\quad \varphi^*(\lambda):=\inf_{x\in[0,+\infty]}(x\lambda+\varphi(x)-\varphi(0)).\]
Clearly the function $\varphi^*$ is increasing and non-positive. Moreover, 
one has \[\varphi^*(0)=\inf_{x\in[0,+\infty]} \varphi(x)-\varphi(0)=\varphi(+\infty)-\varphi(0).\] 
Therefore, for any $\lambda\in[0,+\infty]$, one has
\[\varphi(+\infty)-\varphi(0)\leqslant\varphi^*(\lambda)\leqslant 0.\]
\end{defi}

\begin{prop}\phantomsection\label{theorem:integral:Legendre:trans}
Let $\varphi$ be a continuous function on $[0, +\infty]$ which is convex on $\mathbb R_{>0}$.
For $p \in {\mathbb R}_{> 1}$, one has
\[
\int_{0}^{+\infty} (-\varphi'(x))^{p} \mathrm{d} x = -(p-1)p \int_{0}^{+\infty}\lambda^{p-2}\varphi^*(\lambda)\mathrm{d}\lambda.
\]
In particular,
\begin{equation}\label{Equ: key equality}
\int_{0}^{+\infty} \varphi'(x)^2 \mathrm{d}x = -2 \int_{0}^{+\infty} \varphi^*(\lambda)\, \mathrm{d}\lambda.
\end{equation}
\end{prop}
\begin{proof}
Since $\varphi'$ is increasing one has
\[
\varphi^*(\lambda) = \inf_{x \in [0, +\infty[} \int_{0}^{x}(\lambda + \varphi'(t))\, \mathrm{d}t = 
\int_{0}^{+\infty} \min \{\lambda + \varphi'(t), 0\}\,\mathrm{d}t.
\]
Therefore, by Fubini's theorem,
\begin{align*}
\int_{0}^{+\infty}\lambda^{p-2}\varphi^*(\lambda)\,\mathrm{d}\lambda & =
\int_{0}^{+\infty}\left(\int_{0}^{+\infty} \lambda^{p-2}\min \{\lambda + \varphi'(t), 0\}\,\mathrm{d} \lambda\right)\,\mathrm{d}t \\
& = \int_{0}^{+\infty} \left( \int_{0}^{-\varphi'(t)} \lambda^{p-2} (\lambda + \varphi'(t))\, \mathrm{d} \lambda\right)\, \mathrm{d}t \\
& = \int_{0}^{+\infty} \left[ \frac{\lambda^{p}}{p}  + \frac{\varphi'(t) \lambda^{p-1}}{p-1} \right]_0^{-\varphi'(t)} \mathrm{d}t \\
& = \frac{-1}{(p-1)p}\int_{0}^{+\infty}(-\varphi'(t))^{p}\,\mathrm{d}t,
\end{align*}
as required.
\end{proof}

\subsection{Convex envelop of asymptotically linear functions}\label{Sec: convex envelop} Let $g:\mathbb R_{>0}\rightarrow\mathbb R$ be an asymptotically linear continuous function (see \S\ref{Sec: function of Green type}). We define the \emph{convex envelop}  of $g$ as the largest convex function $\widebreve{g}$ on $\mathbb R_{>0}$ which is bounded from above by $g$. Note that $\widebreve{g}$ identifies with the supremum of all affine functions bounded from above by $g$.

\begin{prop}\phantomsection\label{Pro: convex envelop of function of Green type}
Let $g:\mathbb R_{>0}\rightarrow\mathbb R$ be an asymptotically linear continuous function. Then $\widebreve{g}$ is also an asymptotically linear continuous function. Moreover, one has $\mu(g)=\mu(\widebreve{g})$ and $g(0)=\widebreve{g}(0)$.
\end{prop}
\begin{proof}
Let $\varphi_g:[0,+\infty]\rightarrow\mathbb R$ be the continuous function such that $\varphi_g(t)=g(t)-\mu(g)t$ on $\mathbb R_{>0}$. Let $M$ be a real number such that $|\varphi_g(t)|\leqslant M$ for any $t\in[0,+\infty]$. One has \[\mu(g)t-M\leqslant g(t)\leqslant\mu(g)t+M.\] Therefore,
\[\mu(g)t-M\leqslant \widebreve{g}(t)\leqslant \mu(g)t+M.\]
By Proposition \ref{prop:conv:properties}, the function \[\varphi_{\widebreve{g}}:\mathbb R_{>0}\rightarrow\mathbb R,\quad \varphi_{\widebreve{g}}(t):=\widebreve{g}(t)-\mu(g)t\] extends continuously on  $[0,+\infty]$. It remains to show that $g(0)=\widebreve{g}(0)$. Let $\varepsilon>0$. The function $t\mapsto (g(t)-g(0)+\varepsilon)/t$ is continuous on $\mathopen{]}{0},{+\infty}\mathclose{]}$ and one has
\[\lim_{t\downarrow0}\frac{g(t)-g(0)+\varepsilon}{t}=+\infty.\]
Therefore this function is bounded from below by a real number $\alpha$. Hence the function $g$ is bounded from below on $\mathbb R_{>0}$ by the affine function \[t\longmapsto \alpha t+g(0)-\varepsilon,\]
which implies that $\widebreve{g}(0)\geqslant {g}(0)-\varepsilon$. Since $g\geqslant \widebreve{g}$ and since $\varepsilon$ is arbitrary, we obtain $\widebreve{g}(0)=g(0)$.
\end{proof}

\section{Green functions on a tree of length $1$}

The purpose of this section is to establish a framework of Green functions on a tree of length $1$, which serves as a fundament of the arithmetic intersection theory of adelic $\mathbb R$-divisors on an arithmetic surface over a trivially valued field.

\subsection{Tree of length 1 associated with a set}\label{Subsec: tree of length 1}

Let $S$ be a non-empty set. We denote by $\mathcal T(S)$ the quotient set of the disjoint union $\coprod_{x\in S}[0,+\infty]$ obtained by gluing the points $0$ in the copies of $[0,+\infty]$. The quotient map from $\coprod_{x\in S}[0,+\infty]$ to $\mathcal T(S)$ is denoted by $\pi$. For each $x\in S$, we denote by $\xi_x:[0,+\infty]\rightarrow\mathcal T(S)$ the restriction of the quotient map $\pi$ to the copy of $[0,+\infty]$ indexed by $x$. The set $\mathcal T(S)$ is the union of $\xi_x([0,+\infty])$, $x\in S$. 

\begin{enonce}{Notation}
{\rm Note that the images of $0$ in $\mathcal T(S)$ by all maps $\xi_x$ are the same, which we denote by $\eta_0$. The image of $+\infty$ by the map $\xi_x$ is denoted by $x_0$. If $a$ and $b$ are elements of $[0,+\infty]$ such that $a<b$, the images of the intervals $[a,b]$, $\mathopen{[}a,b\mathclose{[}$, $\mathopen{]}a,b\mathclose{]}$, $\mathopen{]}a,b\mathclose{[}$ by $\xi_x$ are denoted by $[\xi_x(a),\xi_x(b)]$, $\mathopen{[}\xi_x(a),\xi_x(b)\mathclose{[}$, $\mathopen{]}\xi_x(a),\xi_x(b)\mathclose{]}$, $\mathopen{]}\xi_x(a),\xi_x(b)\mathclose{[}$ respectively.}
\end{enonce}

\begin{defi}
We denote by $t:\mathcal T(S)\rightarrow [0,+\infty]$ the  map which sends an element $\xi\in \xi_x([0,+\infty])$ to the unique number $a\in [0,+\infty]$ such that $\xi_x(a)=\xi$. In other words, for any $x\in S$, the restriction of $t(\ndot)$ to $[\eta_0,x_0]$ is the inverse of the injective map $\xi_x$. We call $t(\ndot)$ the \emph{parametrisation map} of $\mathcal T(S)$.
\end{defi}

\begin{defi}
We equip $\mathcal T(S)$ with the following topology. A subset $U$ of $\mathcal T(S)$ is open if and only if the conditions below are simultaneously satisfied:
\begin{enumerate}[label=\rm(\arabic*)]
\item for any $x\in S$, $\xi_x^{-1}(U)$ is an open subset of $[0,+\infty]$,
\item if $\eta_0\in U$, then $U$ contains $[\eta_0,x_0]$ for all but finitely many $x\in S$. 
\end{enumerate}
By definition, all maps $\xi_x:[0,+\infty]\rightarrow\mathcal T(S)$ are continuous. However, if $S$ is an infinite set, then the parametrisation map $t(\ndot)$ is \emph{not} continuous.

Note that the topological space $\mathcal T(S)$ is compact. We can visualise it as an infinite tree of depth $1$ whose root is $\eta_0$ and whose leaves are $x_0$ with $x\in S$.
\end{defi}

\subsection{Green functions}\label{Subsec: Green functions}

Let $S$ be a non-empty set and $w:S\rightarrow\mathbb R_{>0}$ be a map. We call \emph{Green function} on $\mathcal T(S)$ any continuous map $g$ from $\mathcal T(S)$ to $[-\infty,+\infty]$ such that, for any $x\in S$, the composition of $g$ with $\xi_x|_{\mathbb R_{>0}}$ defines an asymptotically linear function on $\mathbb R_{>0}$. For any $x\in S$, we denote by $\mu_x(g)$ the unique real number such that the function
\[(u\in\mathbb R_{>0})\longmapsto
g(\xi_x(u))-\mu_x(g)u. \]
extends to a continuous function on $[0,+\infty]$. We denote by $\varphi_g:\mathcal T(S)\rightarrow\mathbb R$ the continuous function on $\mathcal T(S)$ such that 
\[\varphi_g(\xi)=g(\xi)-\mu_x(g)t(\xi) \text{ for any $\xi\in[\eta_0,x_0]$, $x\in S$.}\]

\begin{rema}\label{Rem: lambda est presque partout nul}
Let $g$ be a Green function on $\mathcal T(S)$. It takes finite values on $\mathcal T(S)\setminus\{x_0\,:\, x\in S\}$. Moreover, for any $x\in S$, the value of $g$ at $x_0$ is finite if and only if $\mu_x(g)=0$. As $g$ is a continuous map, $g^{-1}(\mathbb R)$ contains all but finitely many $x_0$ with $x\in S$. In other words, for all but finitely many $x\in S$, one has $\mu_x(g)=0$. Note that the Green function $g$ is bounded if and only if $\mu_x(g)=0$ for any $x\in S$.
\end{rema}

\begin{defi}\label{Def: canonical divisor}
Let $g$ be a Green function on $\mathcal T(S)$. We denote by $g_{\operatorname{can}}$ the map from $\mathcal T(S)$ to $[-\infty,+\infty]$ which sends $\xi\in[\eta_0,x_0]$ to $\mu_x(g)t(\xi)$. Note that the composition of $g_{\operatorname{can}}$ with $\xi_x|_{\mathbb R_{>0}}$ is a linear function on $\mathbb R_{>0}$. We call it the \emph{canonical Green function} associated with $g$. Note that there is a unique bounded Green function $\varphi_g$ on $\mathcal T(S)$ such that $g=g_{\mathrm{can}}+\varphi_g$. We call it the \emph{bounded Green function} associated with $g$. The formula $g=g_{\mathrm{can}}+\varphi_g$ is called the \emph{canonical decomposition} of the Green function $g$. If $g=g_{\mathrm{can}}$, we say that the Green function $g$ is \emph{canonical}.
\end{defi}

\begin{prop}\label{Pro: constant except countable}
Let $g$ be a Green function on $\mathcal T(S)$. For all but countably many $x\in S$, the restriction of $g$ on $[\eta_0,x_0]$ is a constant function.
\end{prop}
\begin{proof}
For any $n\in\mathbb N$ such that  $n\geqslant 1$, let $U_n$ be set of $\xi\in \mathcal T(S)$ such that \[|g(\xi)-g(\eta_0)|<n^{-1}.\] This is an open subset of $\mathcal T(S)$ which contains $\eta_0$. Hence there is a finite subset $S_n$ of $S$ such that $[\eta_0,x_0]\subset U_n$ for any $x\in S\setminus S_n$. Let $S'=\bigcup_{n\in\mathbb N,\,n\geqslant 1}S_n$. This is a countable subset of $S$. For any $x\in S\setminus S'$ and any $\xi\in[\eta_0,x_0]$, one has $g(\xi)=g(\eta_0)$
\end{proof}

\begin{rema}
It is clear that, if $g$ is a Green function on $\mathcal T(S)$, for any $a\in\mathbb R$, the function $ag:\mathcal T(S)\rightarrow [-\infty,+\infty]$ is a Green function on $\mathcal T(S)$. Moreover, the canonical decomposition of Green functions allows to define the sum of two Green functions. Let $g_1$ and $g_2$ be two Green functions on $\mathcal T(S)$. We define $g_1+g_2$ as $(g_{1,\mathrm{can}}+g_{2,\mathrm{can}})+(\varphi_{g_1}+\varphi_{g_2})$. 

Note that the set of all Green functions, equipped with the addition and the multiplication by a scalar, forms a vector space over $\mathbb R$.
\end{rema}

\subsection{Pairing of Green functions}\label{Subsec: pairing}

Let $S$ be a non-empty set and $w:S\rightarrow\mathbb R_{>0}$ be  a map, called a \emph{weight function}. We say that a Green function $g$ on $\mathcal T(S)$ is \emph{pairable with respect to $w$} if the following conditions are satisfied:
\begin{enumerate}[label=\rm(\arabic*)]
\item for any $x\in S$, the function $\varphi_g\circ \xi_x|_{\mathbb R_{>0}}$ belongs to $L_1^2(\mathbb R_{>0})$ (see \S\ref{Sec: function of Green type}),
\item one has
\[\sum_{x\in S}w(x)\int_{\mathbb R_{>0}}(\varphi_g\circ \xi_x|_{\mathbb R_{>0}})'(u)^2\,\mathrm{d}u<+\infty.\]
\end{enumerate}
For each $x\in S$ we fix a representative of the function $(\varphi_g\circ\xi_x|_{\mathbb R_{>0}})'$ and we denote by \[\varphi_g':\bigcup_{x\in S}\,\mathopen{]}\eta_0,x_0\mathclose{[}\longrightarrow \mathbb R\] the function which sends $\xi\in\mathopen{]}\eta_0,x_0\mathclose{[}$
to $(\varphi_g\circ\xi_x|_{\mathbb R_{>0}})'(t(\xi))$.

We equip $\coprod_{x\in S}[0,+\infty]$
with the disjoint union of the weighted Lebesgue measure $w(x)\,\mathrm{d}u$, where $\mathrm{d}u$  denotes the Lebesgue measure on $[0,+\infty]$. We denote by $\nu_{S,w}$ the push-forward of this measure by the projection map
\[\coprod_{x\in S}[0,+\infty]\longrightarrow\mathcal T(S).\]
Then the function $\varphi_g'$ is square-integrable with respect to the measure $\nu_{S,w}$.

\begin{defi}\phantomsection\label{Def: pairing of Green functions} Note that pairable Green functions form a vector subspace of the vector space of Green functions.
Let $g_1$ and $g_2$ be pairable Green functions on $\mathcal T(S)$. We define the \emph{pairing} of $g_1$ and $g_2$ as
\[\sum_{x\in S}w(x)\Big(\mu_x(g_1){g_2}(\eta_0)+\mu_x(g_2){g_1}(\eta_0)\Big) - \int_{\mathcal T(S)}\varphi_{g_1}'(\xi)\varphi_{g_2}'(\xi)\,\nu_{S,w}(\mathrm{d}\xi),\]
denoted by $\langle g_1,g_2\rangle_w$,  called the \emph{pairing} of Green functions $g_1$ and $g_2$. Note that $\emptyinnprod_w$ is a symmetric bilinear form on the vector space of pairable Green functions.
\end{defi}

\subsection{Convex Green functions}\label{Sec:convex Green}

Let $S$ be a non-empty set. We say that a Green function $g$ on $\mathcal T(S)$ is \emph{convex} if, for any element $x$ of $S$, the function ${g\circ\xi_x}$ on $\mathbb R_{>0}$ is convex. 

\begin{enonce}{Convention}
\rm If $g$ is a convex Green function on $\mathcal T(S)$, by convention  we choose, for each $x\in S$, the right derivative of $\varphi_g\circ\xi_x|_{\mathbb R_{>0}}$ to represent the derivative of $\varphi_g\circ\xi_x|_{\mathbb R_{>0}}$ in the sense of distribution. In other words, $\varphi_g'\circ\xi_x|_{\mathbb R_{>0}}$ is given by the right derivative of the function $\varphi_g\circ\xi_x|_{\mathbb R_{>0}}$. Moreover, for any $x\in S$, we denote by $\varphi_g'(\eta_0;x)$ the element $\varphi_{g\circ\xi_x}'(0)\in[-\infty,0]$ (see Proposition \ref{prop:conv:properties} \ref{Item: limit of left derivative}). We emphasis that $\varphi_{g\circ\xi_x}'(0)$ could differ when $x$ varies.
\end{enonce}


\begin{defi}
Let $g$ be a Green function on $\mathcal T(S)$. We call \emph{convex envelop} of $g$ and we denote by $\widebreve{g}$ the continuous map from $\mathcal T(S)$ to $[-\infty,+\infty]$ such that, for any $x\in S$, $\widebreve{g}\circ\xi_x|_{\mathbb R_{>0}}$ is the convex envelop of $g\circ\xi_x|_{\mathbb R_{>0}}$ (see \S\ref{Sec: convex envelop}). By Proposition \ref{Pro: convex envelop of function of Green type}, the function $\widebreve{g}$ is well defined and defines a convex Green function on $\mathcal T(S)$. Moreover, it is the largest convex Green function on $\mathcal T(S)$ which is bounded from above by $g$.
\end{defi}

\begin{prop}\phantomsection\label{Pro: convex envelop of Green functions}
Let $g$ be a Green function on $\mathcal T(S)$. The following equalities hold:
\[g_{\mathrm{can}}=\widebreve{g}_{\!\mathrm{can}},\quad
g(\eta_0)=\widebreve{g}(\eta_0),\quad \widebreve{\varphi}_{\!g}=\varphi_{\widebreve{g}}.\]
\end{prop}
\begin{proof}
The first two equalities follows from Proposition \ref{Pro: convex envelop of function of Green type}. The third equality comes from the first one and the fact that $\widebreve{g}=g_{\mathrm{can}}+\widebreve{\varphi}_{\!g}$.
\end{proof}

\subsection{Infimum slopes}
\label{Sec: infimum slopes}

Let $S$ be a non-empty set and $g$ be a Green function on $\mathcal T(S)$. For any $x\in S$, we denote by $\mu_{\inf,x}(g)$ the element
\[\inf_{\xi\in\mathopen{]}\eta_0,x_0\mathclose{[}}\frac{g(\xi)}{t(\xi)}\in\mathbb R\cup\{-\infty\}.\]
Clearly one has $\mu_{\inf,x}(g)\leqslant\mu_{x}(g)$. Therefore, by Remark \ref{Rem: lambda est presque partout nul} we obtain that $\mu_{\inf,x}(g)\leqslant 0$ for all but finitely many $x\in S$. If $w:S\rightarrow\mathbb R_{\geqslant 0}$ is a weight function, we define the \emph{global infimum slope} of $g$ with respect to $w$ as
\[\sum_{x\in X^{(1)}}\mu_{\inf,x}(g) w(x)\in\mathbb R\cup\{-\infty\}.\]
This element is well defined because $\mu_{\inf,x}(g)\leqslant 0$ for all but finitely many $x\in S$. If there is no ambiguity about the weight function (notably when $S$ is the set of closed points of a regular projective curve cf. Definition~\ref{def:mu:inf}), the global infimum slope of $g$ is also denoted by $\mu_{\inf}(g)$.


\begin{prop}\label{Pro: relation between mu and varphig}
Let $g$ be a convex Green function on $\mathcal T(S)$. For any $x\in S$ one has  
\[\mu_{\inf,x}(g-g(\eta_0))=\mu_x(g)+\varphi_g'(\eta_0;x).\] 
\end{prop}
\begin{proof}
This is a direct consequence of Proposition \ref{prop:conv:properties} \ref{Item: limit of left derivative}.
\end{proof}

\section{Graded linear series}

Let $k$ be a field and $X$ be a regular projective curve over $\Spec k$. We denote by $X^{(1)}$ the set of closed points of $X$. By \emph{$\mathbb R$-divisor} on $X$, we mean an element in the free $\mathbb R$-vector space generated by $X^{(1)}$. We denote by $\operatorname{Div}_{\mathbb R}(X)$ the $\mathbb R$-vector space of $\mathbb R$-divisors on $X$. If $D$ is an element of $\operatorname{Div}_{\mathbb R}(X)$, the coefficient of $x$ in the expression of $D$ into a linear combination of elements of $X^{(1)}$ is denoted by $\operatorname{ord}_x(D)$. If $\operatorname{ord}_x(D)$ belongs to $\mathbb Q$ for any $x\in X^{(1)}$, we say that $D$ is a \emph{$\mathbb Q$-divisor}; if $\operatorname{ord}_x(D)\in\mathbb Z$ for any $x\in X^{(1)}$, we say that $D$ is a \emph{divisor} on $X$. The subsets of $\operatorname{Div}_{\mathbb R}(X)$ consisting of $\mathbb Q$-divisors and divisors are denoted by $\operatorname{Div}_{\mathbb Q}(X)$ and $\operatorname{Div}(X)$, respectively. 

Let $D$ be an $\mathbb R$-divisor on $X$. We define the \emph{degree} of $D$ to be 
\begin{equation}\label{Equ: degree of an r divisor}\deg(D):=\sum_{x\in X^{(1)}}[k(x):k]\operatorname{ord}_x(D),\end{equation}
where for $x\in X$, $k(x)$ denotes the residue field of $x$.   Denote by $\operatorname{Supp}(D)$ the set of all $x\in X^{(1)}$ such that $\operatorname{ord}_x(D)\neq 0$, called the \emph{support} of the $\mathbb R$-divisor $D$. This is  a finite subset of $X^{(1)}$. Although $X^{(1)}$ is an infinite set, \eqref{Equ: degree of an r divisor} is actually a finite sum: one has
\[\deg(D)=\sum_{x\in\operatorname{Supp}(D)}\operatorname{ord}_x(D)[k(x):k].\]

Denote by $\operatorname{Rat}(X)$ the field of rational functions on $X$. If $f$ is a non-zero element of $\operatorname{Rat}(X)$, we denote by $(f)$ the \emph{principal divisor} associated with $f$, namely the divisor on $X$ given by
\[\sum_{x\in X^{(1)}}\operatorname{ord}_x(f) x,\]
where $\operatorname{ord}_x(f)\in\mathbb Z$ denotes the valuation of $f$ with respect to the discrete valuation ring $\mathcal O_{X,x}$. The map $\mathrm{Rat}(X)^{\times}\rightarrow\mathrm{Div}(X)$ is additive and hence induces an $\mathbb R$-linear map 
\[\mathrm{Rat}(X)^{\times}_{\mathbb R}:=\mathrm{Rat}(X)^{\times}\otimes_{\mathbb Z}\mathbb R\longrightarrow\mathrm{Div}_{\mathbb R}(X),\]
which we still denote by $f\mapsto (f)$.

\begin{defi}We say that an $\mathbb R$-divisor $D$ is \emph{effective} if $\operatorname{ord}_x(D)\geqslant 0$ for any $x\in X^{(1)}$. We denote by $D\geqslant 0$ the condition ``\emph{$D$ is effective}''. For any $\mathbb R$-divisor $D$ on $X$, we denote by $H^0(D)$ the set
\[\{f\in\operatorname{Rat}(X)^{\times}\,:\,(f)+D\geqslant 0\}\cup\{0\}.\]
It is a finite-dimensional $k$-vector subspace of $\operatorname{Rat}(X)$. We denote by $\mathrm{genus}(X)$  
the genus of the curve $X$ relatively to $k$. The theorem of Riemann-Roch implies that, if $D$ is a divisor such that $\deg(D)>2\operatorname{genus}(X)-2$, then one has
\begin{equation}\label{Equ: Riemann-Roch}
\dim_k(H^0(D))=\deg(D)+1-\operatorname{genus}(X).
\end{equation}
We refer the readers to \cite[Lemma 2.2]{MR3299845} for a proof.

Let $D$ be an $\mathbb R$-divisor on $X$. We denote by $\Gamma(D)_{\mathbb R}^{\times}$ the set
\[\{f\in\mathrm{Rat}(X)_{\mathbb R}^{\times}\,:\,(f)+D\geqslant 0\}.\]
This is an $\mathbb R$-vector subspace of $\mathrm{Rat}(X)_{\mathbb R}^{\times}$. Similarly, we denote by $\Gamma(D)_{\mathbb Q}^{\times}$ the $\mathbb Q$-vector subspace
\[\{f\in\mathrm{Rat}(X)_{\mathbb Q}^{\times}\,:\,(f)+D\geqslant 0\}\]
of $\mathrm{Rat}(X)_{\mathbb Q}^{\times}$. Note that one has
\begin{equation}\label{Equ: Gamma D times}\Gamma(D)_{\mathbb Q}^{\times}=\bigcup_{n\in\mathbb N,\,n\geqslant 1}\{f^{\frac 1n}\,:\,f\in H^0(nD)\setminus\{0\}\}.\end{equation}
\end{defi}

\begin{defi}
Let $D$ be an $\mathbb R$-divisor on $X$. We denote by $\lfloor D\rfloor$ and $\lceil D\rceil$ the divisors on $C$ such that 
\[\operatorname{ord}_x(\lfloor D\rfloor)=\lfloor \operatorname{ord}_x(D)\rfloor,\quad\operatorname{ord}_x(\lceil D\rceil)=\lceil\operatorname{ord}_x(D)\rceil.\]
Clearly one has $\deg(\lfloor D\rfloor)\leqslant\deg(D)\leqslant\deg(\lceil D\rceil)$. Moreover, 
\begin{gather}\label{Equ: lower bound of floor D}
\deg(\lfloor D\rfloor)> \deg(D)-\sum_{x\in\operatorname{Supp}(D)}[k(x):k],\\
\label{Equ: upper bound of ceil D}
\deg(\lceil D\rceil)<\deg(D)+\sum_{x\in\operatorname{Supp}(D)}[k(x):k].
\end{gather}

 Let $(D_i)_{i\in I}$ be a family of $\mathbb R$-divisors on $X$ such that \[\sup_{i\in I}\operatorname{ord}_{x}(D_i)=0\] for all but finitely many $x\in X^{(1)}$. We denote by $\sup_{i\in I}D_i$ the $\mathbb R$-divisor such that 
\[\forall\,x\in X^{(1)},\quad \operatorname{ord}_x\big(\sup_{i\in I}D_i\big)=\sup_{i\in I}\operatorname{ord}_x(D_i).\]

\end{defi}

\begin{prop}\phantomsection\label{Pro: asymptotic RR}
Let $D$ be an $\mathbb R$-divisor on $X$ such that $\deg(D)\geqslant 0$. One has 
\begin{equation}\label{Equ: limit of H0 is deg}\lim_{n\rightarrow+\infty}\frac{\dim_k(H^0(nD))}{n}=\deg(D).\end{equation}
\end{prop}
\begin{proof}
We first assume that $\deg(D)>0$. By \eqref{Equ: lower bound of floor D}, for sufficiently positive integer $n$, one has $\deg(\lfloor nD\rfloor)>2\operatorname{genus}(X)-2$.
Therefore, \eqref{Equ: Riemann-Roch} leads to
\[\dim_k(H^0(\lfloor nD\rfloor))= \deg(\lfloor nD\rfloor)+1-\mathrm{genus}(X).\]
Moreover, since $\deg(D)>0$ one has $\deg(\lceil nD\rceil)\geqslant n\deg(D)>2\operatorname{genus}(X)-2$ for sufficiently positive $n\in\mathbb N_{\geqslant 1}$. Hence 
\eqref{Equ: Riemann-Roch} leads to
\[\dim_k(H^0(\lceil nD\rceil))= \deg(\lceil nD\rceil)+1-\mathrm{genus}(X).\]
Since $H^0(\lfloor nD\rfloor)\subseteq H^0(nD)\subseteq H^0(\lceil nD\rceil)$, we obtain
\[\frac{\deg(\lfloor nD\rfloor)+1-\operatorname{genus}(X)}{n}\leqslant\frac{\dim_k(H^0(nD))}{n}\leqslant\frac{\deg(\lceil nD\rceil)+1-\operatorname{genus}(X)}{n}.\]
Taking limit when $n\rightarrow+\infty$, by \eqref{Equ: lower bound of floor D} and \eqref{Equ: upper bound of ceil D} we obtain \eqref{Equ: limit of H0 is deg}.

We now consider the case where $\deg(D)=0$. Let $E$ be an effective $\mathbb R$-Cartier divisor such that $\deg(E)>0$. For any $\varepsilon>0$ one has
\[\begin{split}\limsup_{n\rightarrow+\infty}\frac{\dim_k(H^0(nD))}{n}&\leqslant\lim_{n\rightarrow+\infty}\frac{\dim_k(H^0(n(D+\varepsilon E)))}{n}\\
&=\deg(D+\varepsilon E)=\varepsilon\deg(E).
\end{split}\]
Since $\varepsilon$ is arbitrary, the equality \eqref{Equ: limit of H0 is deg} still holds.
\end{proof}

\begin{prop}\phantomsection\label{Pro: sup of s in Gamma D}
Let $D$ be an $\mathbb R$-divisor on $X$ such that $\deg(D)>0$. Then one has
\begin{equation}\label{Equ: supremum is D}\sup_{s\in\Gamma(D)^{\times}_{\mathbb Q}}(s^{-1})=D.\end{equation}
\end{prop}
\begin{proof}
For any $s\in\Gamma(D)_{\mathbb Q}^{\times}$ one has \[\operatorname{ord}_x(s)+\operatorname{ord}_x(D)\geqslant 0\] and hence $\operatorname{ord}_x(s^{-1})\leqslant\operatorname{ord}_x(D)$.

    For any $x\in X^{(1)}$ and any $\varepsilon>0$, we pick an $\mathbb R$-divisor $D_{x,\varepsilon}$ on $X$ such that $D-D_{x,\varepsilon}$ is effective, $\operatorname{ord}_x(D_{x,\varepsilon})=\operatorname{ord}_x(D)$ and $0<\deg(D_{x,\varepsilon})<\varepsilon$. Since $\deg(D_{x,\varepsilon})>0$, the set $\Gamma(D_{x,\varepsilon})_{\mathbb Q}^{\times}$ is not empty (see \eqref{Equ: Gamma D times} and Proposition \ref{Pro: asymptotic RR}). This set is also contained in $\Gamma(D)^{\times}_{\mathbb Q}$ since $D_{x,\varepsilon}\leqslant D$. Let $f$ be an element of $\Gamma(D_{x,\varepsilon})_{\mathbb Q}^{\times}$. One has \[D_{x,\varepsilon}+(f)\geqslant 0\quad\text{ and }\quad\deg(D_{x,\varepsilon}+(f))=\deg(D_{x,\varepsilon})<\varepsilon.\] Therefore
\[\operatorname{ord}_x(D+(f))=\operatorname{ord}_x(D_{x,\varepsilon}+(f))\leqslant\frac{\varepsilon}{[\kappa(x):k]},\]
which leads to 
\[\operatorname{ord}_x(f^{-1})\geqslant\operatorname{ord}_x(D)-\frac{\varepsilon}{[\kappa(x):k]}.\]
Since $\varepsilon>0$ is arbitrary, we obtain 
\[\sup_{s\in\Gamma(D)_{\mathbb Q}^{\times}}\operatorname{ord}_x(s^{-1})=\operatorname{ord}_x(D).\]
\end{proof}

\begin{rema}\phantomsection\label{Rem: degree 0 effective}
Let $D$ be an $\mathbb R$-divisor on $X$. Note that one has
\[\sup_{s\in\Gamma(D)_{\mathbb R}^{\times}}(s^{-1})\leqslant D.\]
Therefore, the above proposition implies that, if $\deg(D)>0$, then  
\[\sup_{s\in\Gamma(D)_{\mathbb R}^{\times}}(s^{-1})= D.\]
This equality also holds when $\deg(D)=0$ and $\Gamma(D)_{\mathbb R}^{\times}\neq\emptyset$. In fact, if $s$ is an element of $\Gamma(D)_{\mathbb R}^{\times}$, then one has $D+(s)\geqslant 0$. Moreover, since $\deg(D)=0$, one has $\deg(D+(s))=\deg(D)+\deg((s))=0$ and hence $D+(s)=0$. Similarly, if $D$ is an $\mathbb R$-divisor on $X$ such that $\Gamma(D)_{\mathbb Q}^{\times}\neq\emptyset$, then the equality
\[\sup_{s\in\Gamma(D)_{\mathbb Q}^{\times}}(s^{-1})=D\]
always holds. 
\end{rema}

\begin{defi}\phantomsection\label{Def: generic point of graded linear series}
Let $\operatorname{Rat}(X)$ be the field of rational functions on $X$. By \emph{graded linear series} on $X$, we refer to a graded sub-$k$-algebra $V_\sbullet=\bigoplus_{n\in\mathbb N}V_nT^n$ of $\operatorname{Rat}(X)[T]=\bigoplus_{n\in\mathbb N}\operatorname{Rat}(X)T^n$ which satisfies the following conditions: 
\begin{enumerate}[label=\rm(\arabic*)]
\item $V_0=k$,
\item there exists $n\in\mathbb N_{\geqslant 1}$ such that $V_n\neq\{0\}$
\item there exists an $\mathbb R$-divisor $D$ on $X$ such that $V_n\subseteq H^0(nD)$ for any $n\in\mathbb N$.
\end{enumerate}
If $W$ is a $k$-vector subspace of $\operatorname{Rat}(X)$, we denote by $k(W)$ the extension \[k(\{f/g\,:\,(f,g)\in (W\setminus\{0\})^2\})\]
of $k$.
If $V_\sbullet$ is a graded linear series on $V$, we set 
\[k(V_\sbullet):=k\Big(\bigcup_{n\in\mathbb N_{\geqslant 1}}\{f/g\,:\,(f,g)\in (V_n\setminus\{0\})^2\}\Big). \]
If $k(V_\sbullet)=\operatorname{Rat}(X)$, we say that the graded linear series $V_\sbullet$ is \emph{birational}.
\end{defi}

\begin{exem}\phantomsection\label{Exe: birational graded linear series}
Let $D$ be an $\mathbb R$-divisor on $X$ such that $\deg(D)>0$. Then the total graded linear series $\bigoplus_{n\in\mathbb N}H^0(nD)$ is birational.
\end{exem}

\begin{prop}\label{Pro: corps engendre par un systeme lineare}
Let $V_\sbullet$ be a graded linear series on $X$. The set \[\mathbb N(V_\sbullet):=\{n\in\mathbb N_{\geqslant 1}\,:\,V_n\neq\{0\}\}\] equipped with the additive law forms a sub-semigroup of $\mathbb N_{\geqslant 1}$. Moreover, for any $n\in\mathbb N(V_\sbullet)$ which is sufficiently positive, one has $k(V_\sbullet)=k(V_n)$. 
\end{prop}
\begin{proof}
Let $n$ and $m$ be elements of $\mathbb N(V_\sbullet)$. If $f$ and $g$ are respectively non-zero elements of $V_n$ and $V_m$, then $fg$ is a non-zero element of $V_{n+m}$. Hence $n+m$ belongs to $\mathbb N(V_\sbullet)$. Therefore, $\mathbb N(V_\sbullet)$ is a sub-semigroup of $\mathbb N_{\geqslant 1}$. In particular, if $d\geqslant 1$ is a generator of the subgroup of $\mathbb Z$ generated by $\mathbb N(V_\sbullet)$, then there exists $N_0\in\mathbb N_{\geqslant 1}$ such that $dn\in\mathbb N(V_\sbullet)$ for any $n\in\mathbb N$, $n\geqslant N_0$.

Since $k \subseteq k(V_\sbullet) \subseteq \operatorname{Rat}(X)$ and $\operatorname{Rat}(X)$ is finitely generated over $k$, the extension $k(V_\sbullet)/k$ is  finitely generated (see \cite[Chapter~V, \S14, $\text{n}^{\circ}$7, Corollary~3]{Bourbaki_Alg}). Therefore,
there exist a finite family $\{n_1,\ldots,n_r\}$ of elements in $\mathbb N_{\geqslant 1}$, together with pairs $(f_i,g_i)\in (V_{dn_i}\setminus\{0\})^2$ such that $k(V_\sbullet)=k(f_1/g_1,\ldots,f_r/g_r)$. Let  $N\in\mathbb N$ such that \[N-\max\{n_1,\ldots,n_r\}\geqslant N_0.\] For any $i\in\{1,\ldots,r\}$ and $M\in\mathbb N_{\geqslant N}$, let $h_{M,i}\in V_{d(M-n_i)}\setminus\{0\}$. Then \[(h_{M,i}f_i,h_{M,i}g_i)\in (V_{dM}\setminus\{0\})^2,\] which shows that $k(V_\sbullet)=k(V_{dM})$.
\end{proof}

\begin{defi}
If $V_\sbullet$ is a graded linear series, we define $\Gamma(V_\sbullet)^{\times}_{\mathbb Q}$ as
\[\bigcup_{n\in\mathbb N_{\geqslant 1}}\{f^{\frac 1n}\,|\,f\in V_n\setminus\{0\}\},\]
and let $D(V_\sbullet)$ be the following $\mathbb R$-divisor 
\[\sup_{s\in\Gamma(V_\sbullet)^{\times}_{\mathbb Q}}(s^{-1}),\]
called the \emph{$\mathbb R$-divisor generated by $V_\sbullet$}.
The conditions (2) and (3) in Definition~\ref{Def: generic point of graded linear series} 
show that the $\mathbb R$-divisor $D(V_\sbullet)$ is well defined and has non-negative degree.
\end{defi}

\begin{prop}\phantomsection\label{Pro: convergence of dim}
Let $V_\sbullet$ be a birational graded linear series on $X$. One has
\begin{equation}\label{Equ: Hilbert Samuel for graded linear series}\lim_{\begin{subarray}{c}n\in\mathbb N,\,V_n\neq\{0\}\\
n\rightarrow+\infty
\end{subarray}}\frac{\dim_k(V_n)}{n}=\deg(D(V_\sbullet))>0.\end{equation}
\end{prop}
\begin{proof}
By definition, for any $n\in\mathbb N$ one has $V_n\subseteq H^0(nD(V_\sbullet))$. Therefore Proposition \ref{Pro: asymptotic RR} leads to 
\[\limsup_{n\rightarrow+\infty}\frac{\dim_k(V_n)}{n}\leqslant\deg(D(V_\sbullet)).\]
Let $p$ be a sufficiently positive integer (so that $\operatorname{Rat}(X)=k(V_p)$).
Let \[V^{[p]}_\sbullet:=\bigoplus_{n\in\mathbb N}\mathrm{Im}(S^nV_p\longrightarrow V_{np})T^n.\]
Clearly one has $D(V_\sbullet^{[p]})\leqslant pD(V_\sbullet)$. Moreover, since $\operatorname{Rat}(X)=k(V_p)$, $X$ identifies with the normalisation of $\mathrm{Proj}(V_\sbullet^{[p]})$, and the pull-back on $X$ of the tautological line bundle on $\mathrm{Proj}(V_\sbullet^{[p]})$ identifies with $\mathcal O(D(V_\sbullet^{[p]}))$. This leads to 
\[\frac{1}{p}\deg(D(V_\sbullet^{[p]}))=\lim_{n\rightarrow+\infty}\frac{\dim_k(V^{[p]}_n)}{pn}\leqslant\liminf_{\begin{subarray}{c}n\in\mathbb N,\, V_n\neq \{0\}\\n\rightarrow+\infty
\end{subarray}}\frac{\dim_k(V_n)}{n}.\]
As the map $p\mapsto \frac 1p D(V_\sbullet^{[p]})$ preserves the order if we consider the relation of divisibility on $p$, by the relation $D(V_\sbullet)=\sup_{p}\frac 1pD(V_\sbullet^{[p]})$ we obtain that 
\[\deg(D(V_\sbullet))=\sup_{p}\frac 1p\deg(D(V_\sbullet^{[p]}))\leqslant\liminf_{\begin{subarray}{c}n\in\mathbb N,\, V_n\neq \{0\}\\n\rightarrow+\infty
\end{subarray}}\frac{\dim_k(V_n)}{n}.\]
Therefore the equality in \eqref{Equ: Hilbert Samuel for graded linear series} holds. 

If $p$ is a positive integer such that $\mathrm{Rat}(X)=k(V_p)$, then $V_p$ admits an element $s$ which is transcendental over $k$. In particular, the graded linear series $V_\sbullet^{[p]}$ contains a polynomial ring of one variable, which shows that
\[\liminf_{n\rightarrow+\infty}\frac{\dim_k(V_n)}{n}>0.\]
\end{proof}

\section{Arithmetic surface over a trivially valued field}

In this section, we fix a commutative field $k$ and we denote by $|\ndot|$ the trivial absolute value on $k$. Let $X$ be a regular projective curve over $\Spec k$. We denote by $X^{\mathrm{an}}$ the \emph{Berkovich topological space} associated with $X$. Recall that, as a set $X^{\mathrm{an}}$ consists of couples of the form $\xi=(x,|\ndot|_\xi)$, where $x$ is a scheme point of $X$ and $|\ndot|_\xi$ is an absolute value on the residue field $\kappa(x)$ of $x$, which extends the trivial absolute value on $k$. We denote by $j:X^{\mathrm{an}}\rightarrow X$ the map sending any pair in $X^{\mathrm{an}}$ to its first coordinate. For any $\xi\in X^{\mathrm{an}}$, we denote by $\widehat{\kappa}(\xi)$ the completion of $\kappa(j(\xi))$ with respect to the absolute value $|\ndot|_\xi$, on which $|\ndot|_\xi$ extends in a unique way. For any regular function $f$ on a Zariski open subset $U$ of $X$, we let $|f|$ be the function on $j^{-1}(U)$ sending any $\xi$ to the absolute value of $f(j(\xi))\in\kappa(j(\xi))$ with respect to $|\ndot|_\xi$. The \emph{Berkovich topology} on $X^{\mathrm{an}}$ is defined as the most coarse topology making the map $j$ and all functions of the form $|f|$ continuous, where $f$ is a regular function on a Zariski open subset of $X$.

\begin{rema}
Let $X^{(1)}$ be the set of all closed points of $X$. The Berkovich topological space $X^{\operatorname{an}}$ identifies with the tree $\mathcal T(X^{(1)})$, where 
\begin{enumerate}[label=(\alph*)]
\item the root point $\eta_0$ of the tree $\mathcal T(X^{(1)})$ corresponds to the pair consisting of the generic point $\eta$ of $X$ and the trivial absolute value on the field of rational functions on $X$,
\item for any $x\in X^{(1)}$, the leaf point $x_0\in\mathcal T(X^{(1)})$ corresponds to the closed point $x$ of $X$ together with the trivial absolute value on the residue field $\kappa(x)$,
\item for any $x\in X^{(1)}$, any $\xi\in\mathopen{]}\eta_0,x_0\mathclose{[}$ corresponds to the pair consisting of the generic point $\eta$ of $X$ and the absolute value $\mathrm{e}^{-t(\xi)\operatorname{ord}_x(\ndot)}$, with $\operatorname{ord}_x(\ndot)$ being the discrete valuation on the field of rational functions $\operatorname{Rat}(X)$ corresponding to $x$.
\end{enumerate}
\end{rema}

\subsection{Metrised divisors} We call \emph{metrised $\mathbb R$-divisor} on $X$ any copy $(D,g)$, where $D$ is an \emph{$\mathbb R$-divisor} on $X$ and $g$ is a Green function on $\mathcal T(X^{(1)})$ such that $\mu_x(g)=\operatorname{ord}_x(D)$ for any $x\in X^{(1)}$ (see \S\ref{Subsec: Green functions}).  If in addition $D$ is a \emph{$\mathbb Q$-divisor} (resp. divisor), we say that $D$ is a \emph{metrised $\mathbb Q$-divisor} (resp. \emph{metrised divisor}).

If $(D,g)$ is a metrised $\mathbb R$-divisor on $X$ and $a$ is a real number, then $(aD,ag)$ is also a metrised $\mathbb R$-divisor, denoted by $a(D,g)$. Moreover, if $(D_1,g_1)$ and $(D_2,g_2)$ are two metrised $\mathbb R$-divisors on $X$, then $(D_1+D_2,g_1+g_2)$ is also a metrised $\mathbb R$-divisor, denoted by $(D_1,g_1)+(D_2,g_2)$. The set $\widehat{\mathrm{Div}}_{\mathbb R}(X)$ of all metrised $\mathbb R$-divisors on $X$ then forms a vector space over $\mathbb R$.

If $(D,g)$ is a metrised $\mathbb R$-divisor on $X$, we say that $g$ is a \emph{Green function of the $\mathbb R$-divisor $D$}.

\begin{rema}
\begin{enumerate}[label=\rm(\arabic*)]
\item Let $(D,g)$ be a metrised $\mathbb R$-divisor on $X$. Note that the $\mathbb R$-divisor part $D$ is uniquely determined by the Green function $g$. Therefore the study of metrised $\mathbb R$-divisors on $X$ is equivalent to that of Green functions on the infinite tree $\mathcal T(X^{(1)})$. The notation of pair $(D,g)$ facilites however the presentation on the study of metrised linear series of $(D,g)$.
\item Let $D$ be an $\mathbb R$-divisor on $X$, there is a unique canonical Green function on $\mathcal T(X^{(1)})$ (see Definition \ref{Def: canonical divisor}), denoted by $g_{D}$, such that $(D,g_D)$ is an metrised $\mathbb R$-divisor. Note that, for any metrised $\mathbb R$-divisor $(D,g)$ which admits $D$ as its underlying $\mathbb R$-divisor, one has $g_D=g_{\mathrm{can}}$ (see Definition \ref{Def: canonical divisor}). In particular, if $(D,g)$ is a metrised $\mathbb R$-divisor such that $D$ is the zero $\mathbb R$-divisor, then the Green function $g$ is bounded.  
\end{enumerate}
\end{rema}

\begin{defi}
Let $\mathrm{Rat}(X)$ be the field of rational functions on $X$ and $\mathrm{Rat}(X)^{\times}_{\mathbb R}$ be the $\mathbb R$-vector space $\mathrm{Rat}(X)^{\times}\otimes_{\mathbb Z}\mathbb R$. For any $\phi$ in $\mathrm{Rat}(X)^{\times}_{\mathbb R}$, the couple $((\phi),g_{(\phi)})$ is called the \emph{principal metrised $\mathbb R$-divisor} associated with $\phi$ and is denoted by $\widehat{(\phi)}$. 
\end{defi}

\begin{defi} 
If $(D,g)$ is a metrised $\mathbb R$-divisor, for any $\phi\in\Gamma(D)^{\times}_{\mathbb R}$, we define 
\begin{equation}\label{Equ: norm g}\norm{\phi}_{g}:=\exp\bigg(-\inf_{\xi\in\mathcal T(X^{(1)})}(g_{(\phi)}+g)(\xi)\bigg).\end{equation}
By convention, $\norm{0}_g$ is defined to be zero.
\end{defi}

\subsection{Ultrametrically normed vector spaces}

Let $E$ be a finite-dimensional vector space over $k$ (equipped with the trivial absolute value). By \emph{ultrametric norm} on $E$, we mean a map $\norm{\ndot}:E\rightarrow\mathbb R_{\geqslant 0}$ such that 
\begin{enumerate}[label=\rm(\arabic*)]
\item for any $x\in E$, $\norm{x}=0$ if and only if $x=0$,
\item $\norm{ax}=\norm{x}$ for any $x\in E$ and $a\in k\setminus\{0\}$,
\item for any $(x,y)\in E\times E$, $\norm{x+y}\leqslant\max\{\norm{x},\norm{y}\}$.
\end{enumerate}
Let $r$ be the rank of $E$ over $k$. We define the \emph{determinant norm associated with $\norm{\ndot}$ the norm $\norm{\ndot}_{\det}$ on $\det(E)=\Lambda^r(E)$ such that
\[\forall\,\eta\in\det(E),\quad \norm{\eta}=\inf_{\begin{subarray}{c}
(s_1,\ldots,s_r)\in E^r\\
s_1\wedge\ldots\wedge s_r=\eta
\end{subarray}}\norm{s_1}\cdots\norm{s_r}.\]}We define the \emph{Arakelov degree} of $(E,\norm{\ndot})$ as
\begin{equation}
\widehat{\deg}(E,\norm{\ndot})=-\ln\norm{\eta}_{\det},
\end{equation}
where $\eta$ is a non-zero element of $\det(E)$.
We then define the \emph{positive Arakelov degree} as
\[\widehat{\deg}_+(E,\norm{\ndot}):=\sup_{F\subset E}\widehat{\deg}(F,\norm{\ndot}_F),\]
where $F$ runs over the set of all vector subspaces of $E$, and $\norm{\ndot}_F$ denotes the restriction of $\norm{\ndot}$ to $F$.

\begin{exem} Let $(D,g)$ be a metrised $\mathbb R$-divisor on $X$.
Note that the restriction of $\norm{\ndot}_g$ to $H^0(D)$ defines an ultrametric norm on the $k$-vector space $H^0(D)$.
\end{exem}

Assume that $(E,\norm{\ndot})$ is a non-zero finite-dimensional ultrametrically normed vector space over $k$. We introduce a Borel probability measure $\mathbb P_{(E,\norm{\ndot})}$ on $\mathbb R$ such that, for any $t\in\mathbb R$,
\[\mathbb P_{(E,\norm{\ndot})}(\mathopen{]}{t},{+\infty}\mathclose{[})=\frac{\dim_k(\{s\in E\,:\,\norm{s}<\mathrm{e}^{-t}\})}{\dim_k(E)}.\]
Then, for any random variable $Z$ which follows $\mathbb P_{(E,\norm{\ndot})}$ as its probability law, one has
\begin{equation}\label{Equ: slope as expectation}\frac{\widehat{\deg}(E,\norm{\ndot})}{\dim_k(E)}=\mathbb E[Z]=\int_{\mathbb R}t\,\mathbb P_{(E,\norm{\ndot})}(\mathrm{d}t)\end{equation}
and
\begin{equation}\label{Equ: positive slope as expectation}
\frac{\widehat{\deg}_+(E,\norm{\ndot})}{\dim_k(E)}=\mathbb E[\max(Z,0)]=\int_{0}^{+\infty}t\,\mathbb P_{(E,\norm{\ndot})}(\mathrm{d}t).
\end{equation}

\subsection{Essential infima} Let $(D,g)$ be a metrised $\mathbb R$-divisor on $X$ such that $\Gamma(D)_{\mathbb R}^{\times}$ is not empty. We define 
\[\lambda_{\operatorname{ess}}(D,g):=\sup_{\phi\in\Gamma(D)_{\mathbb R}^{\times}}\inf_{\xi\in X^{\mathrm{an}}}(g_{(\phi)}+g)(\xi),\]
called the \emph{essential infimum} of the metrised $\mathbb R$-divisor $(D,g)$. By \eqref{Equ: norm g}, we can also write $\lambda_{\operatorname{ess}}(D,g)$ as
\[\sup_{\phi\in\Gamma(D)_{\mathbb R}^{\times}}\Big(-\ln\norm{\phi}_g\Big).\]

\begin{prop}\label{Pro: superadditivity}
Let $(D_1,g_1)$ and $(D_2,g_2)$ be metrised $\mathbb R$-divisors such that $\Gamma(D_1)_{\mathbb R}^{\times}$ and $\Gamma(D_2)_{\mathbb R}^{\times}$ are non-empty. Then one has
\begin{equation}\label{Equ: lambda ess superadditive}\lambda_{\mathrm{ess}}(D_1+D_2,g_1+g_2)\geqslant\lambda_{\mathrm{ess}}(D_1,g_1)+\lambda_{\mathrm{ess}}(D_2,g_2).\end{equation}
\end{prop}
\begin{proof}
Let $\phi_1$ and $\phi_2$ be elements of $\Gamma(D_1)_{\mathbb R}^{\times}$ respectively. One has $\phi_1\phi_2\in\Gamma(D_1+D_2)_{\mathbb R}^{\times}$. Moreover, 
\[g_{(\phi_1\phi_2)}=g_{(\phi_1)}+g_{(\phi_2)}.\]
Therefore 
\[g_{(\phi_1\phi_2)}+(g_1+g_2)=(g_{(\phi_1)}+g_1)+(g_{\phi_2}+g_2),\]
which leads to
\[\begin{split}&\quad\;\Big(\inf_{\xi\in X^{\mathrm{an}}}\big(g_{(\phi_1)}+g_1\big)(\xi)\Big)+\Big(\inf_{\xi\in X^{\mathrm{an}}}\big(g_{(\phi_2)}+g_2\big)(\xi)\Big)\\&\leqslant \inf_{\xi\in X^{\mathrm{an}}}\big(g_{(\phi_1\phi_2)}+(g_1+g_2)\big)(\xi)\\
&\leqslant\lambda_{\mathrm{ess}}(D_1+D_2,g_1+g_2).\end{split}\]
Taking the supremum with respect to $\phi_1\in\Gamma(D_1)_{\mathbb R}^{\times}$ and $\phi_2\in\Gamma(D_2)_{\mathbb R}^{\times}$, we obtain the inequality \eqref{Equ: lambda ess superadditive}.
\end{proof}


\begin{rema}In the literature, the essential infimum of  height function is studied in the number field setting. We can consider its analogue in the setting of Arakelov geometry over a trivially valued field. For any closed point $x$ of $X$, we define the height of $x$ with respect to $(D,g)$ as 
\[h_{(D,g)}(x):=\varphi_g(x_0),\]
where $\varphi_g=g-g_{\mathrm{can}}$ is the bounded Green function associated with $g$ (see Definition \ref{Def: canonical divisor}), and $x_0$ denotes the point of $X^{\mathrm{an}}$ corresponding to the closed point $x$ equipped with the trivial absolute value on its residue field. In particular, for any element $x\in X^{(1)}$ outside of the support of $D$, one has \[h_{(D,g)}(x)=g(x_0).\] Then the \emph{essential infimum} of the height function $h_{(D,g)}$ is defined as
\[\mu_{\mathrm{ess}}(D,g):=\sup_{Z\subsetneq X}\inf_{x\in X^{(1)}\setminus Z}h_{(D,g)}(x),\]
where $Z$ runs over the set of closed subschemes of $X$ which are different from $X$ (namely a finite subset of $X^{(1)}$). If $\Gamma(D)_{\mathbb R}^{\times}$ is not empty, one has
\[\lambda_{\mathrm{ess}}(D,g)\leqslant\sup_{\phi\in\Gamma(D)_{\mathbb R}^{\times}}\inf_{x\in X^{(1)}}(g_{(\phi)}+g)(x_0).\] For each $\phi\in\Gamma(D)^{\times}_{\mathbb R}$, one has
\[\inf_{x\in X^{(1)}}(g_{(\phi)}+g(x_0))\leqslant\inf_{x\in X^{(1)}\setminus(\mathrm{Supp}(D)\cup\mathrm{Supp}((\phi)))}g(x_0),\]
which is clearly bounded from above by $\mu_{\mathrm{ess}}(D,g)$. Therefore, one has 
\begin{equation}\label{Equ: lambda r ess bounded}\lambda_{\mathrm{ess}}(D,g)\leqslant\mu_{\mathrm{ess}}(D,g).\end{equation}
\end{rema}
The following proposition implies that $\lambda_{\mathrm{ess}}(D,g)$ is actually finite.

\begin{prop}\phantomsection\label{Pro: essential minimu bounded}
Let $(D,g)$ be a metrised $\mathbb R$-divisor on $X$. One has $\mu_{\mathrm{ess}}(D,g)=g(\eta_0)$, where $\eta_0$ denotes the point of $X^{\mathrm{an}}$ corresponding to the generic point of $X$ equipped with the trivial absolute value on its residue field.
\end{prop}
\begin{proof}
Let $\alpha$ be a real number that is $>g(\eta_0)$. The set
\[\{\xi\in X^{\mathrm{an}}\,:\,g(\xi)<\alpha\}\]
is an open subset of $X^{\mathrm{an}}$ containing $\eta_0$ and hence there exists a finite subset $S$ of $X^{(1)}$ such that 
$g(x_0)<\alpha$ for any $x\in X^{(1)}\setminus S$. Therefore we obtain $\mu_{\mathrm{ess}}(D,g)\leqslant\alpha$. Since $\alpha>g(\eta_0)$ is arbitrary, we get $\mu_{\mathrm{ess}}(D,g)\leqslant g(\eta_0)$.

Conversely, if $\beta$ is a real number such that $\beta<g(\eta_0)$, then
\[\{\xi\in X^{\mathrm{an}}\,:\,g(\xi)>\beta\}\]
is an open subset of $X^{\mathrm{an}}$ containing $\eta_0$. Hence there exists a finite subset $S'$ of $X^{(1)}$ such that $g(x_0)>\beta$ for any $x\in X^{(1)}\setminus S'$. Hence $\mu_{\mathrm{ess}}(D,g)\geqslant\beta$. Since $\beta<g(\eta_0)$ is arbitrary, we obtain $\mu_{\mathrm{ess}}(D,g)\geqslant g(\eta_0)$.
\end{proof}

\begin{lemm}\label{Lem: linear independence}
Let $r\in\mathbb N_{\geqslant 1}$ and $s_1,\ldots,s_r$ be elements of $\mathrm{Rat}(X)_{\mathbb Q}^{\times}$ and $a_1,\ldots,a_r$ be real numbers which are linearly independent over $\mathbb Q$. Let $s:=s_1^{a_1}\cdots s_r^{a_r}\in\mathrm{Rat}(X)_{\mathbb R}^{\times}$. Then for any $i\in\{1,\ldots,r\}$ one has $\operatorname{Supp}((s_i))\subset\operatorname{Supp}((s))$.
\end{lemm}
\begin{proof}
Let $x$ be a closed point of $X$ which does not lie in the support of $(s)$. One has
\[\sum_{i=1}^r\operatorname{ord}_x(s_i)a_i=0\]
and hence $\operatorname{ord}_x(s_1)=\ldots=\operatorname{ord}_x(s_r)=0$ since $a_1,\ldots,a_r$ are linearly independent over $\mathbb Q$.
\end{proof}

\begin{lemm}\phantomsection\label{Lem: approximation by rational solutions}
Let $n$ and $r$ be two positive integers, $\ell_1,\ldots,\ell_n$ be linear forms on $\mathbb R^r$ of the form
\[\ell_j(\boldsymbol{y})=b_{j,1}y_1+\cdots+b_{j,r}y_r,\text{ where $(b_{j,1},\ldots,b_{j,r})\in\mathbb Q^r$}\] and $q_1,\ldots,q_n$ be non-negative real numbers. Let $\boldsymbol{a}=(a_1,\ldots,a_r)$ be an element of $\mathbb R_{>0}^{r}$ which forms a linearly independent family over $\mathbb Q$, and such that $\ell_j(\boldsymbol{a})+q_j\geqslant 0$ for any $j\in\{1,\ldots,n\}$. Then, for any $\varepsilon>0$, there exists a sequence 
\[\boldsymbol{\delta}^{(m)}=(\delta^{(m)}_1,\ldots,\delta^{(m)}_r),\quad m\in\mathbb N\]
in $\mathbb R_{>0}^r$, which converges to $(0,\ldots,0)$ and verifies the following conditions:
\begin{enumerate}[label=\rm(\arabic*)]
\item for any $j\in\{1,\ldots,n\}$, one has $\ell_j(\boldsymbol{\delta}^{(m)})+\varepsilon q_j\geqslant 0$,
\item for any $m\in\mathbb N$ and any $i\in\{1,\ldots,r\}$, one has $\delta_i^{(m)}+a_i\in\mathbb Q$.
\end{enumerate}
\end{lemm}
\begin{proof}Without loss of generality, we may assume that $q_1=\cdots=q_d=0$ and $\min\{q_{d+1},\ldots,q_n\}>0$. 
Since $a_1,\ldots,a_r$ are linearly independent over $\mathbb Q$, for $j\in\{1,\ldots,d\}$, one has $\ell_j(\boldsymbol{a})>0$. Hence there exists an open convex cone $C$ in $\mathbb R_{>0}^r$ which contains $\boldsymbol{a}$, such that $\ell_j(\boldsymbol{y})>0$ for any $\boldsymbol{y}\in C$ and $j\in\{1,\ldots,d\}$. Moreover, if we denote by $\norm{\ndot}_{\sup}$ the norm on $\mathbb R^r$ (where $\mathbb R$ is equipped with its usual absolute value) defined as
\[\norm{(y_1,\ldots,y_r)}_{\sup}:=\max\{|y_1|,\ldots,|y_r|\},\] then there exists $\lambda>0$ such that, for any $\boldsymbol{z}\in C$ such that $\norm{z}_{\sup}<\lambda$ and any $j\in\{d+1,\ldots,n\}$, one has $\ell_j(\boldsymbol{z})+\varepsilon q_j\geqslant 0$. Let \[C_\lambda=\{\boldsymbol{y}\in C\,:\,\norm{\boldsymbol{y}}_{\sup}<\lambda\}.\]
It is a convex open subset of $\mathbb R^r$. For any $\boldsymbol{y}\in C_\lambda$ and any $j\in\{1,\ldots,n\}$, one has
\[\ell_j(\boldsymbol{y})+\varepsilon q_j\geqslant 0.\] 
Since $C_\lambda$ is open and convex, also is its translation by $-\boldsymbol{a}$. Note that the set of rational points in a convex open subset of $\mathbb R^r$ is dense in the convex open subset. Therefore, the set of all points $\boldsymbol{\delta}\in C_\lambda$ such that $\boldsymbol{\delta}+\boldsymbol{a}\in\mathbb Q^r$ is dense in $C_\lambda$. Since  $(0,\ldots,0)$ lies on the boundary of $C_\lambda$, it can be approximated by a sequence $(\boldsymbol{\delta}^{(m)})_{m\in\mathbb N}$ of elements in $C_\lambda$ such that $\boldsymbol{\delta}^{(m)}+\boldsymbol{a}\in\mathbb Q^r$ for any $m\in\mathbb N$. 
\end{proof}

\begin{rema}\label{Rem: suite d'approximation}
We keep the notation and hypotheses of Lemma \ref{Lem: approximation by rational solutions}. For any $m\in\mathbb N$, and $j\in\{1,\ldots,n\}$ one has
\[\ell_j(\boldsymbol{a}+\boldsymbol{\delta}^{(m)})+(1+\varepsilon)q_j\geqslant 0,\]
or equivalently,
\[\ell_j\Big(\frac{1}{1+\varepsilon}(\boldsymbol{a}+\boldsymbol{\delta}^{(m)})\Big)+q_j\geqslant 0.\]
Therefore, one can find a sequence $(\boldsymbol{a}^{(p)})_{p\in\mathbb N}$ of elements in $\mathbb Q^r$ which converges to $\boldsymbol{a}$ and    such that 
\[\ell_j(\boldsymbol{a}^{(p)})+q_j\geqslant 0\]
hods for any $j\in\{1,\ldots,n\}$ and any $p\in\mathbb N$.
\end{rema}

\begin{prop}\label{Pro:lambda ess sur Q}
Let $(D,g)$ be an arithmetic $\mathbb R$-divisor on $X$ such that $\Gamma(D)_{\mathbb Q}^{\times}\neq\emptyset$. One has \begin{equation}\label{Equ: interpretation of lambda ess}\begin{split}\lambda_{\mathrm{ess}}(D,g)&=\sup_{\phi\in\Gamma(D)_{\mathbb Q}^{\times}}\inf_{\xi\in X^{\mathrm{an}}}(g_{(\phi)}+g)(\xi)=\sup_{\phi\in\Gamma(D)_{\mathbb Q}^{\times}}\Big(-\ln\norm{\phi}_g\Big)\\
&=\sup_{n\in\mathbb N,\,n\geqslant 1}\frac 1n\sup_{s\in H^0(nD)\setminus\{0\}}\Big(-\ln\norm{s}_{ng}\Big).
\end{split}
\end{equation}
\end{prop}
\begin{proof}
By definition one has
\[\Gamma(D)_{\mathbb Q}^{\times}=\bigcup_{n\in\mathbb N,\,n\geqslant 1}\{s^{\frac 1n}\,:\,s\in H^0(nD)\setminus\{0\}\}.\]
Moreover, for $\phi\in\Gamma(D)_{\mathbb Q}^{\times}$, one has
\[\inf_{\xi\in X^{\mathrm{an}}}(g_{(\phi)}+g)(\xi)=-\ln\norm{\phi}_g.\]
Therefore the second and third equalities of \eqref{Equ: interpretation of lambda ess} hold. 
To show the first equality, we denote temporarily by $\lambda_{\mathbb Q,\mathrm{ess}}(D,g)$ the second term of \eqref{Equ: interpretation of lambda ess}.

Let $a$ be an arbitrary positive rational number. The correspondance $\Gamma(D)_{\mathbb Q}^{\times}\rightarrow\Gamma(aD)_{\mathbb Q}^{\times}$ given by $\phi\mapsto \phi^a$ is a bijection. Moreover, for $\phi\in\Gamma(D)_{\mathbb Q}^{\times}$ one has $\norm{\phi^a}_{ag}=\norm{\phi}_g^a$. Hence the equality \begin{equation}\label{Equ: linearity of lambda ess}\lambda_{\mathbb Q,\mathrm{ess}}(aD,ag)=a\lambda_{\mathbb Q,\mathrm{ess}}(D,g)\end{equation} holds.

By our assumption, we can choose $\phi\in\Gamma(D)_{\mathbb Q}^{\times}$. For $\mathbb K\in\{\mathbb Q,\mathbb R\}$, the map
\[\alpha_{\psi}:\Gamma(D)^{\times}_{\mathbb K}\longrightarrow\Gamma(D+(\psi))_{\mathbb K}^{\times},\quad \phi\longmapsto\phi\psi^{-1}\]
is a bijection. Moreover, for any $\phi\in\Gamma(D)_{\mathbb K}^{\times}$, 
\[\norm{\phi}_g=\norm{\alpha_{\psi}(\phi)}_{g+g_{(\psi)}}.\]
Hence one has
\begin{gather}\label{Equ: invariance under linear equivalence}\lambda_{\mathbb Q,\mathrm{ess}}(D,g)=\lambda_{\mathbb Q,\mathrm{ess}}(D+(\psi),g+g_{(\psi)}),\\\lambda_{\mathrm{ess}}(D,g)=\lambda_{\mathrm{ess}}(D+(\psi),g+g_{(\psi)}).\end{gather} 
Furthermore, for any $c\in\mathbb R$, one has
\begin{gather}\lambda_{\mathbb Q,\mathrm{ess}}(D,g+c)=\lambda_{\mathbb Q,\mathrm{ess}}(D,g)+c,\label{Equ: add a constant}\\\label{Equ: add a constant2}\lambda_{\mathrm{ess}}(D,g+c)=\lambda_{\mathrm{ess}}(D,g)+c.\end{gather}
Therefore, to prove the proposition, we may assume without loss of generality that $D$ is effective and $\varphi_g\geqslant 0$.

By definition one has $\lambda_{\mathbb Q,\mathrm{ess}}(D,g)\leqslant\lambda_{\mathrm{ess}}(D,g)$. To show the converse inequality, it suffices to prove that, for any $s\in\Gamma(D)_{\mathbb R}^{\times}$, one has 
\[-\ln\norm{s}_{g}\leqslant\lambda_{\mathbb Q,\mathrm{ess}}(D,g).\]
We choose $s_1,\ldots,s_r$ in $\mathrm{Rat}(X)_{\mathbb Q}^{\times}$ and $a_1,\ldots,a_r$ in $\mathbb R_{>0}$ such that $a_1,\ldots,a_r$ are linearly independent over $\mathbb Q$ and that $s=s_1^{a_1}\cdots s_r^{a_r}$. By Lemma \ref{Lem: linear independence}, for any $i\in\{1,\ldots,r\}$, the support of $(s_i)$ is contained in that of $(s)$. Assume that $\operatorname{Supp}((s))=\{x_1,\ldots,x_n\}$. Since $s\in\Gamma(D)_{\mathbb R}^{\times}$, for $j\in\{1,\ldots,n\}$, one has
\begin{equation}\label{Equ: sum of ai ord xi}a_1\operatorname{ord}_{x_j}(s_1)+\cdots+a_r\operatorname{ord}_{x_j}(s_r)+\operatorname{ord}_{x_j}(D)\geqslant 0.\end{equation}
By Lemma \ref{Lem: approximation by rational solutions}, for any rational number $\varepsilon>0$, there exists a sequence 
\[(\delta_1^{(m)},\ldots,\delta_{r}^{(m)}),\quad m\in\mathbb N\]
in $\mathbb R^r$, which converges to $(0,\ldots,0)$, and such that, 
\begin{enumerate}[label=\rm(\arabic*)]
\item for any $j\in\{1,\ldots,n\}$ and any $m\in\mathbb N$, one has
\[\delta_1^{(m)}\operatorname{ord}_{x_j}(s_1)+\cdots+\delta_r^{(m)}\operatorname{ord}_{x_j}(s_r)+\varepsilon\operatorname{ord}_{x_j}(D)\geqslant 0.\]
\item for any $i\in\{1,\ldots,r\}$ and any $m\in\mathbb N$, $\delta_i^{(m)}+a_i\in\mathbb Q$.
\end{enumerate}
For any $m\in\mathbb N$, let
\[s^{(m)}=s_1^{\delta_1^{(m)}}\!\!\cdots\; s_r^{\delta_r^{(m)}}\in\Gamma(\varepsilon D)_{\mathbb R}^{\times}.\]
The conditions (1) and (2) above imply that $s\cdot s^{(m)}\in\Gamma((1+\varepsilon)D)_{\mathbb Q}^{\times}$. 
Hence one has
\[\inf_{\xi\in X^{\mathrm{an}}}\big((1+\varepsilon)g+g_{(s\cdot s^{(m)})}\big)(\xi)\leqslant\lambda_{\mathbb Q,\mathrm{ess}}((1+\varepsilon)D,(1+\varepsilon)g).\]
Since $D$ is effective and $\varphi_g\geqslant 0$ by $s^{(m)}\in\Gamma(\varepsilon D)_{\mathbb R}^{\times }$, one has
\[\varepsilon g+g_{(s^{(m)})}\geqslant\varepsilon\varphi_g\geqslant 0.\]
Therefore we obtain
\[-\ln\norm{s}_g=\inf_{\xi\in X^{\mathrm{an}}}(g+g_{(s)})(\xi)\leqslant\lambda_{\mathbb Q,\mathrm{ess}}((1+\varepsilon)D,(1+\varepsilon)g)=(1+\varepsilon)\lambda_{\mathbb Q,\mathrm{ess}}(D,g),\]
where the last equality comes from \eqref{Equ: linearity of lambda ess}. Taking the limit when $\varepsilon\in\mathbb Q_{>0}$ tends to $0$, we obtain the desired inequality.
\end{proof}

\subsection{$\chi$-volume}
Let $(D,g)$ be a metrised $\mathbb R$-divisor on $X$.
We define the \emph{$\chi$-volume} of $(D,g)$ as
\[\widehat{\mathrm{vol}}_\chi(D,g):=\limsup_{n\rightarrow+\infty}\frac{\widehat{\deg}(H^0(nD),\norm{\ndot}_{ng})}{n^2/2}.\]
This invariant is similar to the $\chi$-volume function in the number field setting introduced in \cite{MR2425137}. Note that, if $\deg(D)<0$, then $H^0(D)=\{0\}$, so that $H^0(nD) = \{ 0 \}$ for all $n \in \mathbb{Z}_{>0}$.
Indeed, if $f \in H^0(D) \setminus \{ 0 \}$,
then $0 \leqslant \deg(D + (f)) = \deg(D) < 0$, which is a contradiction.
Hence $\widehat{\operatorname{vol}}_{\chi}(D,g)=0$.

\begin{prop}\label{Pro:formula:avol:g:g:prime}
Let $D$ be an $\mathbb{R}$-divisor on $X$, and $g$ and $g'$ be Green functions of $D$. If $g \leqslant g'$, then $\widehat{\mathrm{vol}}_{\chi}(D, g) \leqslant \widehat{\mathrm{vol}}_{\chi}(D, g')$.
\end{prop}

\begin{proof}
Note that $\|\ndot\|_{ng} \geqslant \|\ndot\|_{ng'}$ on $H^0(X, nD)$, so that one can see that $\|\ndot\|_{ng, \det} \geqslant \|\ndot\|_{ng', \det}$
on $\det H^0(X, nD)$. Therefore we obtain 
\[
\widehat{\deg}(H^0(X, nD), \|\ndot\|_{ng}) \leqslant \widehat{\deg}(H^0(X, nD), \|\ndot\|_{ng'})
\]
for all $n \geqslant 1$. Thus the assertion follows.
\end{proof}

\begin{prop}\label{Pro:volume chi translation}
Let $(D,g)$ be a metrised $\mathbb R$-divisor such that $\deg(D)\geqslant 0$. For any $c\in\mathbb R$, one has
\begin{equation}\label{Equ: volume chi translation}\widehat{\operatorname{vol}}_{\chi}(D,g+c)=2c\deg(D)+\widehat{\operatorname{vol}}_{\chi}(D,g).\end{equation}
\end{prop}
\begin{proof}
For any $n\in\mathbb N$, one has $\norm{\ndot}_{n(g+c)}=\norm{\ndot}_{ng+nc}=\mathrm{e}^{-nc}\norm{\ndot}_{ng}$. Therefore, one has
\[\widehat{\deg}(H^0(nD),\norm{\ndot}_{n(g+c)})=\widehat{\deg}(H^0(nD),\norm{\ndot}_{ng})+nc\dim_k(H^0(nD)).\]
Note that, by Proposition \ref{Pro: asymptotic RR}, 
\[\dim_k(H^0(nD))=\deg(D)n+o(n),\quad n\rightarrow+\infty.\]
Therefore, one has
\[\frac{\widehat{\deg}(H^0(nD),\norm{\ndot}_{n(g+c)})}{n^2/2}=\frac{\widehat{\deg}(H^0(nD),\norm{\ndot}_{ng})}{n^2/2}+2c\deg(D)+o(1),\quad n\rightarrow+\infty.\]
Taking the superior limit when $n
\rightarrow+\infty$, we obtain \eqref{Equ: volume chi translation}.

\end{proof}

\begin{defi}
Let $(D,g)$ be a metrised $\mathbb R$-divisor such that $\deg(D)>0$. We denote by ${\Gamma}(D,g)_{\mathbb R}^\times$ the set of $s\in\Gamma(D)_{\mathbb R}^{\times}$ such that $\norm{s}_g< 1$. Similarly, we denote by ${\Gamma}(D,g)_{\mathbb Q}^{\times}$ the set of $s\in\Gamma(D)_{\mathbb Q}^{\times}$ such that $\norm{s}_g< 1$. 

For any $t \in\mathbb R$ such that $t<\lambda_{\operatorname{ess}}(D,g)$, we let $D_{g,t}$ be the $\mathbb R$-divisor
\[\sup_{s\in {\Gamma}(D,\,g-t)_{\mathbb Q}^{\times}}(s^{-1}).\] 
For sufficiently negative number $t$ such that $\norm{s}_g<\mathrm{e}^{-t}$ for any $s\in\Gamma(D)_{\mathbb R}^{\times}$, one has \[{\Gamma}(D,g-t)_{\mathbb Q}^{\times}={\Gamma}(D)_{\mathbb Q}^{\times}\] and hence, by Proposition \ref{Pro: sup of s in Gamma D}, $D_{g,t}=D$. If $t\geqslant\lambda_{\mathrm{ess}}(D,g)$, by convention we let $D_{g,t}$ be the zero $\mathbb R$-divisor.
\end{defi}

\begin{prop}\phantomsection\label{Pro: lim of Vt is Dgt}
Let $(D,g)$ be a metrised $\mathbb R$-divisor such that $\deg(D)>0$, and $t\in\mathbb R$ such that $t<\lambda_{\mathrm{ess}}(D,g)$. Let \[V_\sbullet^t(D,g):=\bigoplus_{n\in\mathbb N}\{s\in H^0(nD)\,:\,\norm{s}_{ng}< \mathrm{e}^{-tn}\}T^n\subseteq K[T].\]
Then one has 
\begin{equation}\label{Equ: deg Dgt}\lim_{n\rightarrow+\infty}\frac{\dim_k(V_n^t(D,g))}{n}=\deg(D_{g,t})>0.\end{equation}
\end{prop}
\begin{proof}By Proposition \ref{Pro: convergence of dim},  it suffices to show that the graded linear series $V^t_{\sbullet}(D,g)$ is birational (see Definition \ref{Def: generic point of graded linear series}). As $\deg(D)>0$, there exists $m\in\mathbb N_{\geqslant 1}$ such that $k(H^0(mD))=\mathrm{Rat}(X)$ (see Example \ref{Exe: birational graded linear series} and Proposition \ref{Pro: corps engendre par un systeme lineare}). Note that the norm $\norm{\ndot}_{mg}$ is a bounded function on $H^0(mD)$. In fact, if $(s_i)_{i=1}^{r_m}$ is a basis of $H^0(kmD)$, as the norm $\norm{\ndot}_{mg}$ is ultrametric, for any $(\lambda_i)_{i=1}^{r_m}\in k^{r_m}$, one has
\[\norm{\lambda_1s_1+\cdots+\lambda_{r_m}s_{r_m}}_{mg}\leqslant\max_{i\in\{1,\ldots,r_m\}}\norm{s_i}_{mg}.\]
We choose $\varepsilon>0$ such that $t+\varepsilon<\lambda_{\mathrm{ess}}(D,g)$. By \eqref{Equ: interpretation of lambda ess} we obtain that there exist $n\in\mathbb N_{\geqslant 1}$ and  $s\in H^0(nD)$ such that $\|s\|_{ng}\leqslant \mathrm{e}^{-n(t+\varepsilon)}$. Let $d$ be a positive integer such that 
\[d>\frac{1}{n\varepsilon}\Big(tm+\max_{i\in\{1,\ldots,r_m\}}\ln\norm{s_i}_{mg}\Big).\]
Then, for any $s'\in H^0(mD)$, one has \[\norm{s^ds'}_{(dn+m)g}<\mathrm{e}^{-(dn+m)t},\]
which means that $s^ds'\in V_{dn+m}^t(D,g)$. Therefore we obtain $k(V_{dn+m}^t(D,g))=\mathrm{Rat}(X)$ since it contains $k(H^0(mD))$. The graded linear series $V_\sbullet^t(D,g)$ is thus birational and \eqref{Equ: deg Dgt} is proved.
\end{proof}

\begin{theo}\label{Thm: probabilistic interpretation}
Let $(D,g)$ be a metrised $\mathbb R$-divisor such that $\deg(D)>0$. Let $\mathbb P_{(D,g)}$ be the Borel probability measure on $\mathbb R$ such that 
\begin{equation}\label{Equ: proba measure of a metrised divisor}\mathbb P_{(D,g)}(\mathopen{]}t,+\infty\mathclose{[})=\deg(D_{g,t})\end{equation}
for $t<\lambda_{\mathrm{ess}}(D,g)$ and $\mathbb P_{(D,g)}(\mathopen{]}t,+\infty\mathclose{[})=0$ for $t\geqslant \lambda_{\mathrm{ess}}(D,g)$.
Then one has
\begin{equation}\label{Equ: vol chi as an integral}\frac{\widehat{\operatorname{vol}}_\chi(D,g)}{2\deg(D)}=\int_{\mathbb R}t\,\mathbb P_{(D,g)}(\mathrm{d}t).\end{equation}
\end{theo}
\begin{proof}
For any $n\in\mathbb N$, let $\mathbb P_n$ be the Borel probability measure on $\mathbb R$ such that 
\[\mathbb P_n(\mathopen{]}t,+\infty\mathclose{[})=\frac{\dim_k(V_n^t(D,g))}{\dim_k(H^0(nD))}\]
for $t<\lambda_{\mathrm{ess}}(D,g)$ and $\mathbb P_n(\mathopen{]}t,+\infty\mathclose{[})=0$ for $t\geqslant \lambda_{\mathrm{ess}}(D,g)$. By Propositions \ref{Pro: lim of Vt is Dgt} and \ref{Pro: asymptotic RR}, one has 
\[\forall\,t\in\mathbb R,\quad \lim_{n\rightarrow +\infty}P_n(\mathopen{]}t,+\infty\mathclose{[})=\mathbb P_{(D,g)}(\mathopen{]}t,+\infty\mathclose{[}).\]
Therefore the sequence of probability measures $(\mathbb P_n)_{n\in\mathbb N}$ converges weakly to $\mathbb P$. Moreover, if we write $g$ as $g_D+f$, where $f$ is a continuous function on $X^{\mathrm{an}}$, then the supports of the probability measures $P_n$ are contained in $[\inf f,g(\eta)]$. Therefore one has
\[\lim_{n\rightarrow+\infty}\int_{\mathbb R}t\,\mathbb P_n(\mathrm{d}t)=\int_{\mathbb R}t\,\mathbb P_{(D,g)}(\mathrm{d}t).\]
By \eqref{Equ: slope as expectation}, for any $n\in\mathbb N_{\geqslant 1}$ such that $H^0(nD)\neq\{0\}$, one has
\[\int_{\mathbb R}t\,\mathbb P_n(\mathrm{d}t)=\frac{\widehat{\deg}(H^0(nD),\norm{\ndot}_{ng})}{\dim_k(H^0(nD))}.\]
Therefore we obtain \eqref{Equ: vol chi as an integral}.
\end{proof}

\begin{rema}
Theorem \ref{Thm: probabilistic interpretation} and Proposition \ref{Pro: asymptotic RR} show that the sequence defining the $\chi$-volume function has a limit. More precisely, if $(D,g)$ is a metrised $\mathbb R$-divisor such that $\deg(D)>0$, then one has 
\[\widehat{\mathrm{vol}}_{\chi}(D,g)=\lim_{n\rightarrow+\infty}\frac{\widehat{\deg}(H^0(nD),\norm{\ndot}_{ng})}{n^2/2}.\]
\end{rema}

\begin{defi}\label{Def: concave transform}
Let $(D,g)$ be a metrised $\mathbb R$-Cartier divisor on $X$ such that $\deg(D)>0$. We denote by $G_{(D,g)}:[0,\deg(D)]\rightarrow\mathbb R$ the function sending $u\in[0,\deg(D)]$ to \[\sup\{t\in\mathbb R_{<g(\eta)}\,:\,\deg(D_{g,t})>u\}.\]  For any $t<g(\eta_0)$ one has
\[\mathbb P_{(D,g)}(\mathopen{]}G_{(D,g)}(\lambda),+\infty\mathclose{[})=\frac{\deg(D_{g,G_{(D,g)}(\lambda)})}{\deg(D)},\]
namely, the probability measure $\mathbb P_{(D,g)}$ coincides with the direct image of the uniform distribution on $[0,\deg(D)]$ by the map $G_{(D,g)}$.
\end{defi}

\begin{prop}\phantomsection\label{Pro: Dgt as R}
Let $(D,g)$ be a metrised $\mathbb R$-divisor such that $\deg(D)>0$. For any $t\in\mathbb R$ such that $t<\lambda_{\mathrm{ess}}(D,g)$, one has
\begin{equation}D_{g,t}=\sup_{s\in{\Gamma}(D,\,g-t)_{\mathbb R}^{\times}}(s^{-1}).
\end{equation}
\end{prop}
\begin{proof}Since $\deg(D)>0$, the set $\Gamma(D)_{\mathbb Q}^{\times}$ is not empty.
Let $\phi\in\Gamma(D)^{\times}_{\mathbb Q}$ and $(D',g')=(D,g)+\widehat{(\phi)}$. By \eqref{Equ: invariance under linear equivalence}, one has $\lambda_{\mathrm{ess}}(D,g)=\lambda_{\mathrm{ess}}(D',g')$. Moreover, the correspondance $s\mapsto s\cdot\phi^{-1}$ defines a bijection from $\Gamma(D,g-t)_{\mathbb K}^\times$ to $\Gamma(D',g'-t)_{\mathbb K}^{\times}$ for $\mathbb K=\mathbb Q$ or $\mathbb R$. Therefore, without loss of generality, we may assume that $D$ is effective. Moreover, by replacing $g$ by $g-t$ and $t$ by $0$ we may assume that $\lambda_{\mathrm{ess}}(D,g)>0$ and $t=0$.

It suffices to check that 
$D_{g,0}\geqslant (s^{-1})$ for any $s\in{\Gamma}(D,g)_{\mathbb R}^{\times}$. We write $s$ as $s_1^{a_1}\cdots s_r^{a_r}$, where $s_1,\ldots,s_r$ are elements of $\operatorname{Rat}(X)_{\mathbb Q}^{\times}$, and $a_1,\ldots,a_r$ are positive real numbers which are linearly independent over $\mathbb Q$.  Assume that $\operatorname{Supp}((s))=\{x_1,\ldots,x_n\}$. By Lemma \ref{Lem: linear independence}, for any $i\in\{1,\ldots,r\}$, the support of $(s_i)$ is contained in $\{x_1,\ldots,x_n\}$. For any $j\in\{1,\ldots,n\}$, one has
\[\operatorname{ord}_{x_j}(D)+\sum_{i=1}^r\operatorname{ord}_{x_j}(s_i)a_i\geqslant 0.\]
By Lemma \ref{Lem: approximation by rational solutions} and Remark \ref{Rem: suite d'approximation}, there  exists a sequence of vectors \[\boldsymbol{a}^{(m)}=(a_1^{(m)},\ldots,a_r^{(m)}),\quad m\in\mathbb N\]
in $\mathbb Q^r$ such that 
\begin{equation}\label{Equ: sm is ok}\operatorname{ord}_{x_j}(D)+\sum_{i=1}^r\operatorname{ord}_{x_j}(s_i)a_i^{(m)}\geqslant 0\end{equation}
and \begin{equation}\label{Equ: limit of sm}\lim_{m\rightarrow+\infty}\boldsymbol{a}^{(m)}=(a_1,\ldots,a_r).\end{equation} For any $m\in\mathbb N$, let \[s^{(m)}=s_1^{a_1^{(m)}}\cdots s_r^{a_r^{(m)}}.\]
By \eqref{Equ: sm is ok} one has $s^{(m)}\in\Gamma(D)_{\mathbb Q}^\times$. Moreover, by \eqref{Equ: limit of sm} and the fact that $\norm{s}_g<1$, for sufficiently positive $m$, one has $\norm{s^{(m)}}_g<1$ and hence $D_{g,0}\geqslant ((s^{(m)})^{-1})$. By taking the limit when $m\rightarrow+\infty$, we obtain $D_{g,0}\geqslant(s^{-1})$.
\end{proof}

\begin{coro}\label{Cor: linearility of vol chi}
Let $(D,g)$ be a metrised $\mathbb R$-Cartier divisor such that $\deg(D)>0$. For any $a>0$ one has
\[\widehat{\deg}_\chi(a D,a g)=a^2\operatorname{\widehat{\deg}}_\chi(D,g).\]
\end{coro}
\begin{proof}
By Proposition \ref{Pro: Dgt as R} one has
\[(a D)_{a g,a t}=a D_{g,t}.\]
By \eqref{Equ: vol chi as an integral} one has
\[\begin{split}\widehat{\operatorname{vol}}_\chi(a D,a g)&=2\int_{- M}^{ \lambda_{\mathrm{ess}}(D,g)}\deg((a D)_{a g,a t})\,\mathrm{d}a t+2a M\deg(D)\\
&=2a^2\int_{-M}^{\lambda_{\mathrm{ess}}(D,g)}\deg(D_{g,t})\,\mathrm{d}t+2a^2M\deg(D)=a^2\operatorname{\widehat{\deg}}_\chi(D,g).\end{split}\]
\end{proof}

\begin{theo}\label{thm:super:additive:vol:chi:deg}
Let $(D_1,g_1)$ and $(D_2,g_2)$ be metrised $\mathbb R$-Cartier divisors such that $\deg(D_1)>0$ and $\deg(D_2)>0$. One has
\[\frac{\widehat{\operatorname{vol}}_\chi(D_1+D_2,g_1+g_2)}{\deg(D_1)+\deg(D_2)}\geqslant\frac{\widehat{\operatorname{vol}}_\chi(D_1,g_1)}{\deg(D_1)}+\frac{\widehat{\operatorname{vol}}_\chi(D_2,g_2)}{\deg(D_2)}\]
\end{theo}
\begin{proof}
Let $t_1$ and $t_2$ be real numbers such that $t_1<\lambda_{\mathrm{ess}}(D_1,g_2)$ and $t_2<\lambda_{\mathrm{ess}}(D_2,g_2)$. For all $s_1\in{\Gamma}(D_1,g_1-t_1)_{\mathbb R}^{\times}$ and $s_2\in{\Gamma}(D_2,g_2-t_2)_{\mathbb R}^{\times}$ one has \[s_1s_2\in {\Gamma}(D_1+D_2,g_1+g_2-t_1-t_2)_{\mathbb R}^{\times}.\] Therefore, by Proposition \ref{Pro: Dgt as R} one has \begin{equation}\label{Equ: superadditives}(D_1+D_2)_{g_1+g_2,t_1+t_2}\geqslant (D_1)_{g_1,t_1}+(D_2)_{g_2,t_2}.\end{equation}
As a consequence, for any $(\lambda_1,\lambda_2)\in[0,\deg(D_1)]\times[0,\deg(D_2)]$, one has \begin{equation}\label{Equ: super additivity of G}G_{(D_1+D_2,g_1+g_2)}(\lambda_1+\lambda_2)\geqslant G_{(D_1,g_1)}(\lambda_1)+G_{(D_2,g_2)}(\lambda_2).\end{equation} Let $U$ be a random variable which follows the uniform distribution on $[0,\deg(D_1)]$. Let $f:[0,\deg(D_1)]\rightarrow[0,\deg(D_2)]$ be the linear map sending $u$ to $u\deg(D_2)/\deg(D_1)$. By Theorem \ref{Thm: probabilistic interpretation} one has
\[\frac{\widehat{\operatorname{vol}}_\chi(D_1+D_2,g_1+g_2)}{2(\deg(D_1)+\deg(D_2))}=\mathbb E[G_{(D_1+D_2,g_1+g_2)}(U+f(U))]\]
since $U+f(U)$ follows the uniform distribution on $[0,\deg(D_1)+\deg(D_2)]$. By \eqref{Equ: super additivity of G} we obtain
\[\begin{split}\frac{\widehat{\operatorname{vol}}_\chi(D_1+D_2,g_1+g_2)}{2(\deg(D_1)+\deg(D_2))}&\geqslant\mathbb E[G_{(D_1,g_1)}(U)]+\mathbb E[G_{(D_2,g_2)}(f(U))]\\
&\geqslant\frac{\widehat{\operatorname{vol}}_\chi(D_1,g_1)}{2\deg(D_1)}+\frac{\widehat{\operatorname{vol}}_\chi(D_2,g_2)}{2\deg(D_2)}.
\end{split}\]
The theorem is thus proved.
\end{proof}

Finally let us consider other properties of $\widehat{\mathrm{vol}}_{\chi}(\ndot)$.

\begin{prop}\label{prop:formula:avol:g:g:prime}
Let $D$ be an $\mathbb{R}$-divisor on $X$ such that $\deg(D)\geqslant 0$, and $g$ and $g'$ be Green functions of $D$. Then one has the following:
\begin{enumerate}[label=\rm(\arabic*)]
\item $2 \deg(D) \min\limits_{\xi \in X^{\mathrm{an}}} \{ \varphi_g(\xi) \}  \leqslant \widehat{\mathrm{vol}}_{\chi}(D, g) \leqslant 2 \deg(D) \max\limits_{\xi \in X^{\mathrm{an}}} \{ \varphi_g(\xi) \}$.

\item $| \widehat{\mathrm{vol}}_{\chi}(D, g) - \widehat{\mathrm{vol}}_{\chi}(D, g')| \leqslant 2 \| \varphi_g - \varphi_{g'} \|_{\sup} \deg(D)$.

\item If $\deg(D) = 0$, then $\widehat{\mathrm{vol}}_{\chi}(D, g) =0$.

\end{enumerate}
\end{prop}

\begin{proof}
(1) If we set $m = \min\limits_{\xi \in X^{\mathrm{an}}} \{ \varphi_g(\xi) \}$ and $M = \max\limits_{\xi \in X^{\mathrm{an}}} \{ \varphi_g(\xi) \}$,
then 
\[
g_D + m \leqslant g \leqslant g_D + M.
\] Note that $\widehat{\mathrm{vol}}_{\chi}(D, g_D) = 0$, so that the assertion follows from
Propositions \ref{Pro:formula:avol:g:g:prime} and \ref{Pro:volume chi translation}.

\medskip
(2) If we set $c = \| \varphi_g - \varphi_{g'} \|_{\sup}$, then $g - c \leqslant g' \leqslant g + c$, so that
(2) follows from Propositions \ref{Pro:formula:avol:g:g:prime} and \ref{Pro:volume chi translation}.

\medskip
(3) is a consequence of (1).

\end{proof}

\begin{prop}\label{Pro: continuity}
Let $V$ be a finite-dimensional vector subspace of $\widehat{\operatorname{Div}}_{\mathbb R}(X)$.
Then $\widehat{\mathrm{vol}}_{\chi}(\ndot)$ is continuous on $V$.
\end{prop}
\begin{proof}
We denote by $V_+$ the subset of $(D,g)$ such that $\deg(D)>0$.
The function $V_+\rightarrow \mathbb R$ given by
$(D,g) \mapsto \widehat{\operatorname{vol}}_{\chi}(D,g)/\deg(D)$
is concave by Corollary~\ref{Cor: linearility of vol chi} and Theorem~\ref{thm:super:additive:vol:chi:deg}, and hence it
is continuous on $V_+$.

We fix $(D, g) \in V$. If $\deg(D)<0$, then there exists a neighbourhood $U$ of $(D,g)$ in $V$ such that $\deg(D')<0$ for any $(D',g')\in U$. Hence $\widehat{\operatorname{vol}}_{\chi}(\ndot)$ vanishes on $U$. If $\deg(D) > 0$, then the above observation shows the continuity at $(D,g)$,
so that we may assume that
$\operatorname{deg}(D) = 0$. Then, by (3), 
$\widehat{\mathrm{vol}}_{\chi}(D, g) = 0$.
Therefore it is sufficient to show that
\[
\lim\limits_{(\varepsilon_{1,n}, \ldots, \varepsilon_{r,n})\to (0,\ldots,0)}\widehat{\mathrm{vol}}_{\chi}(\varepsilon_{1,n}(D_1, g_1) + \cdots + \varepsilon_{r,n}(D_r, g_r) + (D, g)) = 0,
\]
where $(D_1, g_1), \ldots, (D_r, g_r) \in V$.
By using (1), 
\begin{multline*}
| \widehat{\mathrm{vol}}_{\chi}(\varepsilon_{1,n}(D_1, g_1) + \cdots + \varepsilon_{r,n}(D_r, g_r) + (D, g))| \\
\leqslant 2 \| \varepsilon_{1,n}\varphi_{g_1} + \cdots + \varepsilon_{r,n}\varphi_{g_r} + \varphi_g \|_{\sup} \deg(\varepsilon_{1,n}D_1 + \cdots + \varepsilon_{r,n}D_r + D).
\end{multline*}
On the other hand, note that
 note that
\[
\begin{cases}
 \lim\limits_{(\varepsilon_{1,n}, \ldots, \varepsilon_{r,n})\to (0,\ldots,0)} \| \varepsilon_{1,n}\varphi_{g_1} + \cdots + \varepsilon_{r,n}\varphi_{g_r} + \varphi_g \|_{\sup} =  \| \varphi_g \|_{\sup},\\
 \lim\limits_{(\varepsilon_{1,n}, \ldots, \varepsilon_{r,n})\to (0,\ldots,0)} \deg(\varepsilon_{1,n}D_1 + \cdots + \varepsilon_{r,n}D_r + D) = \deg(D) = 0.
\end{cases}
\]
Thus the assertion follows.

\end{proof}

\subsection{Volume function}

Let $(D,g)$ be a metrised $\mathbb R$-divisor on $X$. We define the \emph{volume} of $(D,g)$ as
\[\widehat{\mathrm{vol}}(D,g):=\limsup_{n\rightarrow+\infty}\frac{\widehat{\deg}_+(nD,ng)}{n^2/2}.\]
Note that this function is analogous to the arithmetic volume function introduced in \cite{MR2496453}.

\begin{prop}\phantomsection\label{Pro: vol hat dg}
Let $(D,g)$ be a metrised $\mathbb R$-divisor such that $\deg(D)>0$. Let $\mathbb P_{(D,g)}$ be the Borel probability 
measure on $\mathbb R$ defined in Theorem \ref{Thm: probabilistic interpretation}. Then one has
\begin{gather}\frac{\widehat{\mathrm{vol}}(D,g)}{2\deg(D)}=\int_{\mathbb R}\max\{t,0\}\,\mathbb P_{(D,g)}(\mathrm{d}t),\\
\widehat{\mathrm{vol}}(D,g)=\int_0^{+\infty}\deg(D_{g,t})\,\mathrm{d}t.\end{gather}
\end{prop}
\begin{proof}
We keep the notation introduced in the proof of Theorem \ref{Thm: probabilistic interpretation}. By \eqref{Equ: positive slope as expectation}, for any $n\in\mathbb N_{\geqslant 1}$ one has
\[\frac{\widehat{\deg}_+(H^0(nD),\norm{\ndot}_{ng})}{\dim_k(H^0(nD))}=\int_{\mathbb R}\max\{t,0\}\,\mathbb P_n(\mathrm{d}t).\]
By passing to limit when $n\rightarrow+\infty$, we obtain the first equality. The second equality comes from the first one  and \eqref{Equ: proba measure of a metrised divisor} by integration by part.
\end{proof}

\section{Positivity}

The purpose of this section is to discuss several positivity conditions of metrised $\mathbb R$-divisors. We fix in this section a field $k$ equipped with the trivial absolute value $|\ndot|$ and a regular integral projective curve $X$ sur $\Spec k$.

\subsection{Bigness and pseudo-effectivity}
Let $(D,g)$ be a metrised $\mathbb R$-divisor on $X$. If $\widehat{\mathrm{vol}}(D,g)>0$, we say that $(D,g)$ is \emph{big}; if for any big metrised $\mathbb R$-divisor $(D_0,g_0)$ on $X$, the metrised $\mathbb R$-divisor $(D+D_0,g+g_0)$ is big, we say that $(D,g)$ is \emph{pseudo-effective}.

\begin{rema}
Let $(D,g)$ be a metrised $\mathbb R$-divisor. Let $n\in\mathbb N$, $n\geqslant 1$. If $H^0(nD)\neq\{0\}$, then $\Gamma(D)_{\mathbb Q}^{\times}$ is not empty. Moreover, for any non-zero element $s\in H^0(nD)$, one has
\[-\ln\norm{s}_g\leqslant n\lambda_{\mathrm{ess}}(D,g)\]
by \eqref{Equ: interpretation of lambda ess}, \eqref{Equ: lambda r ess bounded} and Proposition \ref{Pro: essential minimu bounded}. In particular, one has
\[\widehat{\deg}_+(H^0(nD),\norm{\ndot}_{ng})\leqslant n\max\{\lambda_{\mathrm{ess}}(D, g),0\}\dim_k(H^0(nD)).\]
Therefore, if $\widehat{\mathrm{vol}}(D,g)>0$, then one has $\deg(D)>0$ and $\lambda_{\mathrm{ess}}(D,g)>0$. Moreover, in the case where $(D,g)$ is big, one has
\begin{equation}\frac{\widehat{\mathrm{vol}}(D,g)}{2\deg(D)}\leqslant \lambda_{\mathrm{ess}}(D,g).\end{equation}
\end{rema}

\begin{prop}\phantomsection\label{Pro: criterion big}
Let $(D,g)$ be a metrised divisor  on $X$. The following assertions are equivalent.
\begin{enumerate}[label=\rm(\arabic*)] 
\item  $(D,g)$  is big.
\item  $\deg(D)>0$ and $\lambda_{\mathrm{ess}}(D,g)>0$
\item $\deg(D)>0$ and there exists $s\in\Gamma(D)_{\mathbb R}^{\times}$ such that $\norm{s}_g<1$.
\item $\deg(D)>0$ and there exists $s\in\Gamma(D)_{\mathbb Q}^{\times}$ such that $\norm{s}_g<1$.
\end{enumerate}
\end{prop}
\begin{proof}
``(1) $\Leftrightarrow$ (2)'' We have seen in the above Remark that, if $(D,g)$ is big, then $\deg(D)>0$ and $\lambda_{\mathrm{ess}}(D,g)>0$. The converse comes from the equality
\[\operatorname{\widehat{\mathrm{vol}}}(D,g)=\int_0^{+\infty}\deg(D_{g,t})\,\mathrm{d}t.\]
proved in Proposition \ref{Pro: vol hat dg}. Note that the function $t\mapsto\deg(D_{g,t})$ is decreasing. Moreover, by Proposition \ref{Pro: lim of Vt is Dgt}, one has $\deg(D_{g,t})>0$ once $t<\lambda_{\mathrm{ess}}(D,g)$. Therefore, if $\lambda_{\mathrm{ess}}(D,g)>0$, then $\widehat{\mathrm{vol}}(D,g)>0$. 

``(2)$\Leftrightarrow$(3)'' comes from the definition of $\lambda_{\mathrm{ess}}(D,g)$.

``(2)$\Leftrightarrow$(4)'' comes from Proposition \ref{Pro:lambda ess sur Q}.
\end{proof}

\begin{coro}
\begin{enumerate}[label=\rm(\arabic*)]
\item If $(D,g)$ is a big metrised $\mathbb R$-divisor on $X$, then, for any positive real number $\varepsilon$, the metrised $\mathbb R$-divisor $\varepsilon(D,g)=(\varepsilon D,\varepsilon g)$ is big.
\item If $(D_1,g_1)$ and $(D_2,g_2)$ are two metrised $\mathbb R$-divisor on $X$ which are big, then $(D_1+D_2,g_1+g_2)$ is also big.
\end{enumerate}
\end{coro}

\begin{proof}
The first assertion follows from Proposition \ref{Pro: criterion big} and the equalities $\deg(\varepsilon D)=\varepsilon\deg(D)$ and $\lambda_{\mathrm{ess}}(\varepsilon (D,g))=\varepsilon\lambda_{\mathrm{ess}}(D,g)$.

We then prove the second assertion. Since $(D_1,g_1)$ and $(D_2,g_2)$ are big, one has $\deg(D_1)>0$, $\deg(D_2)>0$, $\lambda_{\mathrm{ess}}(D_1,g_1)>0$, $\lambda_{\mathrm{ess}}(D_2,g_2)>0$. Therefore, $\deg(D_1+D_2)=\deg(D_1)+\deg(D_2)>0$. Moreover, by \eqref{Equ: lambda ess superadditive} one has 
\[\lambda_{\mathrm{ess}}(D_1+D_2,g_1+g_2)\geqslant\lambda_{\mathrm{ess}}(D_1,g_1)+\lambda_{\mathrm{ess}}(D_2,g_2)>0.\]
Therefore $(D_1+D_2,g_1+g_2)$ is big.

\end{proof}

\begin{coro}\label{Cor: pseudoeffective implies lambda ess positive}
Let $(D,g)$ be a metrised $\mathbb R$-divisor on $X$ such that $\deg(D)>0$. Then $(D,g)$ is pseudo-effective if and only if $\lambda_{\mathrm{ess}}(D,g)\geqslant 0$.
\end{coro}
\begin{proof}
Suppose that $(D,g)$ is pseudo-effective. Since $\deg(D)>0$, by \eqref{Equ: add a constant2} there exists $c>0$ such that $\lambda_{\mathrm{ess}}(D,g+c)>0$ (and thus $(D,g+c)$ is big by Proposition \ref{Pro: criterion big}). Hence for any $\varepsilon\in\mathopen{]}0,1\mathclose{[}$, 
\[(1-\varepsilon)(D,g)+\varepsilon(D,g+c)=(1-\varepsilon)\Big((D,g)+\frac{\varepsilon}{1-\varepsilon}(D,g+c)\Big)\]
is big. Therefore,
\[\lambda_{\mathrm{ess}}\big((1-\varepsilon)(D,g)+\varepsilon(D,g+c)\big)=\lambda_{\mathrm{ess}}(D,g+\varepsilon c)=\lambda_{\mathrm{ess}}(D,g)+\varepsilon c>0.\]
Since $\varepsilon\in\mathopen{]}0,1\mathclose{[}$ is arbitrary, we obtain $\lambda_{\mathrm{ess}}(D,g)\geqslant 0$.

In the following, we assume that $\lambda_{\mathrm{ess}}(D,g)\geqslant 0$ and we prove that $(D,g)$ is pseudo-effective. For any big metrised $\mathbb R$-divisor $(D_1,g_1)$ one has
\[\deg(D+D_1)=\deg(D)+\deg(D_1)>0\]
and, by \eqref{Equ: lambda ess superadditive},
\[\lambda_{\mathrm{ess}}(D+D_1,g+g_1)\geqslant\lambda_{\mathrm{ess}}(D,g)+\lambda_{\mathrm{ess}}(D_1,g_1)>0.\] 
Therefore $(D+D_1,g+g_1)$ is big.
\end{proof}

\begin{prop}\phantomsection\label{Pro: pseudo effective, mu inf positive}
Let $(D,g)$ be a metrised $\mathbb R$-divisor on $X$ which is pseudo-effective. Then one has $\deg(D)\geqslant 0$ and $g(\eta_0)\geqslant 0$.
\end{prop}
\begin{proof}
Let $(D_1,g_1)$ be a big metrised $\mathbb R$-divisor. For any $\varepsilon>0$, the metrised $\mathbb R$-divisor $(D+\varepsilon D_1,g+\varepsilon g_1)$ is big. Therefore, by Proposition \ref{Pro: criterion big}, one has
\[\deg(D+\varepsilon D_1)=\deg(D)+\varepsilon\deg(D_1)>0.\] 
Moreover, by Proposition \ref{Pro: criterion big}, the inequality \eqref{Equ: lambda r ess bounded} and Proposition \ref{Pro: essential minimu bounded}, one has
\[g(\eta_0)+\varepsilon g_1(\eta_0)\geqslant\lambda_{\mathrm{ess}}(D+\varepsilon D_1,g+\varepsilon g_1)>0.\]
Since $\varepsilon>0$ is arbitrary, we obtain $\deg(D)\geqslant 0$ and $g(\eta_0)\geqslant 0$.
\end{proof}


\subsection{Criteria of effectivity up to $\mathbb R$-linear equivalence}

Let $(D,g)$ be a metrised $\mathbb R$-divisor on $X$. We say that $(D,g)$ is \emph{effective} if $D$ is effective and $g$ is a non-negative function. We say that two metrised $\mathbb R$-divisor are \emph{$\mathbb R$-linear equivalent} if there exists an element $\varphi\in\mathrm{Rat}(X)_{\mathbb R}^{\times}$ such that \[(D_2,g_2)=(D_1,g_1)+\widehat{(\varphi)}.\]
By Proposition \ref{Pro: criterion big}, if $(D,g)$ is big, then it is $\mathbb R$-linearly equivalent to an effective metrised $\mathbb R$-divisor.

\begin{defi}\label{def:mu:inf}
Let $(D,g)$ be a metrised $\mathbb R$-divisor on $X$. We denote by $\mu_{\inf}(g)$ the value
\[\sum_{x\in X^{(1)}}\mu_{\inf,x}(g)[k(x):k]\in\mathbb R\cup\{-\infty\},\]
where by definition (see \S\ref{Sec: infimum slopes})
\[\mu_{\inf,x}(g)=\inf_{\xi\in\mathopen{]}\eta_0,x_0\mathclose[}\frac{g(\xi)}{t(\xi)}.\]
Note that \[\mu_{\inf,x}(g)\leqslant\lim_{\xi\rightarrow x_0}\frac{g(\xi)}{t(\xi)}=\operatorname{ord}_x(D).\] Therefore,
\begin{equation}\label{Equ: mu inf}\mu_{\inf}(g)\leqslant\sum_{x\in X^{(1)}}\operatorname{ord}_x(D)[k(x):k]=\deg(D).\end{equation}
Moreover, if $D_1$ is an $\mathbb R$-divisor and $g_{D_1}$ is the canonical Green function associated with $D_1$, then one has
\begin{equation}\label{Equ: mu inf x translated by canonical green}
\forall\,x\in X^{(1)},\quad \mu_{\inf,x}(g+g_{D_1})=\mu_{\inf,x}(g)+\operatorname{ord}_x(D_1)
\end{equation}
and hence
\begin{equation}\label{Equ: mu inf translated by canonical green}\mu_{\inf}(g+g_{D_1})=\mu_{\inf}(g)+\deg(D_1).\end{equation}
\end{defi}

The invariant $\mu_{\inf}(\ndot)$ is closely related to the effectivity of a metrised $\mathbb R$-divisor.

\begin{prop}\label{Pro: effefctive implies mu inf}
Let $(D,g)$ be a metrised $\mathbb R$-divisor. Assume that there exists an element $\phi\in\Gamma(D)_{\mathbb R}^{\times}$   such that $g+g_{(\phi)}\geqslant 0$. Then for all but a finite number of $x\in X^{(1)}$ one has $\mu_{\inf,x}(g)= 0$. Moreover, $\mu_{\inf}(g)\geqslant 0$.
\end{prop}
\begin{proof}
By \eqref{Equ: mu inf x translated by canonical green}
, for any $x\in X^{(1)}$ one has
\[\mu_{\inf,x}(g+g_{(\phi)})=\mu_{\inf,x}(g)+\operatorname{ord}_x(\phi).\]
Therefore, for all but a finite number of $x\in X^{(1)}$, one has 
\[\mu_{\inf,x}(g)=\mu_{\inf,x}(g+g_{(\phi)})\geqslant 0.\]
Note that $\mu_{\inf,x}(g)\leqslant\mathrm{ord}_x(D)$ for any $x\in X^{(1)}$, and hence $\mu_{\inf,x}(g)\leqslant 0$ for $x\in X^{(1)}\setminus\operatorname{Supp}(D)$.  We then deduce that $\mu_{\inf,x}(g)$ vanishes for all but finitely many $x\in X^{(1)}$.
Moreover, by \eqref{Equ: mu inf translated by canonical green} one has
\[\mu_{\inf}(g)=\mu_{\inf}(g+g_{(\phi)})\geqslant 0.\] 
\end{proof}

\begin{prop}\label{Pro: criterion of effecitivty up to R linear equivalence}
Let $(D,g)$ be a metrised $\mathbb R$-divisor on $X$.
\begin{enumerate}[label=\rm(\arabic*)]
\item\label{Item: r linearly equivalent} $(D,g)$ is $\mathbb R$-linearly equivalent to an effective metrised $\mathbb R$-divisor if and only if there exists $s\in\Gamma(D)_{\mathbb R}^{\times}$ with $\norm{s}_g\leqslant 1$.
\item\label{Item: effective implies pseudoeffecgive} If $(D,g)$ is $\mathbb R$-linearly equivalent to an effective metrised $\mathbb R$-divisor, then $(D,g)$ is pseudo-effective.
\item\label{Item: criterion for deg 0} Assume that $\mu_{\inf,x}(g)\geqslant 0$ for all but finitely many $x\in X^{(1)}$ and $\mu_{\inf}(g)> 0$, then $(D,g)$ is $\mathbb R$-linearly equivalent to an effective metrised $\mathbb R$-divisor.
\item\label{Item: criterion of effecitvity 2} Assume that $\mu_{\inf,x}(g)\geqslant 0$  for all but finitely many $x\in X^{(1)}$, and $\mu_{\inf}(g)=0$, then $(D,g)$ is $\mathbb R$-linearly equivalent to an effective metrised $\mathbb R$-divisor if and only if the $\mathbb R$-divisor $\sum_{x\in X^{(1)}}\mu_{\inf,x}(g)x$ is principal.
\end{enumerate}
\end{prop}
\begin{proof}
\ref{Item: r linearly equivalent} Let $s$ be an element of $\Gamma(D)_{\mathbb R}^{\times}$, one has
\[(D,g)+\widehat{(s)}=(D+(s),g_{(s)}+g).\]
By definition, $D+(s)$ is effective. Moreover, 
\[-\ln\norm{s}_g=\inf(g_{(s)}+g).\]
Therefore, $\norm{s}_g\leqslant 1$ if and only if $g_{(s)}+g\geqslant 0$.

\ref{Item: effective implies pseudoeffecgive} Since there exists $s\in\Gamma(D)_{\mathbb R}^{\times}$ such that $\norm{s}_g\leqslant 1$, one has $\lambda_{\mathrm{ess}}(D,g)\geqslant 0$ and $\deg(D)\geqslant 0$.  Let $(D_1,g_1)$ be a big metrised $\mathbb R$-divisor. By Proposition \ref{Pro: criterion big}, one has $\deg(D)>0$ and $\lambda_{\mathrm{ess}}(D,g)>0$. Therefore, 
\[\deg(D+D_1)=\deg(D)+\deg(D_1)>0,\]
and, by Proposition \ref{Pro: superadditivity},
\[\lambda_{\mathrm{ess}}(D+D_1,g+g_1)\geqslant \lambda_{\mathrm{ess}}(D,g)+\lambda_{\mathrm{ess}}(D_1,g_1)>0.\]
Still by Proposition \ref{Pro: criterion big}, we obtain that $(D+D_1,g+g_1)$ is big.

\ref{Item: criterion for deg 0} Let $S$ be a finite subset of $X^{(1)}$ which contains $\mathrm{Supp}(D)$ and all $x\in X^{(1)}$ such that $\mu_{\inf,x}(g)<0$, and which satisfies the inequality
\[\sum_{x\in S}\mu_{\inf,x}(g)[k(x):k]>0.\]
Since the $\mathbb R$-divisor $\sum_{x\in S}\mu_{\inf,x}(g)x$ has a positive degree, there exists an element $\varphi$ of $\mathrm{Rat}(X)_{\mathbb R}^{\times}$ such that
\begin{equation}\label{Equ: bound of ord x varphi}\mathrm{ord}_x(\varphi)\geqslant\begin{cases}
-\mu_{\inf,x}(g),&\text{if $x\in S$},\\
0,&\text{if $x\in X^{(1)}\setminus S$}.
\end{cases}\end{equation}
Note that $\mu_{\inf,x}(g)\leqslant\operatorname{ord}_x(D)$ for any $x\in X^{(1)}$. Hence $\varphi\in\Gamma(D)_{\mathbb R}^{\times}$. Moreover, by \eqref{Equ: bound of ord x varphi} one has
\[g+g_{(\varphi)}\geqslant 0.\]
Hence $(D,g)+\widehat{(\varphi)}$ is effective.

\ref{Item: criterion of effecitvity 2} Note that $\mu_{\inf,x}(g)\leqslant\operatorname{ord}_x(D)=0$ for any $x\in X^{(1)}\setminus\operatorname{Supp}(D)$, we obtain that $\mu_{\inf,x}(g)=0$ for all but finitely many $x\in X^{(1)}$. Therefore $\sum_{x\in X^{(1)}}\mu_{\inf,x}(g)x$ is well-defined as an $\mathbb R$-divisor on $X$.

Assume that the $\mathbb R$-divisor $\sum_{x\in S}\mu_{\inf,x}(g)x$ is principal, namely of the form $(\varphi)$ for some $\varphi\in\mathrm{Rat}(X)_{\mathbb R}^{\times}$. Then the metrised $\mathbb R$-divisor \[(D,g)-\widehat{(\varphi)}\] is effective. Conversely, if $\phi$ is an element of $\operatorname{Rat}(X)_{\mathbb R}^{\times}$ which is different from $-\sum_{x\in X^{(1)}}\mu_{\inf,x}(g)x$, then there exists $x\in X^{(1)}$ such that $\operatorname{ord}_x(\phi)<-\mu_{\inf,x}(g)$ since
\[\sum_{x\in X^{(1)}}\operatorname{ord}_x(\phi)[k(x):k]=-\sum_{x\in X^{(1)}}\mu_{\inf,x}(g)[k(x):k]=0.\]
Therefore the function $g+g_{(\phi)}$ can not be non-negative.
\end{proof}

Combining Propositions \ref{Pro: effefctive implies mu inf} and \ref{Pro: criterion of effecitivty up to R linear equivalence}, we obtain the following criterion of effectivity up to $\mathbb R$-linear equivalence for metrised $\mathbb R$-divisors.

\begin{theo}\label{Thm: criterion of effecitivity}
Let $(D,g)$ be a metrised $\mathbb R$-divisor on $X$. Then $(D,g)$ is $\mathbb R$-linearly equivalent to an effective metrised $\mathbb R$-divisor if and only if $\mu_{\inf,x}(g)=0$ for all but finitely many $x\in X^{(1)}$ and if one of the following conditions holds:
\begin{enumerate}[label=\rm(\alph*)]
\item $\mu_{\inf}(g)>0$,
\item $\sum_{x\in X^{(1)}}\mu_{\inf,x}(g)x$ is a principal $\mathbb R$-divisor on $X$.
\end{enumerate}
\end{theo}

\subsection{Criterion of pseudo-effectivity} By using the criteria of effectivity up to $\mathbb R$-linear equivalence in the previous subsection, we prove a numerical criterion of pseudo-effectivity in terms of the invariant $\mu_{\inf}(\ndot)$.

\begin{lemm}\label{Lem: closedness of pseudoeffective}
Let $(D,g)$ be a metrised $\mathbb R$-divisor.
Assume that $(D,g+\varepsilon)$ is pseudo-effective for any $\varepsilon>0$. Then $(D,g)$ is also pseudo-effective.
\end{lemm}
\begin{proof}
Let $(D_1,g_1)$ be a big metrised $\mathbb R$-divisor. By Proposition \ref{Pro: criterion big}, one has $\deg(D_1)>0$ and $\lambda_{\mathrm{ess}}(D_1,g_1)>0$. Let $\varepsilon$ be a positive number such that $\varepsilon<\lambda_{\mathrm{ess}}(D_1,g_1)$. By \eqref{Equ: add a constant2} one has
\[\lambda_{\mathrm{ess}}(D_1,g_1-\varepsilon)=\lambda_{\mathrm{ess}}(D_1,g_1)-\varepsilon>0.\]
Hence $(D_1,g_1-\varepsilon)$ is big (by Proposition \ref{Pro: criterion big}). Therefore,
\[(D,g)+(D_1,g_1)=(D+D_1,g+g_1)=(D,g+\varepsilon)+(D_1,g_1-\varepsilon)\]
is big. 
\end{proof}

\begin{prop}\phantomsection\label{Pro: pseudoeffective}
A metrised $\mathbb R$-divisor $(D,g)$ on $X$ is pseudo-effective if and only if $\mu_{\inf}(g)\geqslant 0$.
\end{prop}
\begin{proof}``$\Longleftarrow$'':
For any $\varepsilon>0$, one has $\mu_{\inf}(g+\varepsilon)>0$. By Theorem \ref{Thm: criterion of effecitivity}, $(D,g+\varepsilon)$ is $\mathbb R$-linearly equivalent to an effective metrised $\mathbb R$-divisor, and hence is pseudo-effective (see Proposition \ref{Pro: criterion of effecitivty up to R linear equivalence}
 \ref{Item: effective implies pseudoeffecgive}). By Lemma \ref{Lem: closedness of pseudoeffective}, we obtain that $(D,g)$ is pseudo-effective.
 
``$\Longrightarrow$'': We begin with the  case where $\deg(D)>0$. If $(D,g)$ is pseudo-effective, then by Corollary \ref{Cor: pseudoeffective implies lambda ess positive}, one has $\lambda_{\mathrm{ess}}(D,g)\geqslant 0$. Hence $(D,g+\varepsilon)$ is big for any $\varepsilon>0$ (by \eqref{Equ: add a constant2} and Proposition \ref{Pro: criterion big}). In particular, one has $\mu_{\inf}(g+\varepsilon)\geqslant 0$ for any $\varepsilon>0$. For each $x\in X^{(1)}$, the function $(\varepsilon>0)\mapsto\mu_{\inf,x}(g+\varepsilon)$ is decreasing and bounded from below by $\mu_{\inf,x}(g)$. Moreover, for any $\xi\in \mathopen{]}\eta_0,x_0\mathclose{[}$ one has
\[\inf_{\varepsilon>0}\frac{g(\xi)+\varepsilon}{t(\xi)}=\frac{g(\xi)}{t(\xi)}\]
and hence
\[\inf_{\varepsilon>0}\mu_{\inf,x}(g+\varepsilon)\leqslant\frac{g(\xi)}{t(\xi)}.\]
Therefore we obtain
\[\inf_{\varepsilon>0}\mu_{\inf,x}(g+\varepsilon)=\mu_{\inf,x}(g).\] 
By the monotone convergence theorem we deduce that
\[\mu_{\inf}(g)=\inf_{\varepsilon>0}\mu_{\inf}(g+\varepsilon)\geqslant 0.\] 

We now treat the general case. Let $y$ be a closed point of $X$. We consider $y$ as an $\mathbb R$-divisor on $X$ and denote it by $D_y$. Let $g_y$ be the canonical Green function associated with $D_y$. As $D_y$ is effective and $g_y\geqslant 0$, we obtain that $(D_y,g_y)$ is effective and hence pseudo-effective. Therefore, for any $\delta>0$, 
\[(D,g)+\delta(D_y,g_y)=(D+\delta D_y,g+\delta g_y)\]
is pseudo-effective. Moreover, one has $\deg(D+\delta D_y)>0$. Therefore, by what we have shown above, one has
\[\mu_{\inf}(g+\delta g_y)=\mu_{\inf}(g)+\delta [k(y):k]\geqslant 0.\]
Since $\delta> 0$ is arbitrary, one obtains $\mu_{\inf}(g)\geqslant 0$.
\end{proof}

\subsection{Positivity of Green functions}

Let $D$ be an $\mathbb R$-divisor on $X$ such that $\Gamma(D)_{\mathbb R}^{\times}$ is not empty. For any Green function $g$ of $D$, we define a map 
\[\widetilde g:X^{\mathrm{an}}\setminus\{x_0\,:\,x\in X^{(1)}\}\longrightarrow \mathbb R\] as follows. For any $\xi\in X^{\mathrm{an}}\setminus\{x_0\,:\,x\in X^{(1)}\}$, let
\begin{equation}\label{Equ: definition of Pg}\widetilde g(\xi):=\sup_{\begin{subarray}{c}s\in\Gamma(D)_{\mathbb R}^{\times}\end{subarray}}\big(\ln|s|(\xi)-\ln\norm{s}_{g}\big).\end{equation}

\begin{prop}
Let $D$ be an $\mathbb R$-divisor on $X$ such that $\Gamma(D)_{\mathbb Q}^{\times}$ is not empty. For any $\xi\in X^{\mathrm{an}}\setminus\{x_0\,:\,x\in X^{(1)}\}$ one has
\begin{equation}
\widetilde{g}(\xi)=\sup_{s\in\Gamma(D)_{\mathbb Q}^{\times}}\big(\ln|s|(\xi)-\ln\norm{s}_g\big).
\end{equation}
\end{prop}
\begin{proof}
Without loss of generality, we may assume that $D$ is effective. For clarifying the presentation, we denote temporarily by 
\[\widetilde g_0(\xi):=\sup_{s\in\Gamma(D)_{\mathbb Q}^{\times}}\big(\ln|s|(\xi)-\ln\norm{s}_g\big).\]

Let $s$ be an element of $\Gamma(D)_{\mathbb R}^{\times}$, which is written in the form $s_1^{a_1}\cdots s_r^{a_r}$, where $s_1,\ldots,s_r$ are elements of $\mathrm{Rat}(X)_{\mathbb Q}^{\times}$ and $a_1,\ldots,a_r$ are positive real numbers, which are linearly independent over $\mathbb Q$. Let $\{x_1,\ldots,x_n\}$ be the support of $(s)$. By Lemma \ref{Lem: linear independence}, for any $i\in\{1,\ldots,r\}$, the support of $(s_i)$ is contained in $\{x_1,\ldots,x_n\}$. Since $s$ belongs to $\Gamma(D)_{\mathbb Q}^{\times}$, for $j\in\{1,\ldots,n\}$, one has
\[a_1\operatorname{ord}_{x_j}(s_1)+\cdots+a_r\operatorname{ord}_{x_j}(s_r)+\operatorname{ord}_{x_j}(D)\geqslant 0.\]
By Lemma \ref{Lem: approximation by rational solutions} and Remark  \ref{Rem: suite d'approximation},  there exist a sequence $(\varepsilon^{(m)})_{m\in\mathbb N}$ in $\mathbb Q_{>0}$ and a sequence \[\boldsymbol{\delta}^{(m)}=(\delta_1^{(m)},\ldots,\delta_r^{(m)}),\quad m\in\mathbb N\] of elements of $\mathbb R_{>0}^r$ which satisfy the following conditions
\begin{enumerate}[label=\rm(\arabic*)]
\item the sequence $(\varepsilon^{(m)})_{m\in\mathbb N}$ converges to $0$,
\item the sequence $(\boldsymbol{\delta}^{(m)})_{m\in\mathbb N}$ converges to $(0,\ldots,0)$,
\item if we denote by $u^{(m)}$ the element 
\[s_1^{\delta_1^{(m)}}\!\!\cdots\, s_r^{\delta_r^{(m)}}\]
in $\mathrm{Rat}(X)_{\mathbb R}^{\times}$, one has $u^{(m)}\in\Gamma(\varepsilon^{(m)}D)^{\times}_{\mathbb R}$ and \[s^{(m)}:=(su^{(m)})^{(1+\varepsilon^{(m)})^{-1}}\in\mathrm{Rat}(X)^{\times}_{\mathbb Q},\] and hence it belongs to $\Gamma(D)_{\mathbb Q}^{\times}$.
\end{enumerate}
Note that one has
\[\norm{su^{(m)}}_{(1+\varepsilon^{(m)})g}\leqslant\norm{s}_g\cdot\norm{u^{(m)}}_{\varepsilon^{(m)}g}.\]
Since $u^{(m)}\in\Gamma(\varepsilon^{(m)}D)^{\times}_{\mathbb R}$, one has
\[-\ln\norm{u^{(m)}}_{\varepsilon^{(m)}}=\inf
\Big(\varepsilon^{(m)}g+\sum_{i=1}^r\delta_i^{(m)}g_{(s_i)}\Big)\geqslant\varepsilon^{(m)}\inf\varphi_g.
\]
Therefore,
\[-\ln\norm{s}_g\leqslant -(1+\varepsilon^{(m)})\ln\norm{s^{(m)}}_g-\varepsilon^{(m)}\inf\varphi_g.\]
Thus 
\[\begin{split}\ln|s|(\xi)-\ln\norm{s}_g&=
(1+\varepsilon^{(m)})\ln|s^{(m)}|(\xi)-\sum_{i=1}^r\delta_i^{(m)}\ln|s_i|(\xi)-\ln\norm{s}_g\\
&\leqslant(1+\varepsilon^{(m)})\widetilde g_0(\xi)-\sum_{i=1}^r\delta_i^{(m)}\ln|s_i|(\xi)-\varepsilon^{(m)}\inf\varphi_g.
\end{split}\]
Taking the limit when $m\rightarrow+\infty$, we obtain
\[\ln|s|(\xi)-\ln\norm{s}_g\leqslant\widetilde g_0(\xi).\]
The proposition is thus proved.
\end{proof}

\begin{prop}\phantomsection\label{Pro: tilde g}
Let $D$ be an $\mathbb R$-divisor on $X$ such that $\Gamma(D)^{\times}_{\mathbb R}$ is not empty. For any Green function $g$ of $D$, the function $\widetilde g$ extends on $X^{\mathrm{an}}$ to a convex Green function of $D$ which is bounded from above by $g$.
\end{prop}
\begin{proof}
We first show that $\widetilde g$ is bounded from above by $g$. For any $s\in\Gamma(D)_{\mathbb R}^{\times}$ one has
\[\forall\,\xi\in X^{\mathrm{an}},\quad -\ln\norm{s}_g=\inf(g_{(s)}+g)\leqslant g(\xi)-\ln|s|(\xi),\]
so that 
\[\forall\,\xi\in X^{\mathrm{an}},\quad \ln|s|(\xi)-\ln\norm{s}_g\leqslant g(\xi).\]
It remains to check that $\widetilde g$ extends by continuity to a convex Green function of $D$.

We first treat the case where $\deg(D)=0$. By Remark \ref{Rem: degree 0 effective} we obtain that $\Gamma(D)_{\mathbb R}^{\times}$ contains a unique element $s$ and one has $D=-(s)$.  Therefore \[\widetilde g=\ln|s|-\ln\norm{s}_g=g_D-\ln\norm{s}_g,\]
which clearly extends to a convex Green function of $D$. 

In the following, we assume that $\deg(D)>0$. Let $x$ be an element of $X^{(1)}$. The function $\widetilde g\circ \xi_x|_{\mathbb R_{>0}}$ (see \S\ref{Subsec: tree of length 1}) can be written as
\[(t\in\mathbb R_{>0})\longmapsto\sup_{s\in\Gamma(D)_{\mathbb R}^{\times}} -t\operatorname{ord}_x(s)-\ln\norm{s}_g,\]
which is the supremum of a family of affine functions on $t>0$. Therefore $\widetilde g\circ\xi_x|_{\mathbb R_{>0}}$ is a convex function on $\mathbb R_{>0}$. This expression also shows that, for any $s\in\Gamma(D)_{\mathbb R}^{\times}$, one has
\[\liminf_{\xi\rightarrow x_0}\frac{\widetilde g(\xi)}{t(\xi)}\geqslant\operatorname{ord}_x(s^{-1}).\]
By Proposition \ref{Pro: sup of s in Gamma D} (see also Remark \ref{Rem: degree 0 effective}), one has
\[\liminf_{\xi\rightarrow x_0}\frac{\widetilde g(\xi)}{t(\xi)}\geqslant \sup_{s\in\Gamma(D)_{\mathbb R}^{\times}}\operatorname{ord}_x(s^{-1})=\operatorname{ord}_x(D).\]
Moreover, since $\widetilde g\leqslant g$ and since $g$ is a Green function of $D$, one has
\[\limsup_{\xi\rightarrow x_0}\frac{\widetilde g(\xi)}{t(\xi)}\leqslant \lim_{\xi\rightarrow x_0}\frac{g(\xi)}{t(\xi)}=\operatorname{ord}_x(D).\] 
Therefore one has 
\[\lim_{\xi\rightarrow x_0}\frac{\widetilde g(\xi)}{t(\xi)}=\operatorname{ord}_x(D).\]
The proposition is thus proved.  
\end{proof}

\begin{defi}\label{Def: psh}
Let $(D,g)$ be a metrised $\mathbb R$-divisor on $X$ such that $\Gamma(D)_{\mathbb R}^{\times}$ is not empty. We call $\widetilde g$ the \emph{plurisubharmonic envelope} of the Green function $g$. In the case where the equality $g=\widetilde g$ holds, we say that the Green function $g$ is \emph{plurisubharmonic}. Note that $\widetilde g$ is bounded from above by the convex envelope $\widebreve{g}$ of $g$.
\end{defi}

\begin{rema}
If we set $\varphi = g - \widetilde{g}$, then $\varphi$ is a non-negative continuous function on $X^{\mathrm{an}}$, 
so that, in some sense, the decomposition $(D, g) = (D, \widetilde{g}) + (0, \varphi)$ gives rise to
a Zariski decomposition of $(D,g)$ on $X$.
\end{rema}

\begin{theo}\label{Thm:criterion of g tilde}
Let $(D,g)$ be an adelic $\mathbb R$-Cartier divisor on $X$ such that $\Gamma(D)_{\mathbb R}^{\times}$ is not empty. Then $\widetilde{g}(\eta_0)=g(\eta_0)$ if and only if $\mu_{\inf}(g-g(\eta_0))\geqslant 0$. Moreover, in the case where these equivalent conditions are satisfied, $\widetilde{g}$ identifies with the convex envelop $\widebreve{g}$ of $g$.
\end{theo}
\begin{proof}
{\bf Step 1:} We first treat the case where $\deg(D)=0$. In this case $\Gamma(D)_{\mathbb R}^{\times}$ contains a unique element $s$ (with $D=-(s)$) and one has (see the proof of Proposition \ref{Pro: tilde g}) \[\widetilde g=g_D-\ln\norm{s}_g.\]  Hence \[\widetilde g(\eta_0)=-\ln\norm{s}_g=\inf(g_{(s)}+g)=\inf\varphi_g.\]
Note that $g(\eta_0)=\varphi_g(\eta_0)$. Therefore, the equality $\widetilde g(\eta_0)=g(\eta_0)$ holds if and only if $\varphi_g$ attains its minimal value at $\eta_0$, or equivalently 
\[\forall\,x\in X^{(1)},\quad\mu_{\inf,x}(g-g(\eta_0))=\operatorname{ord}_x(g).\]
In particular, if $\widetilde g(\eta_0)=g(\eta_0)$, then
\[\mu_{\inf}(g-g(\eta_0))=\sum_{x\in X^{(1)}}\operatorname{ord}_x(g)[k(x):k]=0.\]
Conversely, if $\mu_{\inf}(g-g(\eta_0))\geqslant 0$, then by \eqref{Equ: mu inf} one obtains that \[\mu_{\inf}(g-g(\eta_0))= 0\]
and the equality $\mu_{\inf,x}(g-g(\eta_0))=\operatorname{ord}_x(g)$ holds for any $x\in X^{(1)}$. Hence $\widetilde g(\eta_0)=g(\eta_0)$.

If $\varphi$ is a bounded Green function on $X^{\mathrm{an}}$,  which is bounded from above by $\varphi_g$, by Proposition \ref{prop:conv:properties} one has \[\varphi(\xi)\leqslant\varphi(\eta_0)\leqslant\varphi_g(\eta_0)=g(\eta_0)\] for any $\xi\in X^{\mathrm{an}}$.
In the case where the inequality $\widetilde g(\eta_0)=g(\eta_0)$ holds, the function $\widetilde g=g_D+g(\eta_0)$ is the largest convex Green function of $D$ which is bounded from above by $g$, namely the equality $\widetilde g=\widebreve{g}$ holds.    

{\bf Step 2:}
In the following, we assume that $\deg(D)>0$. By replacing $g$ by $g-g(\eta_0)$ it suffices to check that, in the case where $g(\eta_0)=0$, the equality $\widetilde g(\eta_0)=0$ holds if and only if $\mu_{\inf}(g)\geqslant 0$. By definition one has 
\[\widetilde{g}(\eta_0)=\sup_{s\in\Gamma(D)_{\mathbb R}^{\times}}(-\ln\norm{s}_g).\]

\noindent{\it Step 2.1:} We first assume that $\widetilde g(\eta_0)=0$ and show that $\mu_{\inf}(g)\geqslant 0$. Let $s$ be an element of $\Gamma(D)_{\mathbb R}^{\times}$. By definition one has
\[-\ln\norm{s}_g=\inf_{\xi\in X^{\mathrm{an}}}(g+g_{(s)})(\xi).\]
Let $(D_1,g_1)$ be a big metrised $\mathbb R$-divisor. We fix $s_1\in\Gamma(D_1)_{\mathbb R}^{\times}$ such that $\norm{s_1}_{g_1}<1$ (see Proposition \ref{Pro: criterion big} for the existence of $s_1$). Since $\widetilde g(\eta_0)=0$, there exists $s\in\Gamma(D)_{\mathbb R}^{\times}$ such that 
\[\norm{ss_1}_{g+g_1}\leqslant\norm{s}_g\cdot\norm{s_1}_{g_1}<1.\]
Therefore $\lambda_{\mathrm{ess}}(D+D_1,g+g_1)>0$ and hence $(D+D_1,g+g_1)$ is big (see Proposition \ref{Pro: criterion big}). We then obtain that $(D,g)$ is pseudo-effective and hence $\mu_{\inf}(g)\geqslant 0$ (see  Proposition \ref{Pro: pseudo effective, mu inf positive}).
\medskip

\noindent{\it Step 2.2:} We now show that $\mu_{\inf}(g)>0$ implies $\widetilde g(\eta_0)=0$. For $\varepsilon>0$, let  \[U_\varepsilon:=\{\xi\in X^{\mathrm{an}}\,:\,g(\xi)>-\varepsilon\}.\] This is an open subset of $X^{\mathrm{an}}$ which contains $\eta_0$. Hence there exists a finite set $X_\varepsilon^{(1)}$ of closed points of $X$, which contains the support of $D$ and such that, for any closed point $x$ of $X$ lying outside of $X_\varepsilon^{(1)}$, one has $g|_{[\eta_0,x_0]}> -\varepsilon$. Moreover, for any $x\in X^{(1)}\setminus\operatorname{Supp}(D)$ one has $\mu_{\inf,x}(g)\leqslant 0$ since $g$ is bounded on $[\eta_0,x_0]$. Therefore, the condition $\mu_{\inf}(g)>0$ implies that 
\begin{equation}\label{Equ:mu total positive}\sum_{x\in X^{(1)}_{\varepsilon}}\mu_{\inf,x}(g)[k(x):k]> 0.\end{equation}
We let $s_\varepsilon$ be an element of $\mathrm{Rat}(X)^{\times}_{\mathbb R}$ such that $\operatorname{ord}_x(s_\varepsilon)\geqslant -\mu_{\inf,x}(g)$ for any $x\in  X_\varepsilon^{(1)}$ and that $\operatorname{ord}_x(s_{\varepsilon})\geqslant 0$ for any $x\in X^{(1)}\setminus X_\varepsilon^{(1)}$. This is possible by the inequality \eqref{Equ:mu total positive}. In fact, the $\mathbb R$-divisor
\[E=\sum_{x\in X_{\varepsilon}^{(1)}}\mu_{\inf,x}(g)\cdot x\]
has a positive degree, and hence $\Gamma(E)_{\mathbb R}^{\times}$ is not empty. 
Note that $\mu_{\inf,x}(g)\leqslant\operatorname{ord}_x(D)$ for any $x\in X^{(1)}$. Therefore $D+(s_\varepsilon)$ is effective. Moreover, for any $x\in X^{(1)}\setminus X_\varepsilon^{(1)}$ and $\xi\in[\eta_0,x_0[$ one has
\[(g-\ln|s_\varepsilon|)(\xi)\geqslant g(\xi)\geqslant -\varepsilon.\]
Therefore we obtain $\norm{s_\varepsilon}\leqslant \mathrm{e}^{\varepsilon}$ since $g-\ln|s_\varepsilon|\geqslant 0$ on $\mathopen{[}\eta_0,x_0\mathclose{[}$ for any $x\in X_\varepsilon^{(1)}$. This leads to $\widetilde g(\eta_0)=0$ since $\varepsilon$ is arbitrary.

\medskip

\noindent{\it Step 2.3:} We assume that $\mu_{\inf}(g)>0$ and show that $\widebreve{g}=\widetilde g$. By definition, for any $x\in X^{(1)}$, the function $\widetilde g\circ\xi_x|_{\mathbb R_{>0}}$ can be written as the supremum of a family of affine functions, hence it is a convex fonction on $\mathbb R_{>0}$ bounded from above by $g$. In the following, we fix a closed point $x$ of $X$. 

Without loss of generality, we may assume that $x$ belongs to $X_\varepsilon^{(1)}$ for any $\varepsilon>0$. Note that for any $\xi\in[\eta_0,x_0]$ one has
\[\widetilde g(\xi)\geqslant\ln|s_\varepsilon|(\xi)-\ln\norm{s_\varepsilon}_g\geqslant \mu_{\inf,x}(g)t(\xi)-\varepsilon. \]
Since $\varepsilon>0$ is arbitrary, one has $\widetilde g(\xi)\geqslant \mu_{\inf,x}(g)t(\xi)$.

Let $a$ and $b$ be real numbers such that $at(\xi)+b\leqslant g(\xi)$ for any $\xi\in[\eta_0,x_0[$. Then one has $b\leqslant 0$ since $g(\eta_0)=0$. Moreover, one has 
\[a=\lim_{\xi\rightarrow x_0}\frac{at(\xi)+b}{t(\xi)}\leqslant\lim_{\xi\rightarrow x_0}\frac{g(\xi)}{t(\xi)}=\operatorname{ord}_x(D). \] We will show that $at(\xi)+b\leqslant \widetilde g(\xi)$ for any $\xi\in[\eta_0,x_0[$. This inequality is trivial when $a\leqslant\mu_x(g)$ since $\widetilde g(\xi)\geqslant\mu_{\inf,x}(g)t(\xi)$ and $b\leqslant 0$. In the following, we assume that $a>\mu_x(g)$.

For any $\varepsilon>0$, we let $s_{\varepsilon}^{a,b}$ an element of $\mathrm{Rat}(X)_{\mathbb R}^{\times}$ such that 
\[\mathrm{ord}_y(s_{\varepsilon}^{a,b})\geqslant\begin{cases}
 -a&\text{if $y=x$,}\\
 -\mu_{\inf,y}(g)&\text{if $y\in X_\varepsilon^{(1)}$, $y\neq x$,}\\
 0&\text{if $y\in X^{(1)}\setminus X_\varepsilon^{(1)}$.}
\end{cases}\]
This is possible since $\mu_{\inf}(g)> 0$ and $a>\mu_{\inf,x}(g)$. Note that $s_\varepsilon^{a,b}$ belongs to $\Gamma(D)^{\times}_{\mathbb R}$. Moreover, for $\xi\in [\eta_0,x_0[$, one has 
\[g(\xi)-\ln|s_\varepsilon^{a,b}|(\xi)\geqslant g(\xi)-at(\xi)\geqslant b;\]
for any $y\in X_\varepsilon^{(1)}\setminus\{x\}$, one has 
\[g(\xi)-\ln|s_\varepsilon^{a,b}|(\xi)=g(\xi)-\mu_{\inf,y}(g)t(\xi)\geqslant 0;\]
for any $y\in X^{(1)}\setminus X_\varepsilon$, one has $g(\xi)-\ln|s_{\varepsilon}^{a,b}|(\xi)\geqslant g(\xi)\geqslant -\varepsilon$. Therefore we obtain 
\[-\ln\norm{s_{\varepsilon}^{a,b}}\geqslant \min\{-\varepsilon,b\}.\]
As a consequence, for any $\xi\in[\eta_0,x_0[$, one has
\[\widetilde g(\xi)\geqslant\ln|s_\varepsilon^{a,b}|(\xi)-\ln\norm{s}_g=at(\xi)+\min\{-\varepsilon,b\}.\]
Since $b\leqslant 0$ and since $\varepsilon>0$ is arbitrary, we obtain $\widetilde g(\xi)\geqslant at(\xi)+b$.
\medskip

\noindent{\bf Step 3:}
In this step, we assume that $\deg(D)>0$ and $\mu_{\inf}(g-g(\eta_0))=0$. We show that and $\widebreve{g}=\widetilde g$. Without loss of generality, we assume that $g(\eta_0)=0$.  Since 
\[\deg(D)=\sum_{x\in X^{(1)}}\mu_x(g)[k(x):k]>0,\]
there exists $y\in X^{(1)}$ such that \[\mu_{\inf,y}(g)<\operatorname{ord}_x(D)=\mu_y(g).\]
We let $g_0$ be the bounded Green function on $\mathcal T(X^{(1)})$ such that $g_0(\xi)=0$ for \[\xi\in\bigcup_{x\in X^{(1)},\,x\neq y}\mathopen{[}\xi_0,x_0\mathclose{]},\]
and 
\[g_0(\xi)=\min\{t(\xi),1\},\quad \text{for $\xi\in\mathopen{[}\eta_0,y_0\mathclose{]}$}.\]
One has $g_0\geqslant 0$, and 
\[\sup_{\xi\in X^{(1)}}g_0(\xi)\leqslant 1.\] 
For any $\varepsilon>0$, we denote by $g_\varepsilon$ the Green function $g+\varepsilon g_0$. One has \[\mu_{\inf,x}(g_\varepsilon)>\mu_{\inf,x}(g)\geqslant 0.\] Moreover, by definition $g_\varepsilon(\eta_0)=0$. Therefore, by what we have shown in Step 2.2, one has 
\[\widetilde g_\varepsilon(\eta_0)=\sup_{s\in\Gamma(D)_{\mathbb R}^{\times}}\big(-\ln\norm{s}_{g_\varepsilon}\big)=0.\]
Note that for any $s\in\Gamma(D)_{\mathbb R}^{\times}$ one has
\[\mathrm{e}^\varepsilon\norm{s}_{g_\varepsilon} \geqslant\norm{s}_g\geqslant\norm{s}_{g_\varepsilon}. \]
Hence we obtain
\[\widetilde g_{\varepsilon}-\varepsilon\leqslant\widetilde g\leqslant\widetilde g_{\varepsilon}.\]
Since $\widetilde g_\varepsilon(\eta_0)=0$ for any $\varepsilon>0$, we obtain $\widetilde g(\eta_0)=0$. Finally, the inequalities
\[g_\varepsilon-\varepsilon\leqslant g\leqslant g_\varepsilon\]
leads to
\[\widebreve{g}_{\!\varepsilon}-\varepsilon\leqslant \widebreve{g}\leqslant\widebreve{g}_{\!\varepsilon}.\]
By what we have shown in Step 2.3, one has $\widetilde{g}_\varepsilon=\widebreve{g}_{\!\varepsilon}$ for any $\varepsilon>0$. Therefore the equality $\widetilde g=\widebreve{g}$ holds.
\medskip
\end{proof}

\begin{coro}
Let $(D,g)$ be a metrised $\mathbb R$-divisor on $X$ such that $\Gamma(D)_{\mathbb R}^{\times}\neq\emptyset$. Then $g$ is plurisubharmonic if and only if it is convex and $\mu_{\inf}(g-g(\eta_0))\geqslant 0$.
\end{coro}

\subsection{Global positivity conditions under metric positivity} Let $X$ be a regular projective curve over $\Spec k$ and $(D,g)$ be a metrised $\mathbb R$-divisor. In this section, we consider global positivity conditions under the hypothesis that $g$ is plurisubharmonic.

\begin{prop}
Let $(D,g)$ be a metrised $\mathbb R$-divisor such that $\Gamma(D)_{\mathbb R}^{\times}$ is not empty and that the Green function $g$ is plurisubharmonic.
\begin{enumerate}[label=\rm(\arabic*)]
\item\label{Item: pseudo effective if g eta positive} $(D,g)$ is pseudo-effective if and only if $g(\eta_0)\geqslant 0$.
\item\label{Item: lambda ess is g eta 0} One has $\lambda_{\mathrm{ess}}(D,g)=g(\eta_0)$.
\item\label{Item: critere de bigness} The metrised $\mathbb R$-divisor $(D,g)$ is big if and only if $\deg(D)>0$ and $g(\eta_0)>0$.
\end{enumerate}
\end{prop}
\begin{proof}

\ref{Item: pseudo effective if g eta positive}  We have already seen in Proposition \ref{Pro: pseudo effective, mu inf positive} that, if $(D,g)$ is pseudo-effective, then $g(\eta_0)\geqslant 0$. It suffices to prove that $g(\eta_0)\geqslant 0$ implies that $(D,g)$ is pseudo-effective. Since $g$ is plurisubharmonic, by  Theorem \ref{Thm:criterion of g tilde} one has
\[\mu_{\inf}(g)\geqslant\mu_{\inf}(g-g(\eta_0))\geqslant 0.\]
By Proposition \ref{Pro: pseudoeffective}, one obtains that $(D,g)$ is pseudo-effective 

\ref{Item: lambda ess is g eta 0} 
By \eqref{Equ: lambda r ess bounded} and Proposition \ref{Pro: essential minimu bounded}, it suffices to prove that $g(\eta_0)\leqslant\lambda_{\mathrm{ess}}(D,g)$.  In the case where $\deg(D)=0$, the hypotheses that $\Gamma(D)_{\mathbb R}^{\times}$ is not empty and $g$ is plurisubharmonic imply that $D$ is a principal $\mathbb R$-divisor, $\Gamma(D)_{\mathbb R}^{\times}$ contains a unique element $s$ with $D=-(s)$, and $g-g(\eta_0)$ is the canonical Green function of $D$ (see the first step of the proof of Theorem \ref{Thm:criterion of g tilde}). Therefore one has \[\lambda_{\mathrm{ess}}(D,g)=-\ln\norm{s}_g=g(\eta_0).\]
In the following we treat the case where $\deg(D)>0$. Since $g$ is plurisubharmonic, by Theorem \ref{Thm:criterion of g tilde} one has $\mu_{\inf}(g-g(\eta_0))\geqslant 0$, so that $(D,g-g(\eta_0))$ is pseudo-effective (see Proposition \ref{Pro: pseudoeffective}). As $\deg(D)>0$, by Corollary \ref{Cor: pseudoeffective implies lambda ess positive} and \eqref{Equ: add a constant2}, one has
\[\lambda_{\mathrm{ess}}(g-g(\eta_0))=\lambda_{\mathrm{ess}}(g)-g(\eta_0)\geqslant 0.\]

\ref{Item: critere de bigness} follows from \ref{Item: lambda ess is g eta 0}  and Proposition \ref{Pro: criterion big}.


\end{proof}

\section{Hilbert-Samuel formula on curves}
\label{Sec:Hilbert-Samuel}

Let $k$ be a field equipped with the trivial valuation.
Let $X$ be a regular and irreducible projective curve of genus $R$ over $k$. The purpose of this section is to prove a Hilbert-Samuel formula for metrised $\mathbb R$-divisors on $X$.


\begin{defi}
We identify $X^{\mathrm{an}}$ with the infinite tree $\mathcal T(X^{(1)})$ and consider the weight function $w:X^{(1)}\rightarrow\mathopen{]}0,+\infty\mathclose{[}$ defined as $w(x)=[k(x):k]$. If $\overline D_1=(D_1,g_1)$ and $\overline{D}_2=(D_2,g_2)$ are metrised $\mathbb R$-divisors on $X$ such that $g_1$ and $g_2$ are both pairable (see Definition \ref{Def: pairing of Green functions}) we define $(\overline D_1\cdot\overline D_2)$ as the pairing $\langle g_1,g_2\rangle_w$, namely
\begin{equation}\label{Equ: coupling Di}\begin{split}(\overline D_1\cdot\overline D_2)=g_2(\eta_0)\deg(&D_1)+g_1(\eta_0)\deg(D_2)\\&-\sum_{x\in X^{(1)}}[k(x):k]\int_0^{+\infty}\varphi_{g_1\circ\xi_x}'(t)\varphi_{g_2\circ\xi_x}'(t)\,\mathrm{d}t.\end{split}\end{equation}  
\end{defi}

\begin{rema}
Assume that $s$ is an element of $\operatorname{Rat}(X)^{\times}_{\mathbb R}$ such that \[\overline D_2=\widehat{(s)}=((s),g_{(s)}).\] One has (see Definition \ref{Def: pairing of Green functions})
\[\begin{split}&\quad\;(\overline D_1,\overline D_2)=\langle g_1,g_{(s)}\rangle_w=g_1(\eta_0)\deg((s))=0.\end{split}\]
\end{rema}

\begin{theo}\label{thm:Hilbert:Samuel:semipositive}
Let $\overline D=(D, g)$ be a metrised $\mathbb R$-divisor on $X$ such that $\Gamma(D)_{\mathbb R}^{\times}\neq\emptyset$ and $g$ is plurisubharmonic.
Then $\widehat{\mathrm{vol}}_{\chi}(\overline D) = (\overline D\cdot\overline D)$.
\end{theo}

\begin{rema}Let $g_D$ be the canonical admissible Green function of $D$ and $\varphi_g := g - g_D$ (considered as a continuous function on $X^{\mathrm{an}}$).
Note that a plurisubharmonic Green function is convex (see Proposition \ref{Pro: tilde g}). Therefore, by Proposition \ref{Pro: relation between mu and varphig}, one has
\[\mu_{\inf,x}(g-g(\eta_0))=g'(\eta_0;x)=\operatorname{ord}_x(D)+\varphi_g'(\eta_0;x).\]
Theorem \ref{Thm:criterion of g tilde} shows that
\begin{equation}\mu_{\inf}(g-g(\eta_0))=\deg(D)+\sum_{x\in X^{(1)}}\varphi_g'(\eta_0;x)[k(x):k]\geqslant 0.\end{equation}
In the case where $\deg(D)=0$, one has $g=g(\eta_0)+g_D$ (see Step 1 in the proof of Theorem \ref{Thm:criterion of g tilde}). Therefore one has
\[(\overline D\cdot\overline D)=2g(\eta_0)\deg(D)=0=\widehat{\operatorname{vol}}_\chi(\overline D),\]
where the last equality comes from (3) of Proposition \ref{prop:formula:avol:g:g:prime}. Therefore, to prove Theorem \ref{thm:Hilbert:Samuel:semipositive}, it suffices to treat the case where $\deg(D)>0$.
\end{rema}

\begin{enonce}{Assumption}\rm\label{assumption:thm:Hilbert:Samuel:semipositive}
Let $\Sigma$ be the set consisting of closed points $x$ of $X$ such that $\varphi_g$ is not a constant function on $[\eta_0, x_0]$. Note that $\Sigma$ is countable by Proposition \ref{Pro: constant except countable}. 
Here we consider additional assumptions (i) -- (iv).

\begin{enumerate}[label=\rm(\roman*)]
\item $D$ is a divisor.
\item $\Sigma$ is finite.
\item $\varphi_g(\eta_0) = 0$.
\item $\mu_{\inf}(g-g(\eta_0))\geqslant 0$.
\end{enumerate}
\end{enonce}

These assumptions actually describes a special case of the setting of the above theorem, but it is an essential case because the theorem in general is a
a consequence of its assertion under these assumptions by using the continuity of $\widehat{\mathrm{vol}}_{\chi}(\ndot)$.
%
%
Before starting the proof of Theorem~\ref{thm:Hilbert:Samuel:semipositive} under the above assumptions, we need to prepare several facts.
For a moment, we proceed with arguments under Assumption~\ref{assumption:thm:Hilbert:Samuel:semipositive}.
Let $L = \mathcal O_X(D)$ and $h$ be the continuous metric of $L$ given by $\exp(-\varphi_g)$. 
For $x\in\Sigma$, let \[a_x := \varphi_g'(\eta_0;x)\quad\text{and}\quad\varphi_x' := \varphi_{g\circ\xi_x}'.\]
For $x\in\Sigma$ and $n\in\mathbb N_{\geqslant 1}$, 
we set $a_{x,n} = \lfloor -n a_x \rfloor$. One has \[a_{x,n} \leqslant -na_x < a_{x,n} + 1\quad\text{and}\quad
\lim_{n\to\infty} \frac{a_{x,n}}{n} = -a_x.\] Moreover, as \[\sum_{x \in \Sigma} a_x[k(x):k] + \deg(L) > 0\] by our assumptions, there exists $n_0\in\mathbb N_{\geqslant 1}$ such that
\[\begin{split}
&\quad\;\frac{2(\operatorname{genus}(X)-1) + \sum_{x \in \Sigma} a_{x, n}[k(x):k] +\sum_{x\in\Sigma}[k(x):x]}{n} \\
&\leqslant \frac{2(\operatorname{genus}(X)-1) + \sum_{x\in\Sigma}[k(x):x]}{n} - \sum_{x \in \Sigma} a_x[k(x):k] < \deg(L)
\end{split}\]
holds for any integer $n\geqslant n_0$, that is,
\begin{equation}\forall\,n\in\mathbb N_{\geqslant n_0},\quad 2(\operatorname{genus}(X)-1) + \sum_{x \in \Sigma} (a_{x,n}+1)[k(x):k] < n \deg(L).\end{equation}
We set 
\[
D_n = \sum_{x \in \Sigma} (a_{x, n}+1) x\quad\text{and}\quad
D_{x, n} = D_n - (a_{x, n}+1) x.
\]
Note that
\[
\begin{cases}
H^0(X, nL \otimes \mathcal O_X(-D_n)) = \{ s \in H^0(X, nL) \,:\, \text{$\operatorname{ord}_x(s) \geqslant a_{x, n}+1$ ($\forall x \in \Sigma$)} \}, \\[1ex]
H^0(X, nL \otimes \mathcal O_X(-D_{x,n} - i x)) \\
\hskip3em = \left\{ s \in H^0(X, nL) \,:\,  \text{$\operatorname{ord}_y(s) \geqslant a_{y, n}+1$ ($\forall y \in \Sigma \setminus \{x\}$) and 
$\operatorname{ord}_x(s) \geqslant i$}\right\}
\end{cases}
\]

\begin{lemm}\label{lemm:sum:subspace}For any integer $n$ such that $n\geqslant 0$, the following assertions hold.
\begin{enumerate}[label=\rm(\arabic*)]
\item 
$\sum_{x \in \Sigma} H^0(X, nL \otimes \mathcal O_X(-D_{x, n}))
= H^0(X, nL)$.

\item One has
\begin{multline*}
\hskip2em H^0(X, nL)/H^0(X, nL \otimes \mathcal O_X(-D_n)) \\
= \bigoplus_{x \in \Sigma} H^0(X, nL \otimes \mathcal O_X(-D_{x,n}))/H^0(X, nL \otimes \mathcal O_X(-D_n))
\end{multline*}
\end{enumerate}
\end{lemm}

\begin{proof}
(1) Let us consider a natural homomorphism
\[
\bigoplus_{x \in \Sigma}nL \otimes \mathcal O_X(-D_{x, n}) \to nL.
\]
Note that the above homomorphism is surjective and the kernel is isomorphic to
$(nL \otimes \mathcal O_X(-D_n))^{\oplus \operatorname{card}(\Sigma) - 1}$. Moreover, by Serre's duality,
\[
H^1(X, nL \otimes \mathcal O_X(-D_n)) \simeq H^0(X, \omega_X \otimes -n L \otimes \mathcal O_X(D_n))^{\vee}
\]
and 
\[\begin{split}
&\quad\;\deg(\omega_X \otimes -n L \otimes \mathcal O_X(D_n))\\& = 2(\operatorname{genus}(X)-1) - n \deg(L) + \sum_{x \in \Sigma} (a_{x, n}+1)[k(x):k] <  0,
\end{split}
\]
so that $H^1(X, nL \otimes \mathcal O_X(-D_n)) = 0$. Therefore one has (1).

\medskip
(2) By (1), it is sufficient to see that if \[\sum_{x \in \Sigma} s_x \in H^0(X, nL \otimes \mathcal O_X(-D_n))\] and
\[\forall\,x\in\Sigma,\quad s_x \in H^0(X, nL \otimes \mathcal O_X(-D_{x,n})),\] then \[s_x \in H^0(X, nL \otimes \mathcal O_X(-D_n))\] for all $x \in \Sigma$.
Indeed, as \[\forall\,y\in\Sigma\setminus\{x\},\quad s_y \in H^0(X, \mathcal O_X(-(a_{x,n}+1) x))\] and \[\sum_{y \in \Sigma} s_y \in H^0(X, \mathcal O_X(-(a_{x,n}+1) x)),\]
we obtain \[s_x \in H^0(X, \mathcal O_X(-(a_{x,n}+1) x)),\] so that $s_x \in H^0(X, \mathcal O_X(-D_n))$, as required.
\end{proof}

\begin{lemm}\label{lem:quotient:one:dim}
For $x \in \Sigma$ and $i \in \{ 0, \ldots, a_{x,n}\}$,
\[
\dim_k \Big( H^0(X, nL \otimes \mathcal O_X(-D_{x,n}-ix))/H^0(X, nL \otimes\mathcal O_X(-D_{x,n}-(i+1)x)) \Big)= [k(x):k].
\]
\end{lemm}

\begin{proof} Let us consider an exact sequence 
\[
0 \to nL \otimes \mathcal O_X(-D_{x,n}-(i+1)x) \to nL \otimes \mathcal O_X(-D_{x, n} -ix) \to k(x) \to 0,
\]
so that, since
\[\begin{split}
&\quad\;\deg(\omega_X \otimes -n L \otimes \mathcal O_X(D_{x, n} + (i+1)x)) \\
&= 2(\operatorname{genus}(X)-1) - n \deg(L) + \big( (i+1) - (a_{x, n}+1)\big)[k(x):k] \\
&\qquad +
\sum_{y \in \Sigma} (a_{y, n}+1)[k(y):k] \\
&\leqslant 2(\operatorname{genus}(X)-1) - n \deg(L)  + \sum_{y \in \Sigma} (a_{y, n}+1)[k(y):k] < 0,
\end{split}\]
one has the assertion as before.
\end{proof}

By Lemma~\ref{lem:quotient:one:dim}, for each $x \in \Sigma$, there are 
\[
s_{x,0}^{(\ell)}, \ldots, s_{x, a_{x,n}}^{(\ell)} \in H^0(X, nL \otimes \mathcal O_X(-D_{x, n})),\quad \ell\in\{1,\ldots,[k(x):k]\}
\]
such that the classes of $s_{x,0}^{(\ell)}, \ldots, s_{x,a_{x,n}}^{(\ell)}$ form a basis of
\[
H^0(X, nL \otimes \mathcal O_X(-D_{x, n})) / H^0(X, nL \otimes \mathcal O_X(-D_n))
\]
and
\[
s_{x, i}^{(\ell)} \in H^0(X, nL \otimes \mathcal O_X(-D_{x, n} -ix)) \setminus H^0(X, nL \otimes\mathcal O_X(-D_{x,n}-(i+1)x))
\]
whose classes form a basis of  
\[H^0(X, nL \otimes \mathcal O_X(-D_{x, n} -ix)) / H^0(X, nL \otimes\mathcal O_X(-D_{x,n}-(i+1)x))\]for $i=0, \ldots, a_{x,n}$. Moreover we choose a basis $\{ t_{1}, \ldots, t_{e_n} \}$ of $H^0(X, nL \otimes \mathcal O_X(-D_n))$. Then, by Lemma~\ref{lemm:sum:subspace},
\[
\Delta_n := \{ t_1, \ldots, t_{e_n} \} \cup \bigcup_{x \in \Sigma}
\big\{ s_{x,0}^{(\ell)}, \ldots, s_{x, a_{x,n}}^{(\ell)}\,:\,\ell\in\{1,\ldots,[k(x):k]\} \big\}
\]
forms a basis of $H^0(X, nL)$.

\begin{lemm}\label{lemma:orthogonal:basis:nL}
\begin{enumerate}[label=\rm(\arabic*)]
\item The equality \[\| s_{x,i}^{(\ell)} \|_{nh} = 
\exp(-n\varphi_x^*(i/n))\] holds for $x \in \Sigma$, $\ell\in\{1,\ldots,[k(x):k]\}$ and $i \in \{ 0, \ldots, a_{x, n}\}$.
Moreover $\|t_j \|_{nh}= 1$ for all $j\in\{1, \ldots, e_n\}$. 

\item
The basis $\Delta_n$ of $H^0(X, nL)$ is orthogonal with respect to $\|\ndot\|_{nh}$.
\end{enumerate}
\end{lemm}

\begin{proof}
First of all, note that, for $s \in H^0(X, nL) \setminus \{ 0 \}$ and $\xi \in X^{\mathrm{an}}$,
\[
- \ln |s|_{nh}(\xi)
= \begin{cases}
t(\xi) \operatorname{ord}_x(s) \geq 0 & \text{if $\xi \in [\eta_0, x_0]$ and $x \not\in \Sigma$}, \\
n\big( t(\xi)(\operatorname{ord}_{x}(s)/n) + \varphi_x(t(\xi))\big) & \text{if $\xi \in [\eta_0, x_0]$ and $x \in \Sigma$},
\end{cases}
\]
so that
\begin{equation}\label{leqn:emma:orthogonal:basis:nL:01}
\| s \|_{nh} = \max \left\{ 1, \ \max_{x \in \Sigma} \{ \exp(-n \varphi_x^*(\operatorname{ord}_x(s)/n)) \}\right\}.
\end{equation}

(1) The assertion follows from \eqref{leqn:emma:orthogonal:basis:nL:01} because $\varphi_x^{*}(\lambda) = 0$ if $\lambda \geqslant -a_x$.


\medskip
(2) Fix $s \in H^0(X, nL) \setminus \{ 0 \}$. We set 
\[
s = b_1 t_1 + \cdots + b_{e_n }t_{e_n} + \sum_{x \in \Sigma} \sum_{i=0}^{a_{x, n}}\sum_{\ell=1}^{[k(x):k]} c_{x, i}^{(\ell)} s_{x, i}^{(\ell)}.
\]


If $s \in H^0(X, nL \otimes \mathcal O_X(-D_n))$, then $c_{x, i}^{(\ell)} = 0$ for all $x, i$ and $\ell$.
Thus 
\[
1 = \max_{j\in\{1, \ldots, e_n\}} \{ |b_j|\cdot\| t_j \|_{nj} \} = \|s\|_{nh}.
\]
Next we assume that $s \not\in H^0(X, nL \otimes \mathcal O_X(-D_n))$. If we set 
\[T = \{ x \in \Sigma \,:\, \operatorname{ord}_x(s) \leqslant a_{x, n} \},\] then
$T \not= \emptyset$ and, for $x \in \Sigma$ and $\ell\in\{1,\ldots,[k(x):k]\}$,
\[
\begin{cases}
c_{x,0}^{(\ell)} = \cdots = c_{x,a_{x,n}}^{(\ell)} = 0 & \text{if $x \not\in T$}, \\
c_{x,0}^{(\ell)} = \cdots = c_{x, \operatorname{ord}_x(s) - 1}^{(\ell)} = 0,\quad
(c_{x, \operatorname{ord}_x(s)}^{(\ell)})_{\ell=1}^{[k(x):k]} \not= (0,\ldots,0) & \text{if $x \in T$}.
\end{cases}
\]
Therefore, by \eqref{leqn:emma:orthogonal:basis:nL:01},
\begin{multline*}
\max \left\{ \max_{j=1, \ldots, e_n }\{ |b_j|\cdot \|t_j \|_{nh} \} , \max_{\substack{x \in \Sigma,\\i=0, \ldots, a_{x, n}}} \{ |c_{x, i}|\cdot \| s_{x, i} \|_{nh} \} \right\} 
= \max_{x\in T,\,\ell} \{ \| s_{x, \operatorname{ord}_x(s) }^{(\ell)} \|_{nh} \} \\ = \max_{x\in \Sigma,\,\ell} \{ \| s_{x, \operatorname{ord}_x(s) }^{(\ell)} \|_{nh} \} = \max_{x\in \Sigma} \{ \exp(-n \varphi_x^*(\operatorname{ord}_x(s)/n))  \} = \| s \|_{nh},
\end{multline*}
as required.
\end{proof}

Let us begin the proof of Theorem~\ref{thm:Hilbert:Samuel:semipositive} under Assumption~\ref{assumption:thm:Hilbert:Samuel:semipositive}.
By Lemma~\ref{lemma:orthogonal:basis:nL} together with Definition \ref{Def: pairing of Green functions} and Proposition~\ref{theorem:integral:Legendre:trans},
\begin{multline*}
\lim_{n\to\infty}\frac{\widehat{\deg}\left( H^0(X, nL), \|\ndot\|_{nh} \right)}{n^2/2} = 2 \sum_{x \in \Sigma} \lim_{n\to\infty}[k(x):k] \sum_{i=0}^{a_{x,n}} \frac{1}{n} \varphi_x^*(i/n) \\
= 2\sum_{x \in \Sigma}[k(x):k] \int_{0}^{-a_x} \varphi_x^*(\lambda) \mathrm{d}\lambda = -\sum_{x \in \Sigma} [k(x):k]\int_{0}^{\infty} (\varphi'_x)^2 \mathrm{d}t = (\overline D\cdot\overline D),
\end{multline*}
as required.

\begin{proof}[Proof of Theorem~\ref{thm:Hilbert:Samuel:semipositive} without additional assumptions]
First of all, note that $\Sigma$ is a countable set (cf. Proposition \ref{Pro: constant except countable}).

\medskip
{\bf Step~1}: (the case where $D$ is Cartier divisor, $\Sigma$ is finite and $\varphi_g'(\eta_0) + \deg(D) > 0$). 
By the previous observation, \[\widehat{\mathrm{vol}}_{\chi}(D, g - g(\eta_0)) = \big((D, g-g(\eta_0))\cdot(D, g-g(\eta_0))\big).\] On the other hand, by Proposition \ref{Pro:volume chi translation}, one has
\[\widehat{\mathrm{vol}}_{\chi}(D, g) = \widehat{\mathrm{vol}}_{\chi}(D, g - g(\eta_0)) + 2 \deg(D) g(\eta_0).\]
Moreover, by the bilinearity of the arithmetic intersection pairing, one has  \[\big(\overline D\cdot\overline D\big) = \big((D, g-g(\eta_0))\cdot (D, g-g(\eta_0))\big) + 2 \deg(D) g(\eta_0).\]
Thus the assertion follows.

\medskip
{\bf Step~2}: (the case where $D$ is Cartier divisor and $\Sigma$ is finite).
For $0 < \varepsilon < 1$, we set $g_{\varepsilon} = g_D^{\mathrm{can}} + \varepsilon \varphi_g$. If $\varphi_g'(\eta_0) = 0$,
then $\Sigma = \emptyset$, so that the assertion is obvious. Thus we may assume that $\varphi_g'(\eta_0) < 0$.
As $\varphi_g'(\eta_0) + \deg(D) \geqslant 0$, we have $\varepsilon \varphi_g'(\eta_0) + \deg(D) > 0$.
Therefore, by Step~1, \[\widehat{\mathrm{vol}}_{\chi}(D, g_{\varepsilon}) = \big((D, g_{\varepsilon})\cdot(D, g_{\varepsilon})\big).\] Thus the assertion follows by Proposition~\ref{Pro: continuity}. 

\medskip
{\bf Step~3}: (the case where $D$ is Cartier divisor and $\Sigma$ is infinite). We write $\Sigma$ in the form of a sequence $\{ x_1, \ldots, x_n, \ldots, \}$. For any $n\in\mathbb Z_{\geqslant 1}$, let $g_n$ be the Green function defined as follows:
\[\forall\,\xi\in X^{\mathrm{an}},\quad
g_n(\xi) = g_D (\xi)+ \begin{cases}
\varphi_g(\xi) & \text{if $\xi\in\bigcup_{i=1}^n[\eta_0,x_{i,0}]$},\\
g(\eta_0) & \text{otherwise}.
\end{cases}
\]
Note that 
\[\lim_{n\to\infty}\sup_{\xi\in X^{\mathrm{an}}} | \varphi_{g_n}(\xi) - \varphi_{g}(\xi)| = 0.\] Indeed, as $\varphi_g$ is
continuous at $\eta_0$, for any $\varepsilon > 0$, there is an open set $U$ of $X^{\mathrm{an}}$ such that
$\eta_0 \in U$ and $| \varphi_g(\xi) - \varphi_g(\eta_0) | \leqslant \varepsilon$ for any $\xi \in U$.
Since $\eta_0 \in U$, one can find $N$ such that $[\eta_0, x_{n,0}] \subseteq U$ for all $n \geqslant N$.
Then, for $n \geqslant N$,
\[
|\varphi_{g}(\xi) - \varphi_{g_n}(\xi)| \begin{cases}
\leqslant \varepsilon  & \text{if $\xi \in [\eta_0, x_{i,0}]$ for some $i > n$},\\
= 0 & \text{othertwise},
\end{cases} 
\]
as required. Thus, by (2) in Proposition~\ref{prop:formula:avol:g:g:prime}, 
the assertion is a consequence of Step~2.

\medskip
{\bf Step~4}: (the case where $D$ is $\mathbb{Q}$-Cartier divisor). 
Choose a positive integer $a$ such that $aD$ is Cartier divisor. Then, by Step~3, 
\[\widehat{\mathrm{vol}}_{\chi}(a\overline D)=(a\overline D\cdot a\overline D)=a^2(\overline D\cdot\overline D).\] By Corollary \ref{Cor: linearility of vol chi}, 
one has $\widehat{\mathrm{vol}}_{\chi}(a\overline D)=a^2\widehat{\mathrm{vol}}_{\chi}(\overline D)$. Hence the equality \[\widehat{\mathrm{vol}}_{\chi}(\overline D)=(\overline D\cdot\overline D)\] holds.

\medskip
{\bf Step~5}: (general case).
By our assumption, there are adelic $\mathbb{Q}$-Cartier divisors $(D_1, g_1), \ldots, (D_r, g_r)$ and $a_1, \ldots, a_r \in \mathbb R_{>0}$ such that
$D_1, \ldots, D_r$ are semiample, $g_1, \ldots, g_r$ are plurisubharmonic, and $(D, g) = a_1(D_1, g_1) + \cdots + a_r(D_r, g_r)$.
We choose sequences $\{ a_{1, n} \}_{n=1}^{\infty}, \ldots, \{a_{r, n}\}_{n=1}^{\infty}$ of positive rational numbers such that
$\lim_{n\to\infty} a_{i, n} = a_i$ for $i=1, \ldots, r$.
We set $(D_n, g_n) = a_{1, n} (D_1, g_1) + \ldots + a_{r, n}(D_r, g_r)$. Then we may assume that $\deg(D_n) > 0$.
By Step~4, then $\mathrm{vol}_{\chi}(D_n, g_n) = (D_n, g_n)^2$.
On the other hand, by Proposition~\ref{Pro: continuity}, 
$\mathrm{vol}_{\chi}(D, g) = \lim_{n\to\infty} \widehat{\mathrm{vol}}_{\chi}(D_n, g_n)$.
Moreover, \[\big((D, g)\cdot(D, g)\big) = \lim_{n\to\infty} \big((D_n, g_n)\cdot (D_n, g_n)\big).\]
Thus the assertion follows.
\end{proof}

\begin{rema}
Let $\overline D_1=(D_1,g_1)$ and $\overline D_2=(D_2,g_2)$ be adelic $\mathbb R$-divisors such that $\deg(D_1)>0$ and $\deg(D_2)>0$. Let $\overline D=(D_1+D_2,g_1+g_2)$. If $g_1$ and $g_2$ are plurisubharmoic, then Theorems \ref{thm:super:additive:vol:chi:deg} and \ref{thm:Hilbert:Samuel:semipositive} lead to the following inequality
\begin{equation}\label{Equ: Hodge index inequality}\frac{(\overline D\cdot\overline D)}{\deg(D)}\geqslant\frac{(\overline D_1\cdot\overline D_1)}{\deg(D_1)}+\frac{(\overline D_2\cdot\overline D_2)}{\deg(D_2)}.\end{equation}
This inequality actually holds without plurisubharmonic condition (namely it suffices that $g_1$ and $g_2$ are pairable). In fact, by \eqref{Equ: coupling Di} one has
\[\frac{(\overline D_i\cdot\overline D_i)}{\deg(D_i)}=2g_i(\eta_0)-\sum_{x\in X^{(1)}}\frac{[k(x):k]}{\deg(D_i)}\int_0^{+\infty}\varphi_{g_i\circ\xi_x}'(t)^2\,\mathrm{d}t\]
for $i\in\{1,2\}$, and
\[\begin{split}&\frac{(\overline D\cdot\overline D)}{\deg(D)}=2(g_1(\eta_0)+g_2(\eta_0))\\
&\qquad-\sum_{x\in X^{(1)}}\frac{[k(x):k]}{\deg(D_1)+\deg(D_2)}\int_0^{+\infty}(\varphi_{g_1\circ\xi_x}'(t)+\varphi_{g_2\circ\xi_x}'(t))^2\,\mathrm{d}t,
\end{split}\]
which leads to
\[\begin{split}&\quad\;(\deg(D_1)+\deg(D_2))\bigg(\frac{(\overline D\cdot\overline D)}{\deg(D)}-\frac{(\overline D_1\cdot\overline D_1)}{\deg(D_1)}-\frac{(\overline D_2\cdot\overline D_2)}{\deg(D_2)}\bigg)\\
&=\sum_{x\in X^{(1)}}[k(x):k]\bigg(\frac{\deg(D_2)}{\deg(D_1)}\int_0^{+\infty}\varphi_{g_1\circ\xi_x}'(t)^2\,\mathrm{d}t+\frac{\deg(D_1)}{\deg(D_2)}\int_0^{+\infty}\varphi_{g_2\circ\xi_x}'(t)^2\,\mathrm{d}t\\
&\qquad\quad-2\int_{0}^{+\infty}\varphi_{g_1\circ\xi_x}'(t)\varphi_{g_2\circ\xi_x}'(t)\,\mathrm{d}t\bigg)\geqslant 0,
\end{split}
\]by using Cauchy-Schwarz inequality and the inequality of arithmetic and geometric means.

The inequality \eqref{Equ: Hodge index inequality} leads to 
\[2(\overline D_1\cdot\overline D_2)\geqslant \frac{\deg(D_2)}{\deg(D_1)}(\overline D_1\cdot\overline D_1)+\frac{\deg(D_1)}{\deg(D_2)}(\overline D_2\cdot\overline D_2).\]
In the case where $(\overline D_1\cdot\overline D_2)$ and $(\overline D_2\cdot\overline D_2)$ sont non-negative, by the inequality of arithmetic and geometric means, we obtain that 
\[(\overline D_1\cdot\overline D_2)\geqslant \sqrt{(\overline D_1\cdot\overline D_1)(\overline D_2\cdot\overline D_2)},\]
where the equality holds if and only if $\overline D_1$ and $\overline D_2$ are proportional up to $\mathbb R$-linear equivalence.
This could be considered as an analogue of the arithmetic Hodge index inequality of Faltings \cite[Theorem 4]{MR740897} and Hriljac \cite[Theorem 3.4]{MR778087}, see also \cite[Theorem 7.1]{MR1189866} and \cite[\S5.5]{MR1681810}.
\end{rema}

\bibliography{intersection}
\bibliographystyle{plain}
\end{document}